\documentclass[11pt]{article}

\usepackage{
    graphicx,
    amsmath,
    latexsym,
    amssymb,
    tikz-cd,
    amsthm,
    color,
    hyperref,
    enumerate,
    pb-diagram,
    bbm
}

\usepackage[margin=1.5in]{geometry}

\usepackage{tikz}
\usetikzlibrary{shapes, positioning, decorations.pathreplacing, arrows}

\usepackage{rotating}
\usepackage{float}

\newtheorem{theorem}{Theorem}
\newtheorem{corollary}[theorem]{Corollary}
\newtheorem{cor}[theorem]{Corollary}
\newtheorem{lemma}[theorem]{Lemma}
\newtheorem{definition}[theorem]{Definition}
\newtheorem{prop}[theorem]{Proposition}
\newtheorem{proposition}[theorem]{Proposition}
\newtheorem{ex}[theorem]{Example}

\newtheorem{remark}[theorem]{Remark}

\newtheorem{claim}[theorem]{Claim}

\newtheorem*{theorem*}{Theorem}

\hypersetup{
    colorlinks,
    citecolor=blue,
    filecolor=blue,
    linkcolor=blue,
    urlcolor=blue
}
\hypersetup{linktocpage}

\newcommand{\nc}{\newcommand}
\nc{\dmo}{\DeclareMathOperator}
\nc{\rnc}{\renewcommand}

\rnc{\phi}{\varphi}
\rnc{\qed}{{\nopagebreak \hfill $\dashv$ \par\bigskip}}

\dmo{\pri}{\textsc{Prime}}
\dmo{\Tr}{\textsc{Tr}}
\dmo{\ACA}{\mathbf{ACA}}
\dmo{\RCA}{\mathbf{RCA}}
\dmo{\WWKL}{\mathbf{WWKL}}
\dmo{\WKL}{\mathbf{WKL}}
\dmo{\PA}{\mathbf{PA}}
\dmo{\AD}{\mathbf{AD}}
\dmo{\PD}{\mathbf{PD}}
\dmo{\CH}{\mathbf{CH}}
\dmo{\AC}{\mathbf{AC}}
\dmo{\ZFC}{\mathbf{ZFC}}
\dmo{\ZF}{\mathbf{ZF}}
\dmo{\DC}{\mathbf{DC}}
\dmo{\CC}{\mathbf{CC}}
\dmo{\RH}{\mathbf{RH}}
\dmo{\GC}{\mathbf{GC}}
\dmo{\crit}{\textsc{Crt}}
\dmo{\Ord}{\textsc{Ord}}
\dmo{\cl}{cl}
\dmo{\tc}{tc}
\dmo{\pre}{\textsc{Pre}}
\dmo{\RUlt}{\textsc{Ult}_0}
\dmo{\std}{\textsc{Std}}
\dmo{\id}{\textsc{id}}
\dmo{\ev}{\textsc{ev}}
\dmo{\cof}{\textsc{Cof}}
\dmo{\Seq}{\textsc{Seq}}
\dmo{\pII}{II}
\dmo{\pI}{I}
\dmo{\ot}{\textsc{Ot}}
\dmo{\BS}{\mathbf{\Sigma}}
\dmo{\BP}{\mathbf{\Pi}}
\dmo{\KB}{KB}
\dmo{\Zs}{\EUR}
\dmo{\sh}{\textsc{Sh}}
\dmo{\rot}{\overline{\mathcal{R}}}
\dmo{\MPT}{\mathbf{MPT}}
\dmo{\DT}{\mathcal{DT}}
\dmo{\RT}{\mathcal{RT}}
\dmo{\difflam}{Diff^\infty(\B T, \lambda)}

\nc{\B}[1]{\mathbb{#1}}
\nc{\C}[1]{\mathcal{#1}}
\nc{\Sc}[1]{\mathscr{#1}}
\nc{\la}{\langle}
\nc{\ra}{\rangle}
\nc{\I}[1]{\operatorname{\mathbf{I}}_{#1}}
\nc{\rest}{\mathbin{\upharpoonright}}
\nc{\func}[3][]{\prescript{#2}{#1}{#3}}
\nc{\LV}[1]{L_{#1}(V_{\lambda+1})}
\nc{\wt}{\widetilde}
\nc{\imp}{\rightarrow}
\nc{\var}[1]{\textbf{#1}}
\nc{\F}[1]{\mathfrak{#1}}
\nc{\diff}{\mathbin{\triangle}}
\nc{\diag}{\mathbin{\mathlarger{\mathlarger{\triangle}}}}
\nc{\lb}{\llbracket}
\nc{\rb}{\rrbracket}
\nc{\suffix}[2]{\(#1^{\text{#2}}\)}
\nc{\suf}[2]{\(#1^{\text{#2}}\)}
\nc{\HRule}{\rule{\linewidth}{0.5mm}}
\nc{\wh}{\widehat}
\nc{\rev}[1]{\textsc{Rev}({#1})}
\nc{\mci}{\ensuremath{\mathcal I}}
\nc{\mcl}{\ensuremath{\mathcal L}}
\nc{\mcq}{\ensuremath{\mathcal Q}}
\nc{\mca}{\ensuremath{\mathcal A}}
\nc{\mcr}{\ensuremath{\mathcal R}}
\nc{\mcu}{\mathcal U}
\nc{\mcv}{\mathcal V}
\nc{\fc}{\ensuremath{\mathcal{FC}}}
\nc{\cg}{\ensuremath{\mathcal{CG}}}
\nc{\mcp}{\mathcal P}
\nc{\mch}{\mathcal H}
\nc{\mcw}{\mathcal W}
\nc{\mce}{{\mathcal E}}
\nc{\mcx}{{\mathcal X}}
\nc{\mcc}{{\mathcal C}}
\nc{\mck}{{\mathcal K}}
\nc{\mcb}{\mathcal B}
\nc{\mcs}{{\mathcal S}}
\nc{\mco}{\mathcal O}
\nc{\mcf}{\mathcal F}
\nc{\mbp}{{\ensuremath{\mathbb P}}}
\nc{\mbq}{{\ensuremath{\mathbb Q}}}
\nc{\mbr}{{\ensuremath{\mathbb R}}}
\nc{\uc}{\ensuremath{{\mathbb T}}}
\nc{\rat}{\ensuremath{{\mathbb Q}}}
\nc{\poN}{\mathbb N}
\nc{\poZ}{\mathbb Z}
\nc{\zt}{\mathbb Z_2}
\nc{\poz}{\mathbb Z}
\nc{\bk}{{\mathbb K}}
\nc{\bl}{\mathbb L}
\nc{\poP}{\ensuremath{\mathbb P}}
\nc{\poQ}{\ensuremath{\mathbb Q}}
\nc{\poR}{\ensuremath{\mathbb R}}
\nc{\poB}{\ensuremath{\mathbb B}}
\nc{\sz}{\Sigma^\poZ}
\nc{\rid}[1]{{\overline{\mcr}}_{#1}}	
\nc{\nn}{\mathbb N}
\nc{\bt}{\mathbb T}
\nc{\mcT}{{\ensuremath{\mathcal T}}}
\nc{\mct}{{\ensuremath{\mathcal T}}}
\nc{\mcD}{{\ensuremath{\mathcal D}}}
\nc{\mcd}{{\ensuremath{\mathcal D}}}
\nc{\pf}{{\par\noindent{$\vdash$\ \ \ }}}
\nc{\acts}{\curvearrowright}
\nc{\klb}{\mathit{k}_{\textsc{lb}}}
\nc{\kmax}{\mathit{k}_{\textsc{Max}}}
\newcommand{\kron}{\mathcal K r}
\newcommand{\mcj}{\mathcal J}
\newcommand{\misal}{\not\Downarrow}
\newcommand{\zoo}{{[0,1)}}
\newcommand{\dbar}{\ensuremath{\bar{d}}}

\newcommand{\bfni}[1]{\noindent {{\bf{#1}}}}
\newcommand{\footremember}[2]{%
    \footnote{#2}
    \newcounter{#1}
    \setcounter{#1}{\value{footnote}}%
}

\title{G\"odel Diffeomorphisms}
\author{
    Matthew Foreman\footremember{uci}{University of California,
    Irvine.}\footnote{ Foreman's research was supported by the National Science
    foundation grant DMS-1700143}
   }

\begin{document}

\maketitle

\begin{abstract}
    A basic problem in smooth dynamics is determining if a system can be
    distinguished from its inverse, i.e., whether a smooth diffeomorphism
    $T$ is isomorphic to \(T^{-1}\). We show that this problem is
    sufficiently general that asking it for particular choices of $T$ is
    equivalent to the validity of well-known number theoretic conjectures
    including the Riemann Hypothesis and Goldbach's conjecture. Further one
    can produce computable diffeomorphisms $T$ such that the question of
    whether $T$ is isomorphic to $T^{-1}$ is independent of ZFC.
\end{abstract}

\tableofcontents

\section{Introduction}
\label{sec:Intro}

When is forward time  isomorphic to backward time for a given dynamical
system? When the acting group is \(\B Z\), this asks when a transformation $T$
is  isomorphic to its inverse. It was not until 1951,
that {Anzai \cite{ANZAI}} refuted a conjecture of Halmos and von Neumann by
exhibiting the first example of a transformation where $T$ is not measure theoretically isomorphic to
its inverse.\footnote{
	See for example Math Review MR0047742 where Halmos states ``By
	constructing an example of the type described in the title the author
	solves (negatively) a problem proposed by the reviewer and von Neumann
	[Ann. of Math. (2) 63, 332--350 (1962); MR0006617]."
} 
In fact the general problem is so complex that it cannot be be resolved
using an arbitrary countable amount of information: {in \cite{part3}, it
was shown that the collection of ergodic Lebesgue measure preserving diffeomorphisms of the 2-torus isomorphic to their inverse is complete analytic and hence not Borel.}

In this paper we show that for a broad class of problems there is a one-to-one
computable method of associating a Lebesgue measure preserving diffeomorphism \(T_P\) of
the two-torus to each problem \(P\) in this class so that:

\begin{itemize}
    \item \(P\) is true
\end{itemize}

if and only if

\begin{itemize}
    \item \(T_P\) is measure theoretically isomorphic to \(T_P^{-1}\).
\end{itemize}

The class of problems is large enough to include the \emph{Riemann Hypothesis},
\emph{Goldbach's Conjecture} and statements such as ``\emph{Zermelo-Frankel Set
Theory (\(\ZFC\)) is consistent.}" In consequence, each of these problems is
{equivalent} to the question of whether \(T\cong T^{-1}\) for the diffeomorphism $T$ of 2-torus canonically associated to that problem.

Restating this,  there is an ergodic diffeomorphism of the two-torus \(T_{\RH}\) such that
\(T_{\RH} \cong {T_{\RH}}^{-1}\) if and only if the Riemann Hypothesis holds, and
a different, non-isomorphic ergodic diffeomorphism \(T_{\GC}\) such that \(T_{\GC} \cong
{T_{\GC}}^{-1}\) if and only if Goldbach's conjecture holds, and so forth.
\smallskip

G\"odel's Second Incompleteness Theorem states that for any recursively
axiomatizable theory $\Sigma$ that is sufficiently strong to prove basic
arithmetic facts, if $\Sigma$ proves the statement ``\emph{$\Sigma$ is
consistent}", then $\Sigma$ is in fact \emph{inconsistent}. The statement
``$\Sigma$ is consistent" can be formalized in the manner of the problems we consider.
 Consider
 the most standard axiomatization for mathematics:  Zermelo-Frankel Set Theory with the Axiom of Choice and the  formalization of its consistency, the 
 statement Con(ZFC).

If \(T_{\ZFC}\) is the diffeomorphism associated with Con(ZFC) then (assuming the consistency of
conventional mathematics) the question of whether \(T_{\ZFC}\cong
{T_{\ZFC}}^{-1}\) is independent of Zermelo-Frankel Set Theory---that is, it
cannot be settled with the usual assumptions of mathematics.

One can compare this with more standard independence results, the most prominent
being the Continuum Hypothesis. Those independence results inherently involve
comparisons between and properties of uncountable objects.  The results
in this paper are about the relationships between finite computable objects.

We now give precise statements of the main theorem and its corollaries. The
machinery for proving these results combines ergodic theory and descriptive set
theory with logical and meta-mathematical techniques originally developed by
G\"odel. While the statements use only standard terminology, it is combined from
several fields. We have included 
{several appendices} in an attempt to
convey this background to non-experts. 

There are several standard references for connections between non-computable sets and analysis and PDE's.  We note one in particular with results of Marian Pour-El and Ian Richards that give an example of a wave equation with computable initial data but no computable solution \cite{Pourel}.

\subsection{The Main Theorem}

\noindent As an informal guide to reading the theorem, we say a couple of words.
More formal definitions appear in later sections. 
 
\begin{itemize}
    \item A function $F$ being computable means that there is a computer program
        that on {input $N$ outputs $F(N)$.}
    \item The diffeomorphisms in the paper are taken to be $C^\infty$ and Lebesgue measure preserving. A
        diffeomorphism $T$ is computable if there is a computer program that
        when serially fed the decimal expansions of a pair $(x,y)\in\bt^2$
        outputs the decimal expansions of $T(x,y)$ and for each $n$ there is a
        computable function computing the decimal expansion of the modulus of
        continuity of the \(n\)-th differential.\footnote{Recent work of Banerjee and Kunde in \cite{BK} allow Theorem \ref{thm:main} to be extended to real analytic functions by improving the realization results in \cite{part1}.} Since computable functions have codes, computable diffeomorphisms also can be coded by natural numbers.
        \item By isomorphism, it is meant \emph{measure isomorphism}. Measure preserving transformations $S:X\to X$ and $T:Y\to Y$ are measure theoretically isomorphic if there is a measure isomorphism $\phi:X\to Y$ such that 
        \[S\circ\phi=T\circ S\] 
        up to a sets of measure zero.
    \item In Appendix~\ref{app:ET} we discuss questions such as \emph{Why
        $\poZ$? Why $\bt^2$? Why $C^\infty$?}.
    
        We use the notation \(\difflam\) for the collection of \(C^\infty\)
        measure-preserving diffeomorphisms of \(\B T^2\).
    \item $\Pi^0_1$ statements are those number-theoretic statements that start
        with a block of universal quantifiers and are followed by Boolean
        combinations of equalities and inequalities of polynomials with
        natural number coefficients.
     \item We fix G\"odel numberings: computable ways of enumerating
         $\Pi^0_1$ statements $\la\phi_n:n\in\nn\ra$ and computer programs $\la
         C_m:m\in\nn\ra$. The \emph{code} of $\phi_n$ is $n$, the \emph{code} of
         $C_m$ is $m$. 
     \item Older literature uses the word \emph{recursive} and more recent
         literature uses the word \emph{computable} as a synonym. We use the
         latter in this paper. Indeed, since none of the phenomenon discussed here involve recursive behavior that is not primitive recursive we use \emph{effective}, and \emph{computable} as synonyms for \emph{primitive recursive.}
\end{itemize}

\noindent Here is the statement of the main theorem.\\

\begin{theorem}\emph{(Main Theorem)}
\label{thm:main}
    There is a computable function 
        \begin{equation*}
        \label{thm:main:eqn}
            F:\{\text{Codes for \(\Pi^0_1\)-sentences}\} \to \{\text{Codes for
            computable diffeomorphisms of $\bt^2$}\}
        \end{equation*}
    such that:

    \begin{enumerate}
        \item {\(N\)} is the code for a true statement if and only if {\(F(N)\)}
            is the code for \(T\), where \(T\) is measure theoretically
            isomorphic to \(T^{-1}\);
        \item For \(M\ne N\), \(F(M)\) is not isomorphic to \(F(N)\).
    \end{enumerate}
The diffeomorphisms in the range of \(F\) are Lebesgue measure preserving and ergodic.
\end{theorem}

We now explicitly draw corollaries.

\begin{cor}
    There is an ergodic diffeomorphism of the two-torus \(T_{\RH}\) such that
    \(T_{\RH}\cong T_{\RH}^{-1}\) if and only if the Riemann Hypothesis holds.
\end{cor}

Similarly:

\begin{cor}
    There is an ergodic diffeomorphism of the two-torus \(T_{\GC}\) such that
    \(T_{\GC}\cong T_{\GC}^{-1}\) if and only if Goldbach's Conjecture holds.
\end{cor}

There are at least two reasons that this theorem is not trivial.  The first is that the function $F$ is computable, hence the association of the diffeomorphism to the $\Pi^0_1$ statement is canonical.  Secondly the function is one-to-one;
 \(T_{\RH}\) encodes the Riemann hypothesis and \(T_{\GC}\) encodes
Goldbach's conjecture and \(T_{\RH}\not\cong T_{\GC}\).

\begin{cor}
\label{sec:intro:cor:independence}
    Assume that \(\ZFC\) is consistent. Then there is a computable ergodic diffeomorphism
    \(T\) of the torus such that \(T\) is measure theoretically isomorphic to
    \(T^{-1}\), but this is unprovable in Zermelo-Frankel set theory together
    with the Axiom of Choice. 
\end{cor}

We note again that there is nothing particularly distinctive about
Zermelo-Frankel set theory with the Axiom of Choice. We {choose} it for the
corollary because it forms the usual axiom system for mathematics. Thus
Corollary~\ref{sec:intro:cor:independence} states an independence result in a
classical form. Similar results can be drawn for theories of the form ``\(\ZFC\)
+ there is a large cardinal'' or simply ZF without the Axiom of Choice.

Finally, these results can be modified quite easily to produce diffeomorphisms
of (e.g.) the unit disc with the analogous properties. Moreover techniques from
the thesis of Banerjee (\cite{Banerjee}) and Banerjee-Kunde (\cite{BK}) can
be used to improve the reduction $F$ so that the range consists of real analytic
maps of the 2-torus. 

We finish this section by thanking Tim Carlson for asking whether Theorem 
\ref{thm:main} can be extended to lightface $\Sigma^1_1$ statements, which it can 
in a straightforward way.  This increases the collection of statements encoded into 
diffeomorphisms to include virtually all standard mathematical  statements.

\paragraph{Primitive recursion} Informally, primitive recursive functions are those that can be computed by a program that uses only \emph{for} statements and not \emph{while} statements. This means that the computational time can be bounded constructively using iterated exponential maps. In the statements of the results we discuss ``computable functions" but in fact all of the functions constructed are primitive recursive.  In particular the functions and computable diffeomorphisms  asserted to exist in Theorem \ref{thm:main} are primitive recursive.

\subsection{Hilbert's 10th problem}

Hilbert's 10th problem asks for a general algorithm for deciding whether 
Diophantine equations have  integer solutions. The existence of
such an algorithm was shown to be impossible by a succession of results of
Davis, Putnam and Robinson culminating a complete solution by Matijasevi\v{c} in
1970 (\cite{maty, DMR}).

Their solution can be recast as a statement very similar to
Theorem~\ref{thm:main}:
    \begin{quotation}
    \noindent There is a computable function 
        \begin{equation*}
            F:\{\text{Codes for \(\Pi^0_1\)-sentences}\} \to
            \{\text{Diophantine Polynomials}\}
        \end{equation*}
    such that {\(N\)} is the code for a true statement if and only if $F(N)$
    has no integer solutions.
    \end{quotation}
Thus their theorem reduces general questions about the truth of $\Pi^0_1$
statements to questions about zeros of polynomials. Theorem~\ref{thm:main} states that 
there is
an effective reduction of the true $\Pi^0_1$ statements to $C^\infty$
transformations isomorphic to their inverse.

\subsection{\(\Pi^0_1\)-sets and G\"odel numberings}

While the interesting corollaries of Theorem~\ref{thm:main} are about the
Riemann Hypothesis, other number theoretic statements, and
 independence results for dynamical systems, it is actually a theorem about
subsets of \(\B N\). In order to prove it, one has to provide a way of
translating between the interesting mathematical objects as they are usually constructed
and the natural numbers that encode them. This is done by means of \emph{G\"odel
numberings}, natural numbers which \emph{code} the structure of familiar
mathematical objects.

The arithmetization of syntax via \emph{G\"odel Numbers} is one of the main
insights in the proofs of the Incompleteness Theorems. It is used to state
``\(\Sigma\) is consistent" (where \(\Sigma\) is an enumerable set of axioms) as
a \(\Pi^0_1\) statement. G\"odel numberings originally appear in \cite{GODEL},
but are covered in any standard logic text such as \cite{ENDERTON}.

The idea behind G\"odel numberings is very simple: let $\la p_n:n\in\nn\ra$ be
an enumeration of the prime numbers. Associate a positive integer to each
symbol: ``$x$" might be 1, ``$0$" might be 2, ``$\forall$" might be 3 and so on.
Then a sequence of symbols of length $k$ can be coded as {$c=2^{n_1} \cdot
3^{n_2} \cdot 5^{n_3} \cdots p_k^{n_k}$.}

\begin{ex}
    {Suppose we use the following coding scheme:
        \begin{center}
            \begin{tabular}{|l||c|c|c|c|c|c|c|}
                \hline
		\emph{Symbol}&x& 0& $\forall$ & $*$ &$=$ & $($ & $)$\\
		\hline
                \emph{Integer}&1&2& 3& 4& 5& 6& 7 \\
                \hline
            \end{tabular}
	\end{center}}
    Then the G\"odel number associated with the sentence:
        \[\forall x(x*0= 0)\]
    is $c=2^3*3^1*5^6*7^1*11^4*13^2*17^5*19^2*23^7$.
\end{ex}

\noindent Clearly the sentence can be uniquely recovered from its code. With
more work, one can also use natural numbers to effectively code computer programs and their computations, sequences of formulas that constitute a proof and many other objects. The methods  use the Chinese
Remainder Theorem.

We now turn to $\Pi^0_1$ sentences.

\begin{definition}
    A sentence \(\phi\) in the language \(\mathcal L_{\PA}=\{+, *, 0, 1, <\}\)
    is \(\Pi^0_1\) if it can be written in the form \((\forall x_0)(\forall
    x_1) \dots (\forall x_n)\psi\), where \(\psi\) is a Boolean combination of
    equalities and inequalities of polynomials in the variables \(x_0, \dots
    x_n\) and the constants $0,1$. (We do not allow unquantified---i.e.,
    \emph{free}---variables to appear in $\phi$.)
\end{definition}

\noindent It is not difficult to show that
    \begin{equation*}
        \{n:\text{\(n\) is the G\"odel number of a \(\Pi^0_1\) sentence
        in a finite language}\}
    \end{equation*}
is a computable set.

It is however, non-trivial to show that some statements such as the Riemann
Hypothesis and the consistency of \(\ZFC\) are provably equivalent to
\(\Pi^0_1\)-statements. The Riemann Hypothesis was shown to be $\Pi^0_1$ by
Davis, Matijasevi\v{c} and Robinson (\cite{DMR}) and a particularly elegant
version of such a statement is due to Lagarias (\cite{lag}). Appendix \ref{GC} exhibits $\Pi^0_1$-statements that are equivalent to the 
 Riemann Hypothesis (using \cite{lag}) and 
Goldbach's Conjecture.

\paragraph{Truth:} We say a sentence \(\phi\) in the language \(\C L_{PA}\) is
\emph{true} if it holds in the structure \((\mathbb N, +, *, 0, 1, <)\).

\subsection{Effectively computable diffeomorphisms}\label{effcompdif}

Since \(\B T^2\) is compact, a \(C^\infty\)-diffeomorphism \(T\) is uniformly
continuous, as are its differentials. Thus, it makes sense to view their moduli
of continuity as functions \(d: \B N \to \B N\) which say, informally, that if
one wishes to specify the map \((x,y)\mapsto T(x, y)\) to within \(2^{-n}\),
then the original point \((x, y)\) must be specified to within a tolerance of
\(2^{-d(n)}\). With better and better information about \((x,y)\), one can
produce better and better information about \(T(x, y)\). This intuitive notion
is formalized by the definitions given below, and in more detail in
Appendix~\ref{app:LB:CRF1}.\footnote{
    Since diffeomorphisms are Lipshitz, we could
    have worked with computable Lipshitz constants rather then computable moduli
    of continuity. The methods give the same collections of computable diffeomorphisms.} We note in passing that the moduli of continuity and approximations are not uniquely defined.

\begin{definition}[Effective Uniform Continuity]\label{effcont}
    We say that a map \(T:\B T^2 \to \B T^2\) is \emph{effectively uniformly
    continuous} if and only if the following two computable functions exist:
        \begin{itemize}
            \item \textbf{A computable  Modulus of Continuity:} A computable function \(d:
                \B N\to \B N\) which, given a target accuracy $\epsilon$ finds the
                $\delta$ within which the source must be known to approximate
                the function within $\epsilon$. 
                
                More concretely, suppose $T:[0,1)\times [0,1)\to [0,1)\times [0,1)$. View elements in $[0,1)$ as their binary expansions.
                Then the first $d(n)$ digits of each of $(x,y)$ determine the first $n$ digits of the two entries of $T(x,y)$.
        
            \item\textbf{A Computable Approximation:} A computable function \(f: (\{0,
                1\} \times \{0, 1\})^{<\B N} \to (\{0, 1\} \times \{0, 1\})^{<\B
                N}\), which, given the first \(d(n)\) digits of
                the binary expansion of \((x, y)\)---or, equivalently, the
                dyadic rational numbers \((k_x \cdot 2^{-d(n)}, k_y \cdot
                2^{-d(n)})\) for \(0\leq k_x, k_y \leq 2^{d(n)}\) closest to \((x,
                y)\)--- outputs the first \(n\) digits of the binary
                expansion of the coordinates of \(T(x, y)\).
        \end{itemize}
\end{definition}

The diffeomorphisms $T$ we build are $C^\infty$ and map from $\bt^2$ to $\bt^2$. Because we are working on $\bt^2$ we can view 
 $T$ as a map from $\mathbb R^2$ to $\mathbb R^2$. The $k^{th}$ differential is determined by the collection of $k^{th}$ partial derivatives 
 $\{{\partial^k\over \partial^ix\partial^{k-i}y}:0\le i\le k\}$
  of $T$ 
  with respect to the standard coordinate system for $\mathbb R^2$. For $k<\infty$, $T$ is effectively  $C^k$ provided that
for each $n<k$ there are computable $d(n, -)$ and $f(n,-)$ that give the moduli of continuity and approximations to the  partial $n^{th}$ derivatives. Being $C^\infty$ requires that the $d(n,-)$ and $f(n,-)$ exist and are uniformly computable; that is that there is a single algorithm that on every input $n\in \nn$ computes $d(n,-)$ and $f(n,-)$.

 For clarity, in  these definitions we discussed functions with domain and range $\bt^2$.  There is no difficulty generalizing effective uniform continuity  to effectively presented metric spaces. The notion of a computable $C^k$ diffeomorphism  also easily generalized to smooth manifolds $M$ and their diffeomorphisms, using atlases. 
  
We note that computable diffeomorphisms are uniquely determined by the
procedures for computing \(d\) and \(f\) and hence they too may be coded using
G\"odel numbers.  The elements of the range of the function $F$ in Theorem \ref{thm:main} code diffeomorphisms in this manner. 

\paragraph{Inverses of recursive diffeomorphisms} It is not true that the inverse of a primitive recursive function 
$f:\nn\to \nn$ is primitive recursive.  However for primitive recursive diffeomorphisms of compact manifolds it is.  Suppose that $M$ is a smooth compact manifold and $T$ is a $C^\infty$-diffeomorphism.  Then $T$ is a diffeomorphism and hence has uniformly Lipschitz differentials of all orders.  Since $T$ is invertible and $M$ is compact, $T^{-1}$ also has uniformly Lipschitz differentials of all orders.  Moreover the Lipschitz constants for $T^{-1}$  are ``one over" the Lipschitz constants for $T$.  It follows in a straightforward way that the inverse of a primitive recursive diffeomorphism on $M$ is a primitive recursive diffeomorphism. 

\subsection{Reductions}

The key idea for proving Theorem~\ref{thm:main} is that of a \emph{reduction}.

\begin{definition}
    Suppose that $A\subseteq X$ and $B\subseteq Y$ and $f:X\to Y$. Then $f$
    reduces $A$ to $B$ if
        \[x\in A \mbox{ iff } f(x)\in B.\]
\end{definition}

The idea behind a reduction is that to determine whether a point $x$ belongs to
$A$ one looks at $f(x)$ and asks whether it belongs to $B$: $f$ reduces the
question ``$x\in A$" to ``$f(x)\in B$. 

For this to be interesting the function $f$ must be relatively simple. In many cases the spaces $X$ and $Y$ are Polish spaces and 
$f$ is taken to be a Borel map. In this paper $X=Y=\nn$ and $F$ is primitive recursive.

 In
\cite{part3} the function $f$ has domain the space of trees (equivalently,
acyclic countable graphs) and has range the space of measure preserving diffeomorphisms of the
two-torus. It reduces the collection of ill-founded trees (those with an
infinite branch or, respectively, acyclic graphs with a non-trivial end) to
diffeomorphisms isomorphic to their inverse. 

The function $f$ is a Borel map. The point there is that if $\{T:T\cong
T^{-1}\}$ were Borel then its inverse by the Borel function $f$ would also have
to be Borel. But the set of ill-founded trees is known not to be Borel. Hence
the isomorphism relation of diffeomorphisms is not Borel.

In the current context the function $F$ in Theorem~\ref{thm:main} maps from a
computable subset of $\mathbb N$ (the collection of G\"odel numbers of $\Pi^0_1$
statements) to $\mathbb N$. It takes values in the collection of codes for
diffeomorphisms of the two-torus. 

Theorem \ref{thm:main} can be restated as saying that  $F$ is a primitive recursive reduction of the collection $A$ of G\"odel numbers for
true $\Pi^0_1$ statements to the collection $B$ of codes for computable measure
preserving diffeomorphisms of the torus that are isomorphic to their inverses.
For {$N\ne M$} the transformation {$F(N)$} is not isomorphic to
{$F(M)$}.

Thus Theorem~\ref{thm:main} can be restated as saying that the collection of
true $\Pi^0_1$ statements is computably reducible to the collection of measure
preserving diffeomorphisms that are isomorphic to their inverses. In the jargon:
the collection of diffeomorphisms isomorphic to their inverses is
``$\Pi^0_1$-hard."

\subsection{Structure of the paper}

The proof of the main theorem in this paper depends on background in two
subjects, requiring the quotation of key results {that}  would be
prohibitive to prove. The actual construction itself---that is, the reduction
\(F\) of the main theorem---is described in its entirety, along with the
intuition behind these results.

The paper heavily uses results proved in  \cite{FRW}, \cite{part3}, \cite{part2} and  \cite{part1}.
When used, the results are quoted, and informal intuition is
given for the proofs. When specific numbered lemmas, theorems and equations from \cite{part3} are referred to, the numbers correspond to the arXiv version cited in the bibliography.

\paragraph{Structure of the paper}
The logical background required for the proof of Theorem~\ref{thm:main} 
is minimal and the exposition is aimed at an audience with a basic
working knowledge of ergodic theory, in particular the Anosov-Katok method. 
  
Section~\ref{sec:Odom} {defines the odometer-based transformations, a large
class of measure preserving symbolic systems.} These are built by iteratively
concatenating words without spacers. We then construct the reduction $F_\mco$
from the true $\Pi^0_1$ statements to the ergodic odometer-based transformations
isomorphic to their inverse.

Section~\ref{sec:CT} moves from symbolic dynamics to smooth dynamics. This
proceeds in two steps. The first step is to define a class of symbolic systems,
the \emph{circular systems} that are realizable as measure preserving
diffeomorphisms of the two-torus. The second step uses the \emph{Global
Structure Theorem} of \cite{part2}, which shows that the category whose objects
are odometer-based systems and whose morphisms are synchronous and
anti-synchronous joinings is canonically isomorphic with the category whose
objects are circular systems and whose objects are synchronous and
anti-synchronous joinings. Thus the odometer-based systems in the range of
$F_\mco$ can be canonically associated with symbolic shifts that are isomorphic
to diffeomorphisms.

Section \ref{K facts} shows that different elements of the range of $\mcf\circ F_\mco$ are not isomorphic, by showing that their Kronecker factors are different. Sections \ref{sec:CT:ss:tor} discusses diffeomorphisms of the torus and how to realize circular systems using method of \emph{Approximation by Conjugacy} due to Anosov and Katok. Section \ref{sec:CT:ss:tor} builds a primitive recursive map $R$ from circular construction sequences to measure preserving diffeomorphisms of $\bt^2$ such that $\bk^c\cong R(\bk^c)$.

In section~\ref{sec:victory} we argue that the functor $\mcf$ defined in the Global
Structure Theorem is itself a reduction when composed with $F_\mco$.
Hence composing $R$, $\mcf$ and  $F_\mco$ gives a
reduction $F$ from the collection of true $\Pi^0_1$ statements to the collection of
ergodic diffeomorphisms of the torus that are isomorphic to their inverse. This
completes the proof of Theorem~\ref{thm:main}.

The overall content of the paper is summarized by
Figure~\ref{the square}. The reduction to odometer-based systems is $F_{\mathcal
O}$,  $\mathcal F$ is the
functorial isomorphism, the realization as smooth transformations is $R$ and the composition $F$ is the reduction in Theorem \ref{thm:main}. 

\begin{figure}[h!]
    \centering
    \includegraphics[height=.25\textheight]{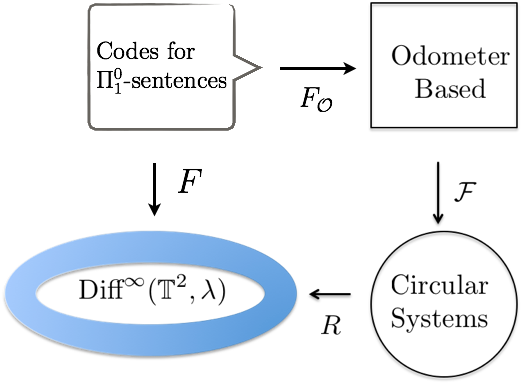}
    \caption{The reduction $F$.}
    \label{the square}
\end{figure}

\paragraph{The Appendix} In the course of the proof of Theorem \ref{thm:main} various numerical
parameters are chosen with complex relationships. The are collected, explicated and shown to be 
coherent in Appendix \ref{NumPar}.

Sections~\ref{sec:Odom} and~\ref{sec:CT} of the body of the paper use certain
standard notions and constructions in ergodic theory and computability
theory. A complete presentation is impossible, but for readers who want an
general overview we present basic ideas from each subject as well exhibit
explicit formulations of certain techniques used in the paper.

Appendix~\ref{app:LB} is an overview of the logical background necessary for
the proof of the theorem. It includes a basic description of $\Pi^0_1$ formulas,
a discussion of bounded quantifiers, how to express Goldbach's conjecture as a
$\Pi^0_1$ formula and the definition of ``truth." Appendix~\ref{app:LB:CT}
gives basic background on recursion theory, computable functions, and primitive
recursion. Appendices~\ref{app:LB:CRF1} and~\ref{app:LB:CRF2} give background on
effectively computable functions. 

Appendix~\ref{app:ET} gives background about ergodic theory and measure theory.
It includes the notion of a measurable dynamical system, the Koopman operator,
and the ergodic theorem. Appendix~\ref{symbsys} describes symbolic systems and
gives the notation and basic definitions and conventions used in this paper.
Appendix~\ref{append:odo} gives basic facts about odometers and odometer-based
systems. These include the eigenvalues of the Koopman Operator associated to an
odometer transformation and the canonical odometer factor associated with an
odometer-based system. Appendix~\ref{factors and joinings} gives basic
definitions including the relationship between joinings and isomorphisms. It
discusses disintegrations and relatively independent products.

Appendix~\ref{diffeos} gives basic definitions of the space of $C^\infty$
diffeomorphisms and gives an explicit construction of a smooth measure preserving
near-transposition of adjacent rectangles. The latter is a tool used in constructing the smooth permutations of
subrectangles of the unit square. These permutations are the basic building
blocks of the approximations to the diffeomorphisms in the reduction. The
section verifies that these are recursive diffeomorphisms with recursive moduli
of continuity and that they can be given primitively recursively. 

\paragraph{Gaebler's Theorem} The writing of this paper began as a collaboration between J. Gaebler and the author  with the goal of recording Foreman's results establishing Theorem \ref{thm:main}. Mathematically, Gaebler was concerned with understanding the foundational significance of Theorem \ref{thm:main}. Though unable to finish this writing project, Gaebler established the following theorem in Reverse Mathematics:

\begin{theorem*}[Gaebler's Theorem] Theorem \ref{thm:main} can be proven in the system ACA$_0$.
\end{theorem*}

This result will appear in a future paper \cite{hans}. 

\paragraph{Acknowledgements} The author has benefited from conversations with a large
number of people. These include J.~Avigad, T. Carlson, S.~Friedman, M.~Magidor, A.~Nies, T.~
Slaman (who pointed out the analogy with Hilbert's 10th problem), J.~Steel, H.~Towsner, B.~Velickovic and others.  B.~Kra was generous with suggestions for the emphases of the paper and
with help editing the introduction. B. Weiss, was always available and as helpful as usual. Finally my colleague A. Gorodetski was indispensable
for providing suggestions about how to edit the paper to make it more accessible to dynamicists. 

\section{Odometer-Based Systems and Reductions}
\label{sec:Odom}
In this section we prove the existence of the preliminary reduction $F_\mco$. 

\begin{theorem}\label{red to odos}
There is a primitive recursive  function $F_\mco$ from the codes for  \(\Pi^0_1\)-sentences to primitive recursive 
construction sequences for ergodic odometer based transformations 
    such that:

    \begin{enumerate}
        \item {\(N\)} is the code for a true statement if and only if {\(F_\mco(N)\)}
            is the code for a construction sequence with limit \(T\), where \(T\) is measure theoretically
            isomorphic to \(T^{-1}\).
        \item For \(M\ne N\), \(F_\mco(M)\) is not isomorphic to \(F_\mco(N)\).
    \end{enumerate}

\end{theorem}

\begin{remark}
When discussing the construction of $F_\mco$ and $F$ we will always have the unstated  assumption  that the input $N$ is  a G\"odel number of a  $\Pi^0_1$-statement.

This is justified by remarking that, though formally the domain of $F_\mco$ (and so of $F$) is the collection of $N$ that are G\"odel numbers of $\Pi^0_1$-statements, the collection of G\"odel numbers of $\Pi^0_1$-statements is primitive recursive. Theorem \ref{thm:main} is equivalent to asking that $F$ be  defined on all of $\nn$ and output a code for the identity map when the input is an $N$ that is not a G\"odel number of a  $\Pi^0_1$-statement.

\end{remark}

\subsection{Basic Definitions}\label{base defs} 

Both Odometer Based and Circular symbolic systems are built using
{\emph{construction sequences}}, a tool we now describe. They code cut-and-stack
constructions and give a collection of words that constitute a clopen basis for
the support of an invariant measure.

Fix a non-empty alphabet $\Sigma$. If $\mcw$ is a collection of words in
$\Sigma$, we will say that $\mcw$ is \emph{uniquely readable} if and only if
whenever $u, v, w\in \mcw$ and $uv=pws$ then either:
    \begin{itemize}
        \item $p=\emptyset$ and $u=w$ or
        \item $s=\emptyset$ and $v=w$.
    \end{itemize}
    A consequence of unique readability is that an arbitrary infinite concatenation of words from $\mcw$ can be 
    \emph{uniquely} parsed into elements of $\mcw$.

Fix an alphabet $\Sigma$. A \emph{Construction Sequence} is a sequence of collections of uniquely readable
words $\la\mcw_n:n\in\nn\ra$ with the properties that:
    \begin{enumerate}
    \item Each word in $\mcw_n$ is in the alphabet $\Sigma$.
    \item {F}or each $n$ all of the words in $\mcw_n$ have the same length
        $q_n$. The number of words in $\mcw_n$ will be denoted $s_n$.
    \item {E}ach $w\in\mcw_{n}$ occurs at least once as a subword of every
        $w'\in \mcw_{n+1}$.
    \item \label{not too much space} {T}here is a summable sequence $\la
        \epsilon_n:n\in\nn\ra$ of positive numbers such that for each $n$, every
        word $w\in \mcw_{n+1}$ can be uniquely parsed into segments 
            \begin{equation}
            \label{words and spacers}
                u_0w_0u_1w_1\dots w_lu_{l+1}
            \end{equation}
	such that each $w_i\in \mcw_n$, $u_i\in \Sigma^{<\nn}$ and for this
	parsing
            \begin{equation}
            \label{small boundary numeric}
                \frac {\sum_i|u_i|} {q_{n+1}}<\epsilon_{n+1}.
            \end{equation}
    \end{enumerate}
The segments $u_i$ in condition~\ref{words and spacers} are called the
\emph{spacer} or \emph{boundary} portions of $w$. The uniqueness requirement in
clause 3 implies unique readability of each word in every $\mcw_n$.

Let $\bk$ be the collection of $x\in \Sigma^\poZ$ such that every finite
contiguous subword of $x$ occurs inside some $w\in \mcw_n$. Then $\bk$ is a
closed shift-invariant subset of $\sz$ that is compact if $\Sigma$ is finite.
The symbolic shift $(\bk, sh)$ will be called the \emph{limit} of $\la \mcw_n:n\in\nn\ra$.

\begin{definition}
\label{principal subwords}
    Let $f\in \bk$ where $\bk$ is built from a construction sequence $\la
    \mcw_n:n\in\nn\ra$. Then by unique readability, for all $n$ there is a
    unique $w\in \mcw_n$ and $a_n\le 0<b_n$ such that $f\rest[a_n, b_n)\in
    \mcw_n$. This $w$ is called the \emph{principal $n$-subword} of $f$.
    If the principal $n$-subword of $f$ lies on $[a_n, b_n)$ we define $r_n(f)=-a_n$, the location of $f(0)$ relative to the interval $[a_n,b_n)$.
    \smallskip 
    
    The construction sequences built in this paper are \hypertarget{stun}{\emph{strongly uniform}} in
that for each $n$ there is a number $f_n$ such that each word $w\in \mcw_n$
occurs exactly $f_n$ times in each word $w'\in \mcw_{n+1}$. It follows that
$(\bk, sh)$ is uniquely ergodic.

    \end{definition} 
We note that in definition~\ref{principal subwords} we must have
$b_n-a_n=q_n$.
 
\paragraph{Notation} For a word $w\in \Sigma^{<\nn}$ we will write $|w|$ for
the length of $w$.

\paragraph{Inverses and reversals} If $\bk$ is a symbolic shift built from a construction sequence $\la
\mcw_n:n\in\nn\ra$ then we can consider its inverse in two ways. The first is
$(\bk, \sh^{-1})$. The second, which we call $\rev{\bk}$ is the system built
from the construction sequence $\la \rev{\mcw_n}:n\in\nn\ra$ where
$\rev{\mcw_n}$ is the collection of reversed words from $\mcw_n$: if $w\in
\mcw_n$ then $w$ written backwards belongs to $\rev{\mcw_n}$. Clearly $(\bk,
\sh^{-1})$ is isomorphic to $(\rev{\bk}, sh)$ and we will use both conventions
depending on context.

\paragraph{Odometer Based construction sequences} A construction sequence with $\mcw_0=\Sigma$ and
built without spacers is called an \emph{odometer-based} construction sequence.
For odometer-based sequences, {Clause 3} of the definition of \emph{Construction Sequence} implies that for odometer based systems
 $\mcw_{n+1}\subseteq\mcw_n^{k_n}$ for some sequence $\la k_n:n\in\nn\ra$ of natural numbers with
$k_n\ge 2$. Hence $|\mcw_{n+1}|\le |\mcw_n|^{k_n}$. In the special case of
odometer sequences we write the length of words in $\mcw_n$ as $K_n$. We note
that $K_n=\prod_{m=0}^{n-1}k_m$.

The sequence $\la k_n:n\in\nn\ra$ determines an \emph{odometer} transformation with domain the
compact space
    \begin{equation*}
        \boldmath{O} =_{def} \prod_n\poz_{k_n}.
    \end{equation*}

The space $\boldmath O$ is naturally a monothetic compact abelian group. We will
denote the group element $(1, 0, 0, 0,\dots)$ by $\bar{1}$, and the result of
adding $\bar{1}$ to itself $j$ times by \hypertarget{barj}{$\bar{j}$}. There is a natural map of
$\boldmath O$ given by $\mco(x)=x+\bar{1}$. Then $\mathcal O$ is a topologically
minimal, uniquely ergodic invertible homeomorphism of $\boldmath O$ that
preserves Haar measure. The map $x\mapsto -x$ is an isomorphism of $\mco$ with
$\mco^{-1}$. (See Appendix~\ref{append:odo} and \cite{FRW} for more background.)

Odometer transformations are characterized by their Koopman operators. They are
discrete spectrum and the group of eigenvalues is generated by the {$K_n$-th}
roots of unity.

\paragraph{{The odometer factor}} If $\bk$ is built from an odometer-based construction sequence and the principal
$n$-subword of $f$ sits at $[-a_n, b_n)$ then the sequence $\la a_n:n\in\nn\ra$
gives a well defined member $\pi_\mco(f)$ of $\boldmath O = \prod_i\poz_{k_i}$.
It is easy to verify that the map $f\mapsto \pi_\mco(f)$ is a factor map. 

A measure preserving transformation is \emph{odometer-based} if it is finite
entropy, ergodic and has an odometer factor. It is shown in \cite{measures}
that every odometer-based transformation is isomorphic to a symbolic shift with
an odometer-based construction sequence.

\subsection{Inverses and factors induced by equivalence relations}
\label{slow and easy}

Fix an odometer based construction sequence $\la \mcw_n:n\in\nn\ra$. If $\mcq$ is an equivalence relation on $\mcw_n$, then elements of $\bk$ can be
viewed as determining sequences of equivalence classes. More precisely if
$\Sigma^*$ is the alphabet consisting of classes $\mcw_n/\mcq$ we can consider
the collection $\mcw_n^*$ of words of length $K_n$ that are constantly equal to
an element of $\Sigma^*$. Let $m>n$. Then for some $K$, the words in $\mcw_m$
are concatenations of $K$-words from $\mcw_n$. Viewed this way, the words in $\mcw_m$
determine a sequence of $K$ many elements of $\mcw_n^*$. Concatenating them we get a word of length 
$K_m$ that is constant on contiguous blocks of length $K_n$.  Let $\mcw_m^*$ be the
collection of words in the alphabet $\Sigma^*$ arising this way. There is a
clear projection map $\pi:\mcw_m\to \mcw_m^*$ that sends two words in $\mcw_m$
to the same word in $\mcw_m^*$ if they induce the same sequence of
$\mcq$-classes. 

Equivalently define the \emph{diagonal} equivalence relation $\mcq^K$ on
$\mcw_n^K$ by setting
    \begin{equation*}
        w_0w_1\dots w_{K-1}\sim w'_0w'_1\dots w'_{K-1}
    \end{equation*}
if and only if for all 
$i, w_i\sim_\mcq w'_i$. Then for two words $w, w'\in \mcw_m, \pi(w)=\pi(w')$ if
and only if $w\sim_{\mcq^K}w'$. Similarly let \(w\in{(\mcw_n/Q)}^K\) and
\(w'\in\mcw_m^K\). Then \(w'\) is a \hypertarget{substitution
instance}{\emph{substitution instance}} of \(w\) if and only if
        \begin{equation*}
            w' = w_0w_1\cdots w_{K-1}\ \text{and}\ w =
            [w_0]_Q[w_1]_Q\cdots[w_{K-1}]_Q.
        \end{equation*}

The sequence $\la \mcw_m^*:m\ge n\ra$ determines a well-defined odometer-based
construction sequence in the alphabet $\Sigma^*$. If we define $\bk_\mcq$ to be
the limit of $\la \mcw_m^*:m\ge n\ra$ then there is a canonical factor map
$\pi_\mcq:\bk\to \bk_\mcq$.

We now discuss how this factor map behaves with inverse transformation. Suppose
that $\poZ_2$ acts freely on $\Sigma^*=\mcw_n/\mcq$. Then for all $K$ we can extend this
action to $(\Sigma^*)^K$ by the \hypertarget {skew-diagonal}
{\emph{skew-diagonal}} action. Suppose that $g$ is the generator of $\poZ_2$.
Define
    \begin{equation*}
        g\cdot([w_0]_\mcq[w_1]_\mcq\dots [w_{K-1}]_\mcq)=g\cdot[w_{K-1}]_\mcq
        g\cdot[w_{K-2}]\dots g\cdot[w_0].
    \end{equation*}
Assume that $\mcw_m^*$ is closed under the skew-diagonal action. Let
    \begin{equation*}
        w=[w_0][w_1][w_{K-1}]\in \mcw_m^*.
    \end{equation*}
Then we can apply $g$ pointwise to the $[w_i]$; i.e. the diagonal action. Since
$\mcw_m^*$ is closed under the skew-diagonal action, the word $g[w_0]g[w_1]\dots
g[w_{K-1}]\in \rev{\mcw_m^*}$.\footnote{We note in passing that being closed under the skew diagonal action 
does not imply that $\mcw_m/\mcq_n$ is closed  under
reverses.}

\begin{lemma}\label{for hans}
    Suppose for all $m>n, \mcw_m^*$ is closed under the skew-diagonal action of
    $g$. Then $\bk_\mcq\cong \rev{\bk_\mcq}$ and the isomorphism takes an $f\in
    \bk_\mcq$ with associated odometer sequence $x$ to an element of $\rev{\bk_\mcq}$ determined by the diagonal action that has 
    associated odometer sequence $-x$. 
\end{lemma}
    
\pf The sequence $\la \rev{\mcw_m^*}: m\ge n\ra$ is a construction sequence for
$\rev{\bk_\mcq}$. The map \[ [w_0][w_1]\dots [w_{K-1}]\mapsto
g[w_0]g[w_1]\dots g[w_{K-1}]\in \rev{\mcw_m^*}\] is an invertible
shift-equivariant map defined on the construction sequences for $\bk_\mcq$ and
$\rev{\bk_\mcq}$ and hence defines an invertible graph joining $\eta_g$ from
$\bk_\mcq$ to $\rev{\bk_\mcq}$\qed 

We note that the graph joining $\eta_g$ does
not depend on which elements of $\mcw_n$ are identified by $\mcq$. Moreover to recover $\rev{\bk}$ from 
$\rev{\bk_\mcq}$ one substitutes the appropriate \emph{reverse} words $\rev{w}$
into a $\mcq$-class $\mcc$.  Frequently
the graph joining $\eta_g$ of $\bk_\mcq$ with $\rev{\bk_\mcq}$ does not come
from a graph joining of $\bk$ with $\rev{\bk}$.

In the construction in \cite{FRW}, which we modify in this paper, this process
is iterated: there is an equivalence relation $\mcq_1$ on $\mcw_{n_1}$ and
another equivalence relation $\mcq_2$ on $\mcw_{n_2}$ with $n_1<n_2$ and
$\mcq_2$ a refinement of the product equivalence relation $\mcq_1^K$ (for the
appropriate $K$). There will be two copies of $\poZ_2$ generated by $g_1$ and
$g_2$ with $g_1$ acting freely on $\mcw_{n_1}/\mcq_1$ and $g_2$ acting freely on $\mcw_{n_2}/\mcq_2$. 

For $i=1, 2$ denote $\mcw_m/(\mcq_i)^K$ by $(\mcw_m^*)_i$. We build two construction sequences consisting of collections of words
made up of equivalence classes $\la (\mcw^*_m)_1:m\ge n_1\ra$ and $\la(\mcw^*_m)_2:m\ge
n_2\ra$ which we assume are closed under the skew-diagonal actions of $g_1$ and
$g_2$ respectively. Let $\bk_1$ be the limit of $\la (\mcw^*_m)_1:m\ge n_1\ra$ and $\bk_2$ the limit of 
$\la(\mcw^*_m)_2:m\ge n_2\ra$.

Then we get a tower 
    \begin{equation*}
        \begin{diagram}
            \node{\bk}\arrow[1]{s}\\
            \node{\bk_{2}}\arrow[1]{s}\\
            \node{\bk_1}
        \end{diagram}
    \end{equation*}
Suppose the $g_2$ action on $\mcq_{2}$ is
\hypertarget {subordinate} {\emph{subordinate}} to the $g_1$ action on
$\mcw_{n_2}/(\mcq_1)^K$; that is, whenever $\mcc_1$ and $\mcc_2$ are classes of
$\mcw_{n_2}/(\mcq_1)^K$ and $\mcw_{n_2}/\mcq_2$ and $\mcc_2\subseteq\mcc_1$,
then $g_2\mcc_2\subseteq g_1\mcc_1$. 
 
Then the various projection maps between $\bk$, $\bk_{\mcq_2}$ and $\bk_{\mcq_1}$   commute
with the shift and the joining $\eta_{g_2}$ of $\bk_{\mcq_2}\times
\rev{\bk_{\mcq_2}}$ extends the joining $\eta_{g_1}$ of $\bk_{\mcq_1}\times
\rev{\bk_{\mcq_1}}$. Given an infinite sequence of equivalence relations
$\mcq_i$, the associated joinings cohere into an invertible graph joining of
$\bk$ with $\rev{\bk}$ if and only if the $\sigma$-algebras associated with the
$\bk_{\mcq_i}$ generate the measure algebra on $\bk$.

\paragraph{Diagonal vs Skew-diagonal actions.}
Since $\acts_n$ extends to both the diagonal and skew-diagonal actions, we summarize the distinct rolls:

\begin{itemize}
	\item The skew-diagonal actions give closure properties on $\mcw_m/\mcq^m_n=(\mcw_m^*)_n$ . 
	\item Because of this closure the diagonal action gives an isomorphism 
	between $\bk_n$ and $\rev{\bk}_n$.  This approximates a potential isomorphism from 
	$\bk$ to $\rev{\bk}$. 
\end{itemize}

\subsection{Elements of the construction} \label{elements}

The construction of the first reduction $F_\mco$ closely parallels the
construction in \cite{FRW} and we refer the reader to that paper for details of
claims made here. For each $N$ the routine $F_\mco(N)$ inductively builds an odometer construction sequence 
$\la \mcw_n : n\in\nn\ra$ in the 
alphabet $\Sigma=\{0,1\}$ with $\mcw_{n+1}\subseteq \mcw_n^{k_n}$.  During the construction we will accumulate inductive numerical requirements.  
Some, such as the $\epsilon_n$'s and the $\varepsilon_n$'s are positive numbers that go to zero rapidly.  Some, such as the $k_n$'s and $l_n$'s are sequences of 
natural numbers that go to infinity.  These numbers depend on $N$, so when necessary we will write $\mcw_n(N)$, $\epsilon_n(N)$, $k_n(N)$, $l_n(N)$ and so forth. 
However for notational simplicity we will drop the $N$ whenever it is clear from context. At stage $n$ in the algorithm $F(N)$ for building $\mcw_n(N)$, for $M<N$ $F$ can recursively refer to objects build by $F(M)$ at stages $\le n$. For example $F(N)$ can assume that $k_n(N-1)$ is known.
 
These sequences of numbers are defined inductively and have complex relationships, requiring some 
verification that they are consistent and can be chosen primitively recursively.  That they are 
consistent is the content of section 10 of \cite{part3}. That they can be chosen primitively 
recursively involves a routine review of the arguments in that paper.  For completeness this is 
done in Appendix \ref{NumPar}.

\begin{description}
\item[Numerical Requirement A]  $s_n=2^{(n+1)e(n-1)}$ for an increasing sequence of natural numbers $e(n)$.
\end{description}

The construction
will use the following auxiliary objects and their properties:
    \begin{enumerate}
    \item A sequence of equivalence relations $\la \mcq_n:1\le n<\infty\ra$.
        Each $\mcq_n$ is an equivalence relation on $\mcw_n$, hence gives a
        factor $\bk_n$ of $\bk$.
    \item The equivalence relation $\mcq_{n+1}$  refines the product
        equivalence relation $(\mcq_n)^{k_n}$ on $\mcw_n^{k_n}$. 
    \item \hypertarget{Ass3}{The sub-$\sigma$-algebra} \(\C H_n\) of $\mcb(\bk)$ corresponding to
        $\bk_n$. In the construction here, as with the original construction,  $\bigcup_n\mch_n$ will generate 
        $\mcb(\bk)$ modulo
        the sets of measure zero with respect to the unique shift-invariant
        measure $\mu$. (This is Lemma \ref{cheat cheat} which uses  \hyperlink{Q4}{specification Q4}.)
    
    {We denote the sub-$\sigma$-algebra of $\mcb(\bk)$ corresponding to the odometer fact by $\mch_0$. 
        Because the odometer factor sits in side each $\bk_n$, $\mch_0\subseteq \mch_n$ for all $n$. }
        
    \item A system of free $\poZ_2$ actions $\acts_n$ on $\mcw_n/\mcq_n$ for
        $n<\Omega$.     
        Denote the generator of $\poZ_2$ corresponding to $\acts_n$ as $g_n$.
    \end{enumerate}

Suppose that $n<m$. As in section \ref{slow and easy}, the words in $\mcw_{m}$ are concatenations of
$K=K_m/K_n$-many words from $\mcw_n$. Hence the product equivalence relation
$(\mcq_n)^K$ gives an equivalence relation on $\mcw_{m}$, which we  call
$\mcq^m_n$.     We will denote $\mcw_m/\mcq^m_n$ by $(\mcw_m^*)_n$.
The $\poZ_2$ actions have the following properties:
    \begin{itemize}
    	 \item $(\mcw_m^*)_n$ is closed under the skew-diagonal action of $g_n$.       
        \item If $n+1<\Omega$, then the $g_{n+1}$ action is \hyperlink{subordinate}{subordinate} to the
            $g_n$ action.
        \item We let $\acts_n$ be the diagonal action of $g_n$ on $\bk_n$.  Since 
        $(\mcw_m^*)_n$ is closed under the skew-diagonal action, $\acts_n$ can be 
        viewed as mapping $(\mcw_m^*)_n$ to $\rev{(\mcw_m^*)_n}$. As described in 
        section~\ref{slow and easy}, for $n < \Omega$, 
            $\acts_n$ canonically creates an isomorphism between $\bk_n$ and
            $\bk_n^{-1}$ that induces the map $x\mapsto -x$ on the odometer
            factor.
    \end{itemize}

Restating this:
if the action $\acts_n$ is non-trivial, then it induces a graph joining $\eta_n$
of $\mch_n$ with $(\mch_n)^{-1}$ that projects to the map $x\mapsto -x$ on the
odometer factor. Assuming $n+1<\Omega$, the action $\acts_{n+1}$ is
{subordinate} to $\acts_n$, the joining $\eta_{n+1}$
projects to the joining $\eta_n$. If $\Omega=\infty$, since the
$\bigcup_n\mch_n$ will generate $\mcb(\bk)$, the $\eta_n$'s will consequently
cohere into a conjugacy of $T$ with $T^{-1}$.

Lemmas 26 and 27 of \cite{FRW} formalize this and show the following conclusion.

\begin{lemma}
    Suppose $\Omega=\infty$. Then there is a measure isomorphism
        $\eta$ of $\bk$ with $\bk^{-1}$ such that for all $n \in \nn$, $\eta$
    induces an isomorphism $\eta_n : \bk_n \to \bk_n$ that coincides with the graph joining determined by the    action of the generator for $\acts_n$ on $\bk_n$.
\end{lemma}

The construction is arranged so that if the number $\Omega$ in clause 3 of the
description of the objects is finite, then $\bk\not\cong\bk^{-1}$. This is done
by making the sequences of equivalence classes of elements of $(\mcw_m^*)_n=\mcw_m/\mcq^m_n$
essentially independent of their reversals subject to the conditions described
above. The specifications given later in this section make this precise.

\subsection{An overview of $F_\mco$.}
\label{sec:Odom:ss:red}

The algorithm for the  reduction $F_\mco$ is diagrammed in Figure~\ref{diagramOfF}.

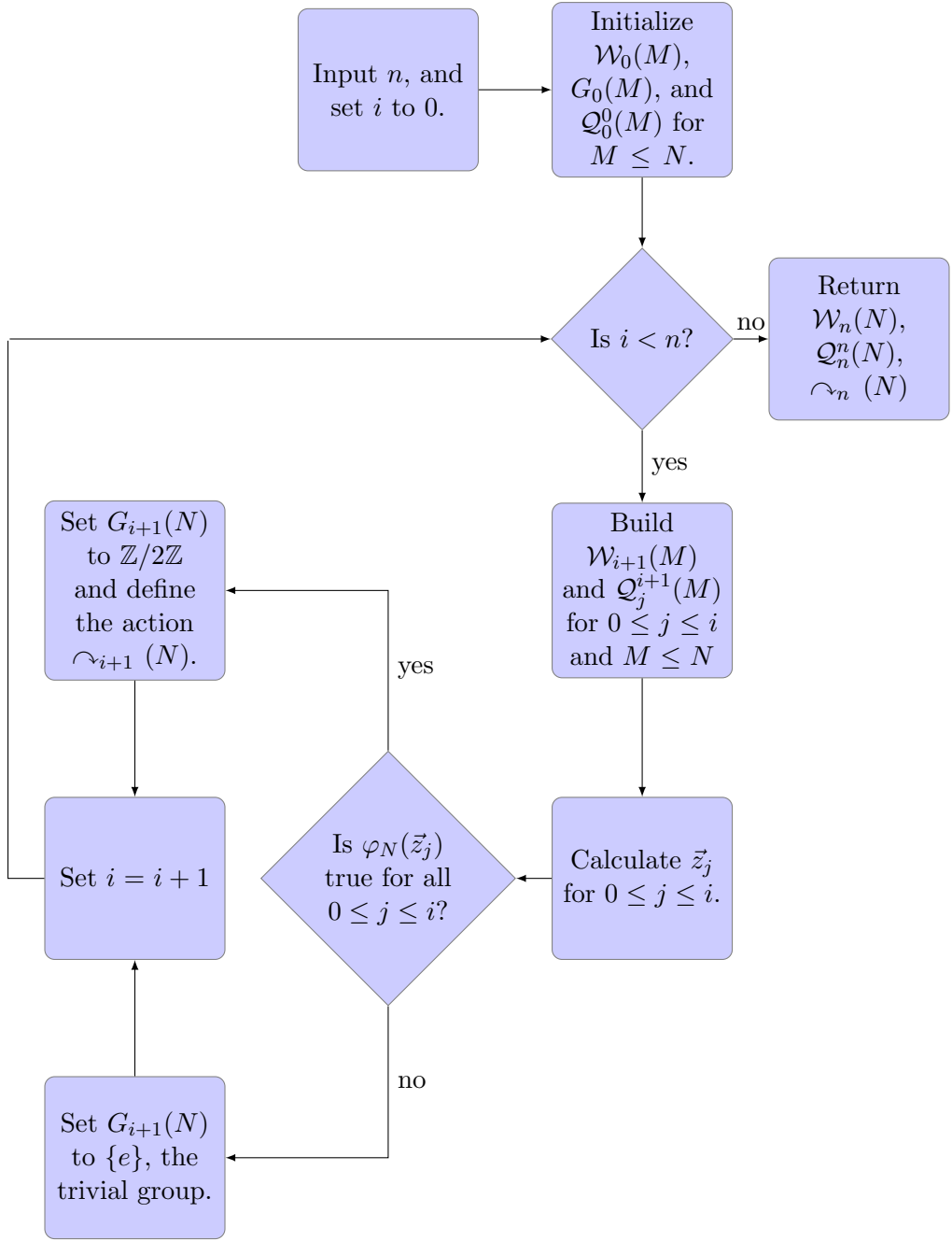
\begin{figure}
  \begin{tikzpicture}[
    decision/.style = {
      draw = black!50,
      diamond,
      fill = blue!20,
      text width = 6em,
      text badly centered,
      node distance = 3cm,
      inner sep = 0pt
    },
    block/.style = {
      draw = black!50,
      rectangle,
      fill = blue!20,
      text width = 6em,
      text centered,
      rounded corners,
      minimum height = 6em,
    },
    node distance = 3cm and 3cm,
    > = latex,
    auto
  ]
    \matrix[column sep = 0.5cm, row sep = 1cm]{
      {} 
    & {}
      & \node[block] (start) {Input \(n\), and set \(i\) to \(0\).};
        & \node[block] (initialize) {Initialize \(\C W_0(M)\), \(G_0(M)\), and \(\C Q^0_0(M)\) for $M\le N$.};
          & \\
      \node[inner sep = 0pt] (p1) {};
    & {}
      & 
        & \node[decision] (s) {Is \(i < n\)?};
          & \node[block] (ret) {Return \(\C W_n(N)\), $\mcq_n^n(N)$, $\acts_n(N)$};\\
      {}
    & \node[block] (yes) {Set \(G_{i+1}(N)\) to \(\B Z / 2 \B Z\) and define the action \(\acts_{i+1}(N)\).};
      & {}
        & \node[block] (build) {Build \(\C W_{i+1}(M)\) and \(\C Q^{i+1}_j(M)\) for \(0 \leq j \leq i\) and $M\le N$};\\
      {}
    & \node[block] (i) {Set \(i = i + 1\)};
      & \node[decision] (decision) {Is \(\phi_N(\vec z_j)\) true for all \(0\leq j \leq i\)?};
        & \node[block] (calc) {Calculate \(\vec z_j\) for \(0\leq j \leq i\).};\\
      {}
    & \node[block] (no) {Set \(G_{i+1}(N)\) to \(\{e\}\), the trivial group.};
      & {}
        & {}\\
  };

  \draw[->] (decision) |- node[near start, swap] {yes} (yes);
  \draw[->] (decision) |- node[near start] {no} (no);
  \draw[->] (s) -- node {no} (ret);
  \draw[->] (s) -- node {yes} (build);
  \draw[->] (build) -- (calc);
  \draw[->] (calc) -- (decision);
  \draw[->] (yes) -- (i);
  \draw[->] (no) -- (i);
  \draw[->] (i) -| (p1) -- (s);
  \draw[->] (start) -- (initialize);
  \draw[->] (initialize) -- (s);
    \end{tikzpicture}
    \caption{The algorithm $R_\phi=F_\mco(N)$.}
    \label{diagramOfF}
\end{figure}

Given $N$, $F_\mco$ determines the $\Pi^0_1$ formula coded by $N$:
    \begin{equation*}
        \phi_N=\forall z_0\forall z_1 \dots \forall z_m\phi(z_0, z_1, \dots z_m).
    \end{equation*}

The function $F$ then uses the formula to generate a computational routine $R_\phi$ that
recursively computes the objects $\mcw_n(N), \mcq_n(N)$ and $\acts_n\!(N)$ (as well as the various numerical parameters that are involved in the construction). 
Here is what
$R_\phi$ does.
\paragraph{The routine $R_\phi$}
    \begin{enumerate}
        \item Fixes a computable enumeration of all $m$-tuples $\la
            \vec{z}_n=(z_0,\dots z_m)_n: n\in \nn\ra$ of natural numbers
        \item On input $n$, $R_\phi$ initializes $i=0$, sets $\mcw_0=\{0,1\}$,
            $\mcq_0$ the trivial equivalence relation with one class and the
            action $\acts_0$ the trivial action.
        \item For $i < n$, $R_\phi$:
            \begin{enumerate}
                \item builds $\mcw_{i+1}, \mcq_{i+1}$, \item computes $\la
                    \vec{z}_j:0\le j\le i\ra$,
                \item Asks:
                    \begin{equation*}
                        \text{``Is $\phi_N(\vec{z}_j)$ true for all $0\le j\le
                        i$?''}
                    \end{equation*}
                
                    Since $\phi_N$ has no unbounded quantifiers, this question 
                    is primitive
                    recursive.
                \item If \emph{yes}, $R_\phi$ builds the action $\acts_{i+1}$
                \item If \emph{no}, $R_\phi$ makes the $\acts_{i+1}$ trivial.
                    (Note that if $i$ is the first integer in this case, then
                    \(\Omega\) will equal \(i+1\).)
             \end{enumerate}
        \item When $i = n$, $R_\phi$ returns $\mcw_n$.
    \end{enumerate}

\subsection{Properties of the words and actions.}
\label{wacs}

We describe the construction sequence, the equivalence relations and the
actions. To start we choose  a prime number $P_0>2$ sufficiently large, and let $\la P_N:N>0\ra$ enumerate the prime numbers 
bigger than $P_0$.

For the construction sequence
corresponding to $F_\mco(N)$, words in $\mcw_1$ have length $P_N$. The
words in $\mcw_n$ will have length ${K_n=P_N2^\ell}$ for some $\ell$ chosen large enough
as specified below. The $K_n$'s will be increasing and $K_m$ divides $K_n$ for $m<n$. Let
$k_n=K_{n+1}/K_n$. Thus $k_n$ is a large power of 2 and each word in $\mcw_{n+1}$
is a concatenation of $k_n$ many words from $\mcw_n$. The number of words in
$\mcw_n$ is $s_n$. We require that $s_n$ divides $s_{n+1}$ and $s_n$ is a power
of 2 that goes to goes to infinity quickly. 
Since $\mcw_{n+1}\subseteq \mcw_n^{k_n}$ this induces lower bounds on the
growth of the $k_n$'s.

The requirements described here are simpler than those in \cite{FRW} as modified
in \cite{part3}, and the ``specifications" used there are appropriately
simplified \emph{or omitted} if not relevant to this proof. The construction
carries along numerical parameters \(\la \epsilon_n \ra\), \(\la k_n \ra\),
\(\la K_n \ra\), \(\la s_n \ra\), \(\la Q_n \ra\), \(\la C_n \ra\), and \(\la
e(n) \ra\). (Showing that the various coefficients are compatible and primitively recursively computable appears in Appends \ref{NumPar}.)
 
As an aid to the reader we use the analogous labels for the simplified specifications
as those that appear in \cite{part3}.

	\begin{description}
    \item[Q4] \hypertarget{Q4}{For} $n\ge 1$, any two $\mcw_n$-words in the same $\mcq_n$
        class agree on an initial segment of proportion at least
        $(1-\epsilon_n)$.
    \item[Q6.] \hypertarget{Q6}{As} a relation on $\mcw_{n+1}$, for $1\le s\le n+1$,  $\mcq^{n+1}_{s}$ refines 
    $\mcq^{n+1}_{s-1}$
        and each $\mcq^n_{s-1}$ class contains $2^{e(n)}$ many $\mcq^n_s$ classes.
  	\end{description}

The point of {\bf Q4} is that the $\mcq_n$ classes approximate words in $\mcw_n$
by specifying arbitrarily long proportions of the words. A consequence of this
is:

\begin{lemma}\label{cheat cheat}
    $\bigcup_n\mch_n$ generates the measure algebra of $\bk$.
\end{lemma}

\pf This is proved in Proposition 23 of \cite{FRW}.\qed
\noindent Thus {\textbf{Q4}} is the justification for \hyperlink{Ass3}{Assertion 3} of
Section~\ref{elements}. 

\bigskip

We now turn to the joining specifications. These are counting requirements that
determine the joining structure. The joining specifications we present here are
more complicated than strictly necessary for the simplified construction in this
paper, but we present them as appear in \cite{FRW} in order to be able to
directly quote the theorems proved there.  We note that specification J10.1 is a strengthening of J10 in \cite{FRW}.

Suppose that $u$ and $v$ are elements of $\mcw_{n+1}\cup \rev{\mcw_{n+1}}$ and
$(u', v')$ an ordered pair from $\mcw_{n}\cup \rev{\mcw_{n}}$. Suppose that $u$
and $v$ are in positions shifted relative to each other by $t$ units. Then an
\emph{occurrence} of $(u', v')$ in $(sh^t(u), v)$ is a $t'$ such that $u'$
occurs in $u$ starting at $t+t'$ and in $v$ starting at $t'$. If $X$ is an alphabet and $\mcw$ is a collection of words in $X$, and $u\in \mcw\cup \rev{\mcw}$ we say 
that $u$ has \emph{forward parity} if $u\in \mcw$ and \emph{reverse} parity if $u\in \rev{\mcw}$. 

By specification \hyperlink{Q4}{Q4} no word in $\mcw_{n+1}$ belongs to $\rev{\mcw_{n+1}}$, so parity is well-defined and unique.  However the words in 
$(\mcw^*_n)_i$ may belong to $\rev{(\mcw^*_n)_i}$ and we view those words as having  both parities.

\begin{description}
    \item[J10.1] \hypertarget{J10.1}{} Let $u$ and $v$ be elements of
        $\mcw_{n+1}\cup\rev{\mcw_{n+1}}$. Let $1\le t<(1-\epsilon_n)(k_n)$. Let
        $j_0$ be a number between {$\epsilon_nk_n$} and $k_n-t$. Then for each
        pair $u', v'\in\mcw_{n}\cup\rev{\mcw_{n}}$ such that $u'$ has the same
        parity as $u$ and $v'$ has the same parity as $v$, let $r(u',v')$ be the
        number of $j<j_0$ such that $(u',v')$ occurs in $(sh^{tK_n}(u),v)$ in
        the $j\cdot K_n$-th position in their overlap. Then
            \begin{equation*}
                \left|\frac {r(u',v')} {j_0} - \frac 1 {
                s_{n}^2}\right|<\epsilon_n.
            \end{equation*}
\end{description}
For fixed  $n$ and $s$, let  \(Q^n_s = |(\mcw^*_n )_s|\) and \(C^n_s\) be the number of
        equivalent elements in each block of the partition \(\C W_n / \C Q^n_s\).
\begin{description}
    \item[J11] \hypertarget{J11}{} Suppose that $ u\in\mcw_{n+1}$ and $v\in
        \mcw_{n+1}\cup \rev{\mcw_{n+1}}$. We let $s=s(u,v)$ be the maximal $i<\Omega$
        such that $[u]_i$ and $[v]_i$ are in the same $\acts_i$-orbit. Let
        $g=g_i$ 
        and $(u', v')\in \mcw_{n}\times (\mcw_{n}\cup\rev{\mcw_n})$ be such that
        $g[u']_s=[v']_s$. Let $r(u',v')$ be the number of occurrences of
        $(u',v')$ in $(u,v)$. Then:
            \begin{equation*}
                \left| \frac {r(u',v')} {k_n} - \frac 1 {Q^{n}_s}\left( \frac 1
                { C^{n}_s}\right)^2\right|<\epsilon_n.
            \end{equation*}
        
\end{description}

The next assumption is a strengthening of a special case of J11.

\begin{description}
    \item[J11.1] \hypertarget{J11.1}{} Suppose that $ u\in\mcw_{n+1}$ and $v\in
        \mcw_{n+1}\cup \rev{\mcw_{n+1}}$ and $[u]_1$ not in the $\acts_1$-orbit
        of $[v]_1$.\footnote{In the language of J11: $s(u,v)=0$, $Q^n_0=1$ and
        $C^n_0=s_n$.} Let $j_0$ be a number between $\epsilon_nk_n$ and $k_n$.
        Suppose that $I$ is either an initial or a tail segment of the interval
        $\{0, 1, \dots K_{n+1}-1\}$ having length $j_0K_n$. Then for each pair
        $u', v'\in\mcw_{n}\cup\rev{\mcw_{n}}$ such that $u'$ has the same parity
        as $u$ and $v'$ has the same parity as $v$, let $r(u',v')$ be the number
        of occurrences of $(u',v')$ in $(u\rest I,v\rest I)$. Then:
            \begin{equation*}
                \left| \frac {r(u',v')} {j_0} - \frac 1 {
                s_n^2}\right|<\epsilon_n.
            \end{equation*}
\end{description}

The properties and specifications described above imply the specifications in
\cite{FRW} as well as J10.1 and J11.1 from \cite{part3}.

\begin{remark}
We note that specification J10.1 implies unique readability of the words in $\mcw_{n+1}$.  This follows by induction on $n$. If the words in 
$\mcw_{n+1}$ were not uniquely readable then we would have $u, v, w\in \mcw_n$ with $uv=pws$ and neither $p$ nor $s$ empty.  But the one of $u$ or $v$ would have to overlap either an initial segment or a tail segment of $w$ of length $K_{n+1}/2$. Suppose it is an initial segment of $w$ and a tail segment $u$.  On this tail segment the $n$-subwords would have to agree exactly with the $n$-subwords of an initial segment of  $w$.  But this contradicts $J10.1$.
\end{remark}
\medskip

Suppose we have built a collection of words $\la \mcw_n : n\in \nn\ra$, equivalence
relations $\la \mcq_n : n\in\nn\ra$ and actions $\la \acts_n : n\in\nn\ra$
satisfying the properties described then we can cite the following results occurring
in \cite{FRW}. Fix a transformation $T$ built with the construction sequence 
$\la \mcw_n : n\in\nn\ra$. Recall that if $\acts_n$ is non-trivial then the generator $g_n\ne 0$ induces an
invertible graph
joining $\eta_n$ of $\bk_n$ with $\bk_n^{-1}$. We quote the following results of
\cite{FRW}, referencing their numbers in that paper. 

\begin{description}
    \item[Theorem 13 and Proposition 32] Suppose that $\eta$ is an ergodic
       {joining of $T$ with $T^{-1}$ that is not a relatively independent joining over the odometer factor.} 
        Then $\eta\rest \mch_0\times \mch_0$ is
        supported on the graph of some \hyperlink{barj}{$\bar{j}$}-shift of the odometer
        factor.

    \item[Proposition 37] If $\eta$ is an ergodic joining of $\bk$ with
        $\bk^{-1}$, then exactly one of the following holds:
            \begin{enumerate}
                \item $\Omega<\infty$ and for some $n\le \Omega$, $j\in \poZ$ and some $\eta_n$, \(\eta\) is
                    the relatively independent joining of $\bk$ with $\bk^{-1}$
                    over the joining $\eta_n\circ(1,sh^{-j})$ of $\bk_n\times
                    \bk_n^{-1}$.
                \item $\Omega=\infty$ and for some $j$, all
                    $n$ the projection of $\eta$ to a joining on $\bk_n\times
                    \bk_n^{-1}$ is of the form $\eta_n\circ(1,sh^{-j})$
            \end{enumerate}
        If $\Omega=\infty$, since the $\mch_n$'s generate,  $\eta$ is an invertible graph joining of $\bk$ with $\bk^{-1}$. In both cases the projection of $\eta_n$ to a joining of the odometer
        factor with itself concentrates on the map $x\mapsto -x$.
\end{description}
Thus it follows that:
\begin{enumerate}
\item If $\bk\cong \bk^{-1}$ then $\Omega=\infty.$ In particular
if $\bk\cong \bk^{-1}$, then the $\Pi^0_1$ statement $\phi_N$ is true.
\item The projection of  $\eta_n\circ(1,sh^{-j})$ to the odometer is of the form $x\mapsto -x-j$.
\item Similarly the projection of $\eta\circ(1,sh^{-j})$ to the odometer is of the form $x\mapsto -x-j$.
\end{enumerate}

Clause 2 of Theorem~\ref{red to odos} requires that if $M\ne N$ are different codes
for $\Pi^0_1$ sentences then the transformation $F_\mco(M)$ is not isomorphic to
$F_\mco(N)$. This is clear because the odometer sequence for $F_\mco(M)$
consists of $k$'s whose prime factors are $2$ and $P_M$, while the odometer
sequence for $F_\mco(N)$ has $k$'s whose prime factors are $2$ and $P_N$. Since $P_M\ne P_M$, 
the odometer factors are not isomorphic. 

Corollary 33 of \cite{FRW} implies that the Kronecker factor of each $F_\mco(N)$ is the odometer factor. Since 
any isomorphism $\phi$
 between $F_\mco(M)$ with
$F_\mco(N)$ must induce an isomorphism of the Kronecker factors,   $\phi$ has to induce an isomorphism of the corresponding odometer
factors, yielding a contradiction. {(See Corollary~\ref{cor:odoms} in
Appendix~\ref{app:ET})}

To finish the proof of Theorem \ref{red to odos} we must show that the words, equivalence relations and actions can be  built primitively
recursively.
\subsection{Building the words, equivalence relations and actions}
\label{bldg the words}

To finish the proof of Theorem~\ref{red to odos} the words $\mcw_n(N)$,
the equivalence relations $\mcq_n(N)$ and actions $\acts_n(N)$ must be constructed and
it must be verified that the construction is primitive recursive.

\paragraph{Note:} Formally we are just constructing actions
$\acts_n$ for $n<\Omega$.  However for notational convenience, when constructing the words at stage $n+1$,  we will write $\acts_i$ when $\Omega\le i<n+1$ with the understanding that it is the trivial identity action.
\bigskip

The collections of words $\mcw_n$ are built probabilistically.  A finitary version of law of large numbers shows that there are primitive recursive bounds on the length of the  words in a collection with the  necessary properties.  The actual collection of words can then be found with an exhaustive search of collections of words of that length, showing that the entire construction is primitive recursive.

\paragraph{{Structure of the induction.}} The collections of  words $\mcw_n$ are built by induction on $n$. For $n\ge 1$ the words in  $\mcw_{n+1}$ are built by iteratively substituting  words into $k_n$-sequences of classes 
$\mcq_i^n$, by  induction on $i\le n$.  We will adapt the notation of section \ref{slow and easy}.

The
length $K_1$ of words in $\mcw_1$ will be a large prime number $P_N$.
 To
pass from stage $n$ to $n+1$, one is required to build the words $\mcw_{n+1}$,
the equivalence relation $\mcq_{n+1}$ and, if $n+1<\Omega$ the action
$\acts_{n+1}$. The length $K_{n+1}$ of the words will be ${2^\ell \cdot
K_n}$ for an $\ell$ taken large enough. 

Suppose we have already chosen $k_n$ and it is a large power of 2. Then $(\mcq^n_i)^{k_n}$ for $0\le i\le
n$ give us a hierarchy of equivalence relations  of potential words as described in
Section~\ref{slow and easy} as well as the diagonal and skew-diagonal actions of $\acts_i$ for $i<\min(n,\Omega)$. 

\begin{remark} The construction of $\mcw_{n+1}$ is top-down.
We construct the 
$(\mcw^*_{n+1})_i=\mcw_{n+1}/\mcq^{n+1}_i$  by induction on $i$ \emph{before} we construct $\mcw_{n+1}$. The equivalence relations get more refined as $i$ increases, so each step  gives more 
information about 
$\mcw_{n+1}$. Having built $(\mcw^*_{n+1})_n$, an additional step constructs creates $\mcw_{n+1}$ and the equivalence relation $\mcq^{n+1}_{n+1}$.
\end{remark}

Start with $i=0$. Then $\mcw_n/\mcq^n_0$ has one element, a string of length
$K_n$ with a single letter. Let $(\mcw_{n+1}^*)_0$ be the single element
consisting of strings of length $k_n \cdot K_n$ in that single letter. 

Each element of $(\mcw^*_{n+1})_1$ is built by substituting $k_n$ elements of 
$(\mcw_n^*)_1$---each of which is a contiguous block of length
$K_n$---into $(\mcw_{n+1}^*)_0$.  We continue this
process inductively, ultimately arriving at $(\mcw^*_{n+1})_n$. 

\begin{center}
    \begin{tabular}{|c|c|}
        \hline
        The elements $X$ being substituted
            & The result of the substitution\\
        into previous words
            & \\ \hline\hline
            &$(\mcw_{n+1}^*)_0$\\ \hline
        $(\mcw_n^*)_1$
            & $(\mcw_{n+1}^*)_1$\\ \hline
        $(\mcw_n^*)_2$
            &$(\mcw_{n+1}^*)_2$\\ \hline
        \vdots 
            &\vdots\\ \hline
         $(\mcw_n^*)_n$
            &$(\mcw_{n+1}^*)_n$\\ \hline
    \end{tabular}
\end{center}

The result of this induction is a sequence of elements of $\mcw_n/\mcq_n$ of
length $k_n*K_n$, that is constant on blocks of length $K_n$. We must finish by substituting elements of 
$\mcw_n$ into the
$\mcw_n/\mcq_n$-classes to get $\mcw_{n+1}$ and defining $\mcq_{n+1}$. 

\paragraph{A step in the induction on $i$.}
 Fix an $i$ and view elements   $(\mcw_{n+1}^*)_i$ as $k_n$-sequences $C_0C_1\dots C_{k_n-1}$ of elements of $(\mcw^*_n)_i$.
Since $\mcq_{i+1}$ refines the diagonal equivalence relation $(\mcq_i)^{K_{i+1}/K_i}$, $(\mcq_{i+1}^n)^{k_n}$ refines 
$(\mcq_i^n)^{k_n}$.
Inside each $\mcq^{n}_i$ class $C_j$,  one can choose a 
$\mcq^n_{i+1}$ class $C'_j\in (\mcw_n^*)_{i+1}$ .  Concatenating these to get $C_0'C_1'\dots C_{k_n-1}'$ we create an element of 
$(\mcw^*_{n+1})_{i+1}$. We do the construction so that result is closed under the skew diagonal action of 
$\acts_{i+1}$.

\begin{remark}
Following section \ref{slow and easy}, elements of $(\mcw_{i+1}^*)_{i+1}$ are constant sequences of length $K_{i+1}$.  Thus the concatenation $C_0'C_1'\dots C_{k_n-1}'$ is a sequence of $k_n*(K_n/K_{i+1})$ many contiguous constant blocks of length $K_{i+1}$.

\end{remark}

\medskip

We now describe how these choices are made. Our discussion is aimed at the case where $n+1<\Omega$, for $n+1\ge \Omega$ take $\acts_{n+1}$ to be the trivial action. Fix a candidate $k$ for $k_n$. View $\{rev\}$ as acting on $(\mcw_n/\mcq_i^n)^{k}=((\mcw_n^*)_i)^k$. Together, the skew-diagonal action of $\acts_i$ and $\{rev\}$ generate an action on $(\mcw_n/\mcq_i^n)^{k}$. Let $R_i$ be a set of representatives  of each orbit of this action.
Fix the number {\(E\)} of
$i+1$-classes desired inside each $i$-class. Consider
    \begin{equation}
    \label{eq:sub_lem_lite}
        \B X_i=\prod_{r\in R_i}\prod_{q=0}^{E-1} S(r,q),
    \end{equation}
where $S(r,q)$ is the collection of all substitution instances of $\mcq^n_{i+1}$
classes into $r$.\footnote{Note that $q$ is a dummy index variable here.} More explicitly, if $r=C_0C_2\dots C_{k-1}$ where $C_j\in
\mcw_n/\mcq^n_i$. Let $C_j^*=\{C':C'\subseteq C_j \mbox{ and } C'\in
\mcw_n/\mcq^n_{i+1}\}$. For each $0\le q\le E-1$, let
    \begin{equation*}
        S(r,q)=\prod_{j=0}^{k-1}C_j^*.
    \end{equation*}

Fix an $r\in R_i$. The every element $\mcw$ of $\prod_{q=0}^{E-1} S(r,q)$ can be viewed as 
a collection of $E$ many words of length $k$ in the language $(\mcw_n^*)_{i+1}$ whose 
$\mcq_i^n$ classes form $r$. 
Each of these $E$ many words can be copied by the $\acts_{i+1}$ action. If $w$ is such a word, and is a substitution instance of $r$ then $\acts_{i+1}(w)$ is a substitution instance of $\acts_i(r)$.  

So comparing elements of $\mcw$ (and their shifts) is the same as comparing potential words in 
$(\mcw_{n+1}^*)_{i+1}$. The action of $\acts_{i+1}$ preserves the frequencies of occurrences of words in

We work with  $\mathbb X_i$ because it can be viewed as a discrete measure space with the counting
measure. The objects being counted in the various specifications correspond to random variables on this
measure space. 

\begin{definition}
If $\la w_{r,q}:r\in R_i, 0\le q<E\ra$ is the collection of words built using the Substitution Lemma passing from stage $i$ to stage $i+1$, the $(\mcw^*_{n+1})_{i+1}$ is the closure of $\{ w_{r,q}:r\in R_i, 0\le q<E\}$ under the skew-diagonal action of 
$\acts_i$.
\end{definition}

\begin{ex}\label{does this count?} If $C\subseteq C_j, D\subseteq C_{j'}$ are
substitution instances,  we have the independent random variables $X_{r,q,j}, X_{r',q',j'}$
taking value 1 at points $\vec{x}\in \mathbb X_i$ where $x(r,q, j)=C$ and 
$x(r',q',j')=D$, respectively. The event that $C$ occurs in $\mathbb X_i$ in the $q^{th}$ word in position $j$ and $D$ occurs in $r'$ in the $(q')^{th}$ word in position $j'$  is the event that both $X_{r',q',j'}=1$ and
$X_{{r,q,j}}=1$. If each $i$-class has {\(p\)} elements then the probability that both $X_{r',q',j'}=1$ and
$X_{{r,q,j}}=1$ is $1/p^2$. 
\end{ex}

The strong law of large numbers tells us that the collection of points in each
$\mathbb X_i$ that do \emph{not} satisfy the specifications (as they are coded in the conclusion of the Substitution Lemma) goes to zero
exponentially fast in $k$.  As $k$ grows, the number of requirements to satisfy the Substitution Lemma grows linearly. Hoeffding's inequality 
(Theorem \ref{Hoeffding's Inequality} below) says that the probabilities stabilize exponentially fast.  The Substitution Lemma follows.

\paragraph{In more detail:} The word construction proceeds by first getting a
very close approximation to what is desired and then \emph{finishing} the
approximations to exactly satisfy the requirements. These two steps correspond to
Proposition 43 and Lemma 41 of \cite{FRW}.

The general setup for the Substitution Lemma (Proposition~\ref{lem:sl}) at stage $n+1$ is as follows:
    \begin{itemize}
        \item An alphabet \(X\) and an equivalence relation \(\C Q\) on
            \(X\), with \(Q\) classes each of cardinality \(C\).

	\item A collection of words $\mcw\subseteq (X/\mcq)^k$ for some $k$.
	 \item Groups $G,H$ with generators $g, h$ that are either $\poZ_2$ or the trivial group. If $H=\poZ_2$ then $G=\poZ_2$. 
	   \item If $G=\poZ_2$ then we have a free action \(G\acts X/Q\) and if $H=\poZ_2$ we also have a free action   \(H\acts X\). Thus the skew-diagonal actions of $G$ on $(X/Q)^k$ and $H$ on $X^k$ are well-defined.  If either group is trivial, then the corresponding actions are trivial.
        \item The \(H\acts X\) action is \hyperlink{subordinate}{subordinate}
            to \(G\acts X/Q\) action via \(\rho\).
        \item Constants \(\epsilon_a, \epsilon_b\in (0,1)\) such that \(\epsilon_b
            < \epsilon_a^2 / 5 |X|\).
        \item A constant $E$ determining the number of substitution instances
            desired for each \(\C Q\) class.
      
        \item If $u, v, w, w'$ are words in the alphabet  $X$, then $r(u,v,sh^i(w),w')$ is the number of $j$ such that $u$ occurs in $w$ starting at $j+i$ and $v$ occurs in $w'$ starting at $j$. Similarly if $u, w$ are words in the alphabet $X$, the $r(u,w)$ is the number of occurrences of $u$ in $w$.
    \end{itemize}
A special case of the Substitution Lemma (Proposition 63 in \cite{FRW}) is:

\begin{prop}[Substitution Lemma]
\label{lem:sl}
   Let $E>0$ be an even number. There is a lower bound $\klb$ depending on
   $(\epsilon_b,\epsilon_a, Q, C, W, E)$ such that for all numbers $k \ge \klb$
   and all symmetric  $\mcw\subseteq {(X/\mcq)^k}$ with cardinality $W$ that are
   closed under the skew-diagonal action of $G$ and $\rev{}$, {\bf if} for all $i$ with $1\le
   i\le (1- \epsilon_b)k$, $u,v\in X/\mcq$ and $w,w'\in \mcw$:
        \begin{equation}
        \label{eqn: J11b}
            \left| \frac {r(u,v,sh^{i}(w),w')} {k-i} - \frac  1
            {Q^2}\right|<\epsilon_b
        \end{equation}
   and  each $u\in X/\mcq$ occurs with frequency $1/Q$ in each $w\in \mcw$,

   {\bf then} there is a collection of words $S\subseteq {X}^k$ consisting of
   substitution instances of $\mcw^k$ such that if $\mcw'=HS\cup \rev{HS}$ we
   have:\footnote{$H$ is acting on $X^k$ by the skew-diagonal action.}
       \begin{enumerate}
           \item \label{item 1}  Every element of $\mcw'$ is a substitution
               instance of an element of $\mcw$ and each element of $\mcw$ has
               exactly $E$ many substitution instances of words in  
               $\mcw'$.

           \item\label{item 2} For each $x\in X$ and each $w\in \mcw'$
                   \begin{equation}
                   \label{eqn: approximate frequencies}
                       \left| \frac {r(x,w)} {k} - \frac 1
                       {|X|}\right|<\epsilon_a
                   \end{equation}
               i.e., the frequency of $x$ in $w$ is within $\epsilon_a$ of
               $1/|X|$.
	 \item \label{2 bis} If $w_1, w_2\in S\cup \rev{S}$ with $[w_1]_\mcq=[w_2]_\mcq$ and 
 $w_2\notin H_0w_1$ and $x,y\in X$ with $[x]=[y]$. Then for $h\in H_0$:\footnote{
                   While there are typographical errors  in the statement of
                   this item in \cite{FRW}, the proof given there yields the
                   correct statement which is 
                   inequality \ref{freeness persists}. Similarly, conclusion \ref{conclusion 3} has been strengthened slightly here in
                   a way that does not materially change the proof. 
                }

                            \begin{equation}
                            \label{freeness persists}
                                \left| \frac {r(x,y,w_1,hw_2)} {k} - \frac 1 {Q\cdot
                                C^2}\right|<\epsilon_a.
                            \end{equation}

                  \item \label{conclusion 3} Let $i$ be a number with $1\le i\le (1-\epsilon_a)k$ and $j_0$ be a number between $\epsilon_ak/2$ and $k-i$, 
                  $x, y\in X$, 
		$w_1, w_2\in\mcw'\cup \rev{\mcw'}$, let $r(x,y)$ be the number of $j<j_0$ such that $(x,y)$ occurs in $(sh^i(w_1),w_2)$ in the 
		$j^{th}$ position. Then
                   \begin{equation}
                   \label{eqn: J11a}
                       \left| \frac {r(x,y)} {j_0} - \frac 1
                       {|X|^2}\right|<\epsilon_a.
                   \end{equation}

            \item\label{fourth item} For all $x, y\in X$ and all $w_1, w_2\in
               \mcw' \cup \rev{\mcw'}$ with different $H$ orbits, 
                    \begin{equation}
                        \label{eqn: estimate before sub} \left|\frac
                        {r([x]_\mcq,[y]_\mcq,[w_1]_\mcq,[w_2]_\mcq)}
                        {k}-c\right|<\epsilon_b
                    \end{equation}
                implies that,
                    \begin{equation}
                    \label{eqn:estimate after sub}
                        \left|\frac{r(x,y,w_1,w_2)} {k} - \frac c
                        {C^2}\right|<\epsilon_a.
                    \end{equation}
       \end{enumerate}
\end{prop}

We remark again that the Law of Large numbers implies that conclusions 1-5 hold for almost all infinite sequences. For example if you perform i.i.d. substitutions of elements of $X$  to create a typical infinite sequence $\vec{w}$, then the density of occurrences of  a given $x$ in $\vec{w}$ will be $1/|x|$.  The Hoeffding inequality says that the finitary approximations to this conclusion converge exponentially fast. As a result, for large enough $k$ it is possible to satisfy conclusions 1-5 with very high probability. 

Another remark is that at each stage we start with a collection of words $\mcw$ closed under reversals and produce another collection of words $\mcw'$ closed under reversals. 
\bigskip
\paragraph{The sequence $e(n)$.} We will have a sequence $e(n)$ such that for $n\ge 1$, 
{$s_{n+1}=2^{(n+2)e(n)}$} that satisfies some growth conditions. (See \emph{Inherited Requirement 2} and \emph{Inherited Requirement 3} in Appendix \ref{NumPar} and Figure \ref{ecosystem} for an explicit statement of these conditions.)
To initialize the construction we take $e(0)=2$.

\paragraph{Finding $k_n$} We now use Proposition \ref{lem:sl} to build the collections of words. We will apply it with $E=2^{e(n)}$ except in one instance where we apply it with $E=2^{2e(n)}$.
To start the inductive construction, we take $P_0$ to be large enough to apply
the Substitution Lemma $\mcq_0$ the trivial equivalence
relation and $\mcw_0=\Sigma=\{0,1\}$. For $N>0$,  since $P_N\ge P_0$, $P_N$ can
also be used for $k_0(N)$ to initialize the construction as described below with $n=0$.

We then choose $k_n$ large enough to allow $n+1$ successive Lemma \ref{lem:sl}-style substitutions for $E=2^{2e(n)}$ corresponding to the equivalence relations  $\mcq^n_i$ for $1\le i\le n$ together with a final substitution of the letters in the base alphabet $\Sigma$ to produce $\mcw_{n+1}$.  (This is a total of $n+2$ substitutions.)

More explicitly, note that each of the \(n+2\) applications of the Substitution Lemma for the various $\mcq_i$ with $E=2^{2e(n)}$ and 
$\epsilon_a=\epsilon_n/100$  and the  finishing
lemma produces a lower bound $\klb^i$.

The following will be important later in the paper:

\begin{description}
\item[Numerical Requirement B] \hypertarget{nrb}{Let} $k_n(N-1)$ be the $k_n$ corresponding to the reduction $F_\mco(N-1)$ and $k_n(N)$ be the $k_n$ corresponding to the reduction $F_\mco(N)$ and $k_n(N)$.  Then  
	\begin{equation}
	k_n(N)\ge k_n(N-1)\label{ugly}
	\end{equation}
\end{description}

Choose a large power of two
    \begin{equation*}
        \kmax > \max \{ \klb^0, \klb^1, \ldots, \klb^n, k_n(N-1) \},
    \end{equation*}
ensuring that it be sufficiently large that \(2^{-\kmax} < \epsilon_n\).
Then, set 
	\begin{equation}
	k_n = \kmax^2*s_{n}.\label{kn def}
	\end{equation}

 Since $k_{n}$ is of this form and $s_{n}$ is a power of 2, this ensures that \(K_{n+1} = P_N \cdot
2^\ell\) for a large \(\ell\). By increasing $\kmax$ if necessary we can also assume 
	\begin{enumerate} 
	\item $1/k_n<\epsilon^3_n/4$.  
	\item $s_{n+1}\le s_n^{k_n}$.
	\end{enumerate}

\paragraph{Building $\mcw_{n+1}/\mcq_i^{n+1}$ for $i\le n$:}

This is done by applying the Substitution Lemma $n$ times to pass from 
$(\mcw^*_{n+1})_0$ successively to $(\mcw^*_{n+1})_n$. At each $i<n$ we substitute 
$2^{e(n)}$ many elements of $(\mcw^*_{n+1})_{i+1}$ into each element of 
$(\mcw^*_{n+1})_i$.

\paragraph{Completing $\mcw_{n+1}$:} Having constructed $\mcw_{n+1}/\mcq_n$ it remains to construct \(\C
W_{n+1}\), $\mcq_{n+1}$ and the action $\acts_{n+1}$. The latter is only relevant if $n+1<\Omega$.

We must ensure that the resulting collection of
words satisfy \hyperlink{Q4}{Q4} and \hyperlink{Q6}{Q6}. This is accomplished by
constructing \emph{two} collections of words, the stems and the tails.\footnote{
    Cf.\ Propositions 66 and 65, and Section 8.3, in \cite{FRW}. }

Start by rewriting $\kmax^2$ as $(\kmax^2-\kmax)+\kmax$. 
\begin{itemize}
                \item \textbf{The tails:}
                To build the tails, which have length $\kmax s_n K_n$, we use Lemma \ref{lem:sl}, with $X=\mcw_n$ and $\mcq=\mcq_n$ to build $2^{2e(n)}$ many substitution instances in each 
		$\mcq^{n+1}_n$-class $C$ of the final $\kmax s_n$ portion of each word in 
		$(\mcw^*_{n+1})_n$. We call these the \emph{tails} corresponding to $C$. 

            \item \textbf{The stems:} The stems have length $(\kmax^2-\kmax)s_nK_n$.  
            We use Lemma \ref{lem:sl}, again with $X=\mcw_n$ and $\mcq=\mcq_n$, to 
            create 
            $2^{e(n)}$ many substitution instances in each initial segment of a 
 $(\mcw^*_{n+1})_n$-word of length $\kmax^2-\kmax$. We call these the \emph{stems} 
 corresponding to the initial segments of the $\mcq_n^{n+1}$-class $C$ of this word.

    \end{itemize}
  The words in $\mcw_{n+1}$ are built one $\mcq_n^{n+1}$ class at a time.  Fix such a class 
$C$.  Then  $C$ has $2^{e(n)}$ many stems in the first  $\kmax^2-\kmax$ and 
$2^{2e(n)}$ many tails in the final segment of length $\kmax$.  Pair each stem with 
$2^{e(n)}$ many tails to create the words in $\mcw_{n+1}$ that belong to $C$. This puts $2^{2e(n)}$ words into each $C$.

Each equivalence class in $\mcq_{n+1}^{n+1}$ consists of taking all words starting with a single fixed stem. It is immediate that there 
are $2^{e(n)}$ many $\mcq_{n+1}^{n+1}$-classes in each $\mcq^{n+1}_n$ class and that each $\mcq^{n+1}_{n+1}$ class has 
$2^{e(n)}$ many words in it.  Moreover each class is associated with a fixed stem of length $\kmax^2-\kmax$ followed by many 
short tails.  Thus specifications  \hyperlink{Q4}{Q4} and  \hyperlink{Q6}{Q6} are  satisfied. 

Finally we note that $\mcw_{n+1}$ was built by $n+2$ many successive substitutions of size $2^{e(n)}$ into equivalence classes.   Thus $s_{n+1}=2^{(n+2)e(n)}$.

\paragraph{Why does this work?} Though it appears in detail in \cite{FRW}, for the reader's edification it may be appropriate to say a few things about how the stems/tails construction affects the statistics. This issue is most cogent in J10.1, where $sh^{tK_n}(u)$ and $v$ are being compared on small portions of their overlaps. By the manner of construction of the stems, where the stem of $sh^{tK_n}(u)$  overlaps with the stem of $v$  conclusion \ref{conclusion 3} of Proposition \ref{lem:sl} holds with 
$\epsilon_a=\epsilon_n/100$.
 
 Since $j_0\ge \epsilon_nk_n$ the total length of the overlap is at least $\epsilon_nk_nK_n$. The tails have length $\kmax s_nK_n$, so the proportion of the overlap 
 taken up by the tails is at most 
 \begin{align*}{2\kmax s_n\over j_0}&\le
 {2\kmax s_n\over \epsilon_nk_n}<{2\kmax\epsilon_n^3\over 100\epsilon_n}\\
 &<{\kmax\epsilon_n^2\over 50}<\epsilon_n/50.
 \end{align*}
 The specification J10.1 approximates the proportion  of $j<j_0$ where $(u',v')$ occur.  This proportion is the weighted average of the  proportion $P_S$ of $j<j_0$ where 
 $(u',v')$ occur in the overlaps of the stems and the proportion $P_T$ of $j<j_0$ where $(u',v')$ occur in an overlap of a stem with a tail. Let $\alpha$ be the proportion of the overlap of $sh^{tK_n}(u)$ and $v$ that occurs on the stems.  By the above, $\alpha>1-\epsilon_n/50$.  Then
\[
 \left|\frac {r(u',v')} {j_0} - \frac 1 {s_{n}^2}\right|=\left|(\alpha P_S + (1-\alpha)P_T)-{1\over s_n^2}\right|
\]
On the overlap of the stems $|P_S-{1\over s^2_n}|<\epsilon_n/100$. Since $P_T\in [0,1]$ and 
 $(1-\alpha)<\epsilon_n/50$, we see that 
 \[
 \left|\frac {r(u',v')} {j_0} - \frac 1 {s_{n}^2}\right|<\epsilon_n.\]
Hence J10.1 holds.

\paragraph{The action $\acts_{n+1}$.}

    \begin{description}
            \item[Case 1 ($n+1\ge\Omega$):] In this case, the action of
                \(\acts_{n+1}\) is trivial, so there is nothing further to be
                done.
            \item[Case 2 ($n+1<\Omega$):] In this case, we need to define
                $\acts_{n+1}$ to be subordinate to $\acts_n$.  Fix a $\mcq^{n+1}_n$-class $C$ and suppose that $C$ gets sent to 
                $D$ by $\acts_n$. Since each $\mcq^{n+1}_n$
                class has the same number of elements we can define $\acts_{n+1}$ so that it induces a bijection between the $\mcq_{n+1}$ subclasses of $C$ and $D$.                 
    \end{description}

\paragraph{The construction of the $\mcw_n, \mcq_n$ and $\acts_n$ is primitive recursive} Here is a standard theorem:

\begin{theorem}[Hoeffding's Inequality]
\label{Hoeffding's Inequality}
    Let \(\la X_n : n \in \B N\ra\) be a sequence of i.i.d.\ Bernoulli random
    variables with probability of success \(p\). Then,
        \begin{equation*}
            \B P \left( \left|\frac 1 n \sum_{k = 0}^{n-1} X_k - p
                \right| > \delta \right) <
                \exp\left(-\frac{n\delta^2}{6}\right).
        \end{equation*}
\end{theorem}

\begin{lemma}\label{HoefPR}
    The construction of the sequence $\la\mcw_n, \mcq_n, \acts_n:n\in\nn\ra$ is
    primitive recursive.
\end{lemma}

\begin{proof}
    The only part of the construction that is not a completely explicit induction is finding the collection of words satisfying the 
    conclusions of the Substitution
    Lemma. For each candidate fixed $k$ one can primitively recursively search \emph{all}
    substitution instances to see if there is a collection of words of length $k*K_N$ that works.  Using Hoeffding's inequality can give an   explicit upper bound for a $k$ that works.  The algorithm first computes a $k_n$ that works and then does the search.
\end{proof}
\noindent This completes the proof of Theorem \ref{red to odos}.
\bigskip

\begin{remark}Two remarks are in order.
\begin{itemize}
\item The asymmetry of the words in the last step of the construction of $\mcw_{n+1}$ appears problematic.  How can the words all be oriented left-to-right stem and tail if they are supposed to be closed under all the various skew-diagonal actions at stage $n+1$ and later?

The answer is that the asymmetries are covered up by the equivalence classes. For example, the words in 
$\mcw_{n+1}/\mcq_{n+1}^{n+1}$ are all constant sequences of length $K_{n+1}$.  If $w\in \mcw_{n+1}$ and $C$ is the 
$\mcq^{n+1}_{n+1}$-class corresponding to $w$ then the word in $\mcw_{n+1}/\mcq_{n+1}^{n+1}$  corresponding to $w$ is simply a string of $K_{n+1}$ $C$'s. Suppose that $\acts_{n+1}(C)=D$. When the action $\acts_{n+1}$ is extended to the skew-diagonal action at a later stage $m$, it simply takes this string of $C$'s to a string of $D$'s in a different place in a reverse word in the alphabet $\mcw_{n+1}/\mcq_{n+1}^{n+1}$. It is completely opaque whether the elements of $D$ have tails on the same side or the opposite side as the tails of words in $C$.

\item  Roughly speaking, Cases 1 and 2 above correspond to Cases 1 and 2
    in section 8.3 of \cite{FRW}, albeit with several differences. A key one  is that
    here, once the construction falls into Case 1, it remains in Case 1.  

\end{itemize}
We note that we have created  inductive lower bounds on the size of $k_n$.
\begin{description}
\item[Numerical Requirement C] \hyperlink{nr1}{$k_n$ is large enough that} $s_{n+1}\le s_n^{k_n}$.
\item[Numerical Requirement D] \hyperlink{nr2}{$k_n$ is large enough} to satisfy the use of the Substitution Lemma \ref{lem:sl} to construct  the words in $\mcw_{n+1}$.  In particular $1/k_n<\epsilon^3_n/4$.
\end{description}
The data for numerical requirement D comes from the coefficients and words and equivalence relations at stages $n-1$ and before.

\end{remark}

\section{Circular Systems and Diffeomorphisms of the Torus}
\label{sec:CT}

By Theorem \ref{red to odos}, we have a primitive recursive reduction
$F_\mco$ from G\"odel numbers of $\Pi^0_1$ sets to uniquely ergodic
odometer-based systems.  However the main theorem is about diffeomorphisms of the torus and it is an open problem whether there is any
smooth ergodic transformation of a compact manifold that has an odometer as a
factor.  
Rather than attack this problem directly, we follow \cite{part3} and do a second transformation
of odometer-based systems into \emph{circular systems}, which can be realized as
diffeomorphisms.
    This is the downward vertical arrow $\mcf$ on the right of figure~\ref{the square}.

Subsection~\ref{sec:CT:ss:circ} covers circular systems and their construction.
The primitive recursive map $\mathcal F$ maps from the
odometer-based systems to circular systems and preserves
\emph{synchronous} and \emph{antisynchronous} factors and conjugacies. In particular, for those
odometer-based systems $\bk$ in the range of $F_\mco$, $\bk$ is or is not
isomorphic to its inverse,  if and only if $\mathcal F(\bk)$ is or is not
isomorphic to its inverse.  We use the language of category theory to describe
the structure that is preserved and define the categorical  isomorphism.

In Subsection~\ref{sec:CT:ss:tor}, the circular systems produced are realized as
smooth diffeomorphisms of the torus. This is done in two steps: first, a given
circular system is realized as a discontinuous map of the torus; second, it is shown that how to smooth the toral map into a diffeomorphisms that is measure theoretically isomorphic  to the circular system.  

\subsection{Circular Systems}
\label{sec:CT:ss:circ}

Like odometer-based systems, circular systems are symbolic systems characterized
by  construction sequences \(\la\C W_n^c: n\in\B N\ra\) of a certain
form. The basic tool for constructing circular systems is the \emph{\(\C C\)-operator.}

\subsubsection{Preliminaries}\label{prelims}

Let \(k, l, q\in\B N\) be arbitrary integers greater than \(1\), and \(p\) be
coprime to \(q\). Let \(0\leq j_i < q\) indicate the unique integer such
that    \begin{equation}
        j_i\cdot p = i\pmod{q}.\label{ji}
    \end{equation}
We can rewrite $j_i$ as $p^{-1}i$ (mod $q$), and reserve the subscript notation for this use.
\begin{definition}[The \(\C C\)-Operator]
    Let \(\Sigma\) be a non-empty finite alphabet and let \(b\) and \(e\) be two new
    symbols not contained in \(\Sigma\). Let \(w_0, \ldots, w_{k-1}\) be
    words in \(\Sigma\cup\{b,e\}\). The \emph{\(\C
    C\)-operator} is given by:
        \begin{equation*}
            \C C(w_0, \ldots, w_{k-1}) = \prod_{i = 0}^{k-1}\prod_{j
            = 0}^{q-1} b^{q - j_i}\cdot w_j^{l- 1}\cdot e^{j_i},
        \end{equation*}
    where ``$\prod$'' indicates concatenation.
    \end{definition}

Fix  a sequence $\la k_n, l_n:n\in\nn\ra$ of positive integers with $k_n\ge 2$ and $l_n$ increasing and 
$\sum_n 1/l_n<\infty$. We follow Anosov-Katok (\cite{AK}) and define  auxiliary sequences of
    integers \(\la p_n: n\in\B N\ra\) and \(\la q_n: n\in\B N\ra\).  Set  \(q_0
    = 1\), \(p_0 = 0\).  Inductively define
        \begin{equation}
        \label{qns}
            q_{n+1} = k_n l_n q_n^2
        \end{equation}
    and
        \begin{equation}
        \label{pns}
            p_{n+1} = k_n l_n p_n q_n + 1.
        \end{equation}
           Note that \(p_n\) and \(q_n\) are coprime for $n\ge 1$.

  Let  $\alpha_n=p_n/q_n$. Then 
    \begin{equation}\label{alphans}
    \alpha_{n+1}=\alpha_n+1/q_{n+1}.
    \end{equation} 
    Since $q_n>l_n$ and
    $\sum_n1/l_n<\infty$, we have $\sum_n1/q_n<\infty$.  Thus the $\alpha_n$ converge to a Liouvillean
    irrational $\alpha\in [0,1)$:
        \begin{align}
            \alpha&=\lim_{n\to\infty} \alpha_n\notag \\
            	&=\sum_{n\ge 1}{1\over q_n}. \label{alphaform}
        \end{align}

\paragraph{Circular Construction sequences}
We first define the notion of a  circular construction sequence.
Fix a non-empty finite alphabet \(\Sigma\cup\{b, e\}\) as above as well as
    positive natural number sequences \(\la k_n: n\in\B N\ra\) and \(\la
    l_n: n\in\B N\ra\), with \(k_n\geq 2\) and \(\la l_n\ra\) strictly
    increasing such that \(\sum_{n=1}^\infty 1/l_n < \infty\).  We take $l_0=1$.

Let $\mcw_0^c=\Sigma$.  For every \(n\), choose a set
    \(P_{n+1}\subseteq(\C W_n^c)^{k_n}\) of \emph{prewords}. Then \(\C
    W^c_{n+1}\) is given by all words of the form
        \begin{equation}
            \label{COP}
            \C C(w_0, \ldots, w_{k_n - 1}) = \prod_{i =
            0}^{k_n-1}\prod_{j = 0}^{q_n-1} b^{q_n - j_i}\cdot
            w_j^{l_n- 1}\cdot e^{j_i}
        \end{equation}
    where \((w_0, \ldots, w_{k_n - 1})\in P_{n+1}\) is a preword. We call $\mcc$
    the \emph{$\mcc$-operator}.

\bigskip

The words created by the \(\C C\)-operator are necessarily uniquely readable.
However, we further demand that the collections of prewords \(\la P_n: n\in\B
N\ra\) are uniquely readable in the sense that each \(k_n\)-tuple of words
\(p\in P_{n+1}\),  considered a word in the alphabet \(\C W^c_n\), is uniquely
readable.  (Unique readability is discussed in  Appendix \ref{symbsys} in definition \ref{unique}.
See the discussion in \cite{part1} for more details. )

\begin{definition}[Circular system]
   Let $\la \mcw^c_n:n\in\nn\ra$ be a circular construction sequence.  Then the limit, which we denote $\bk^c$ is a 
   \emph{circular system}.
   \end{definition}
\noindent To emphasize that a given construction sequence is circular we denote it $\la \mcw_n^c:n\in\nn\ra$.   
   \smallskip

 In this paper the circular construction sequences will be \hyperlink{stun}{strongly uniform}. As a consequence the resulting symbolic shift is uniquely
ergodic and we can write $\bk^c=((\Sigma\cup\{b, e\})^{\B Z},\mcb, \mu, \sh)$
where $\mu$ is the unique shift invariant measure on $\bk^c$.

\begin{ex}\label{circ fact}
    Let $\Sigma=\{*\}$.  Then $|\mcw^c_0|=1$. Passing from from $\mcw^c_n$ to
    $\mcw_{n+1}^c$ one inductively shows that for all $n, |\mcw_n^c|=1$. Define
    $\mck_\alpha$ to be the limit of the resulting construction sequence. 

    Suppose that $\la \mcu^c_n:n\in\nn\ra$ is another circular construction sequence
    in an alphabet $\Lambda$  with the same coefficients $\la k_n, l_n:n\in\nn\ra$
    having a limit $\bl^c$.   Define a map $\pi:\bl^c\to \mck_\alpha$ by setting 
        \[
            \pi(f)(n) = 
                \begin{cases}
                    * & \mbox{if }f(n)\in \Lambda\\
                    b & \mbox{ if $f(n)=b$},\\
                    e & \mbox{ if $f(n)=e$}.
                \end{cases}
        \]
    Then $\pi$ is a factor map of symbolic systems.  Hence $\mck_\alpha$ is a
    factor of every circular system with coefficients $\la k_n, l_n:n\in\nn\ra$.
\end{ex}

\subsubsection{Rotation Factors}
\label{sec:rot_fact}

For $\alpha\in [0,1]$, let  $\mcr_\alpha:S^1\to S^1$ be rotation by $2\pi
\alpha$ radians. Equivalently we view \(\C R_{\alpha}:[0,1)\to [0,1)\)  as given
by \(x\mapsto x + \alpha \pmod{1}\).  This rotation \(\C R_\alpha\) plays the
same role with respect to circular systems as the canonical odometer factor
plays with respect to the odometer-based systems of Section~\ref{sec:Odom}. 

\begin{lemma}[The Rotation Factor]
\label{lrf}
    Let $\alpha=\lim \alpha_n$ be defined from a sequence $\la k_n, l_n:n\in\nn\ra$
    from equation \ref{alphaform}.  Then $\mck_\alpha\cong
    \mcr_\alpha$.  In particular if  \((\B K^c, \mcb,  \nu, \sh)\) is a circular
    system in the alphabet \(\Sigma\cup\{b, e\}\) with parameters  \(\la k_n, l_n:
    n\in\B N\ra\), then there is a canonical factor map \(\rho:\B K^c\to\C
    \mcr_\alpha\).
\end{lemma}

\begin{proof}[Proof sketch]
    For almost every \(x\in \bk_\alpha\), there is an $N$ for all \(n\geq N\) there are \(a_n,
    b_n\ge 0\) such that \(x\rest[-a_n,b_n)\) is some word in \(\C W^c_n\). All
    words in \(\C W_n^c\) have the same length, \(q_n\), so we can define the
    following quantity:
        \begin{equation*}
            \rho_n(x) = a_n\left(\frac{p_n}{q_n}\right).
        \end{equation*}

    Straightforward algebraic manipulations give that
        \begin{equation*}
            \left|\rho_{n+1}(x) - \rho_n(x) < \frac{2}{q_n}\right|
        \end{equation*}
    whence it is clear that \(\rho_n(x)\to\rho(x)\in[0, 1)\). Since
        \begin{equation*}
            \rho_n(\sh(x)) = \rho_n(x) + \frac{p_n}{q_n}
        \end{equation*}
    taking limits shows that \(\rho(\sh(x)) = \rho(x) + \alpha\), as desired.
\end{proof}

See Theorem 52 in \cite{part3} for a complete proof.

\paragraph{Distinguishing $\alpha$'s} Theorem \ref{thm:main} demands that if $M\ne N$, then $F(M)\not\cong F(N)$.  This is achieved by arranging that the Kronecker factors of $F(M)$ and $F(N)$ are non-isomorphic rotations of the circle. This requires that $\alpha(N)\ne \alpha(M)$ and that $\mck_{\alpha(N)}$ is the Kronecker factor of the limit sequence $\bk^c(N)$.
Recall  that for each $N$ we have a prime number  $P_N$ which we take for $k_0$ and and we  build sequences $\la k_n(N), l_n(N):n\in\nn\ra$, which in turn, yield sequences $\la p_n, q_n, k_n, l_n:n\in\nn\ra(N)$ and $\la \alpha_n(N):n\in\nn\ra$ which converge to an irrational $\alpha(N)$. 

For each $N$ we 
take $l_0(N)=1$, so $\alpha_1(N) = \tfrac 1{P_N}$. The sequence $\la k_n(N):n\in\nn\ra$ is defined in the construction of the odometer construction sequences as described after Lemma \ref{lem:sl}. The $l_n$'s are chosen in the construction of the circular sequences and diffeomorphisms.  They must satisfy some lower bounds on their growth, which we describe later.

To ensure different rotation factors correspond to different \(\Pi^0_1\)
sentences, we also put the following growth requirement on the $\la l_n(N):n\in\nn\ra$ sequences:
\begin{description}
    \item[Numerical Requirement E]\hypertarget{nre}{Growth Requirement} on the $l_n$'s:
    \[l_n(N)\ge l_n(N-1).\]
    \end{description}

\begin{lemma}\label{remember the alphas}
Suppose that the $k_n(N-1), k_n(N), l_n(N-1)$ and $l_n(N)$ satisfy \hyperlink{nrb}{Requirements B} and \hyperlink{nre}{E}.  Then $\alpha(N-1)>\alpha(N)$.
\end{lemma}
	  
	  \pf Note that $k_0(N-1)=P_{N-1}<P_N=k_0(N)$, so $q_1(N-1)<q_1(N)$ Since $q_{n+1}=k_nl_nq_n^2$, $k_n(N)\ge k_n(N-1)$ and $l_n(N)\ge l_n(N-1)$ one sees inductively that for all $n$ $q_n(N-1)\le q_n(N)$. 
	  
	   By equation \ref{alphaform} we see that 
	
	   \begin{align}
	   \alpha(N-1)&=\sum_{n\ge 1}{1\over q_n(N-1)}\notag\\
	   &>\sum_{n\ge 1}{1\over q_n(N)}=\alpha(N).\notag
	   \end{align}
	   
	   \qed

\paragraph{Synchronous and Anti-synchronous joinings}The system $\mck_\alpha$
gives a symbolic representation of the rotation $\mcr_\alpha$ by $2\pi \alpha$
radians.  The inverse transform $\rev{\mck_\alpha}$ is therefore a
representation of rotation by $2\pi(1-\alpha)\equiv 2\pi(-\alpha)$ radians.
Moreover the conjugacies $\phi:S^1\to S^1$ between $\mcr_\alpha$ and
$\mcr_\alpha^{-1}=\mcr_{-\alpha}$ are of the form $z\mapsto \bar{z}*e^{2\pi i
\delta}$ for some $\delta$.  For combinatorial reasons we fix a particular
conjugacy $\natural: \mck_\alpha\to \rev{\mck_\alpha}$ that is described explicitly in
\cite{part2}. Thus
 $\natural$ is given by  the map defined on $S^1$ by $z\mapsto \bar{z}*e^{2\pi i \gamma}$ for
some particular $\gamma$.  In additive notation on $[0,1)$ this becomes $x\mapsto -x+\delta\  (\mod 1)$ for some $\delta$.

The importance of rotation factors and odometer factors in the sequel is their
function as ``timing mechanisms.'' Joinings between odometer-based systems
induce joinings on the underlying odometers; the same holds true of circular
systems.

\begin{definition}[Synchronous and Anti-synchronous Joinings]
    We define two kinds of joinings, \emph{synchronous} and
    \emph{anti-synchronous}.
        \begin{itemize}
            \item Let \(\B K_1\) and \(\B K_2\) be odometer-based systems
                sharing the same parameter sequence \(\la k_n: n\in\B
                N\ra\).  Let \(\eta\) be a joining between \(\B K_1\) and \(\B
                K_2\).  Then \(\eta\) induces a joining \(\eta_\pi\) between
                \(\B K_1\) and \(\B K_2\)'s copies of the underlying odometer
                \(\C O\).  The joining \(\eta\) is \emph{synchronous} if
                \(\eta_\pi\) is the graph joining corresponding to the identity map
                from \(\C O\) to \(\C O\). A joining \(\eta\) between $\bk_1$
                and $\bk_2$ is \emph{anti-synchronous} if \(\eta_\pi\) is the
                graph joining corresponding to the map \(x\mapsto -x\) from
                $\mco$ to $\mco^{-1}$.
            \item Let \(\B K^c_1\) and \(\B K^c_2\) be  circular systems sharing
                the same parameter sequence \(\la k_n, l_n: n\in\B N\ra\).  Let
                \(\eta\) be a joining between \(\B K^c_1\) and \(\B K^c_2\).
                Then \(\eta\) induces a joining \(\eta_\pi\) between \(\B K^c_1\)
                and \(\B K^c_2\)'s copies of the rotation factor, \(\C K_\alpha\).
                The joining \(\eta\) is \emph{synchronous} if \(\eta_\pi\) is
                the graph joining corresponding to the identity on
                $\mck_\alpha\times \mck_\alpha$. A joining \(\eta\) between
                $\bk^c_1$ and $(\bk^c_2)^{-1}$ is \emph{antisynchronous} if
                \(\eta_\pi\) restricts to the graph joining corresponding to
                \(\natural:\mck_\alpha\to (\mck_\alpha)^{-1}.\)
        \end{itemize}
\end{definition}

\subsubsection{Global Structure Theorem}
\label{sec:gst}

Odometer-based systems and Circular systems that share the same parameter
sequence $\la k_n:n\in\nn\ra$ have  similar joining structures.  We begin by
defining two categories.

Fix a parameter sequences $\la k_n:n\in\nn\ra$ and $\la l_n:n\in\nn\ra$ with
$\sum 1/l_n<\infty$. Let $\mco B$ be the category whose objects consist of all
ergodic odometer-based systems with coefficients $\la k_n:n\in\nn\ra$. A
morphism of $\mco B$  is either a synchronous graph joining between $\bk$ and
$\bl$ or an anti-synchronous graph joining between $\bk$ and
$\bl^{-1}$. Let $\mcc B$ be the category whose objects consist of ergodic
circular systems built with coefficients  $\la k_n, l_n:n\in\nn\ra$ and whose
morphisms consist of synchronous and anti-synchronous graph joinings from
$\bk^c$ with $(\bl^c)^{\pm 1}$.

The main result of \cite{part2} is the following:

\begin{theorem}[Global Structure Theorem]
\label{gst}
    The categories $\mco B$ and $\mcc B$ are isomorphic by a functor $\mcf$ that
    takes synchronous joinings to synchronous joinings, anti-synchronous
    joinings to anti-synchronous joinings and isomorphisms to isomorphisms.
\end{theorem}

To prove Theorem~\ref{gst} one must define the map $\mcf$ on objects, and on morphisms
and then show that it is a bijection and preserves composition. Since we will
only be concerned here with how effective $\mcf$ is we confine ourselves to
defining it and refer the reader to \cite{part2} for complete proofs.  In
\cite{hans}, the proof is discussed to understand the strength of the assumptions  needed to prove it. 

We begin by defining $\mcf$ on the objects.

\paragraph{Defining $\mcf$ on objects.}
    Let an $\bk$ be an odometer-based system with associated construction and
    parameter sequences \(\la\C W_n: n\in\B N\ra\) and \(\la k_n: n\in\B N\ra\).
    Let   \(\la l_n:n\in\B N\ra\) be an arbitrary sequence of positive integers
    growing fast enough that $\sum_n 1/l_n<\infty$. Inductively define a
    map
     \(\C F\) taking the construction sequence for an odometer-based system \(\B K\) to a construction sequence for a circular system
    \(\B K^c\) by  applying the \(\C C\)-operator.   Define
maps \(c_n:\C W_n\to\C W_n^c\) as
    follows:
        \begin{itemize}
            \item Let \(\C W^c_0=\Sigma\) and \(c_0\) be the identity.
            \item Suppose that \(c_n\) and \(\C W_n^c\) have been defined. Let
                    \begin{equation*}
                        \C W^c_{n+1} = \{\C C(c_n(w_0), \ldots,
                        c_n(w_{k_n-1})) : w_0w_1\cdots 
                        w_{k_n-1}\in\C W_{n+1}\}
                    \end{equation*}
                and $w_i\in \mcw_{n}$.  Define $c_{n+1}:\mcw_{n+1} \to
                \mcw_{n+1}^c$ by setting 
                    \begin{equation*}
                        c_{n+1}(w_0\cdots w_{k_n-1}) = \C
                        C(c_n(w_0),\ \ldots,\ c_n(w_{k_n-1})).
                    \end{equation*}
                where  \(w_i\in \C W_n\) with \(w_0\cdots w_{k_n-1}\in\C
                W_{n+1}\).
        \end{itemize}
 The construction sequence $\la \mcw_n^c:n\in\nn\ra$ then gives rise to a circular system $\bk^c$.  The functor $\mcf$ will associate $\bk^c$ with $\bk$.                       
 \paragraph{Lifting measures and joinings}   
    
    We need to lift measures on odometer based systems to measures on circular systems  for two reasons:
    \begin{enumerate}
    \item To complete the definition of $\mcf$ on objects, given an odometer based system $(\bk, \mu)$ we need to canonically associate a measure $\mu^c$ to  $\bk^c$. Then $\mcf(\bk,\mu)=(\bk^c,\mu^c)$. 
    
     In the context of this paper this first reason is not pressing:  the construction sequences in the range of $F_\mco$  are \hyperlink{stun}{strongly uniform}, hence uniquely ergodic. Thus there is only one candidate for $\mu^c$. However to complete the definition of $F$ we need to understand what happens for arbitrary ergodic $\mu$.
     
    \item To define $\mcf$ on morphisms, given a joining $\mcj$ between $(\bk,\mu)$ and $(\bl,\nu)$ we need to associate a joining  $\mcj^c$ between $(\bk^c, \mu^c)$ with $(\bl^c,\nu^c)$.
    
    For the second issue, and to deal with general odometer based systems $(\bk, \mu)$, we review the notion of \emph{generic sequences} of words.  These were introduced in \cite{Benjy} and used in the proof of Theorem \ref{gst}  \cite{part2}. 
    
     \end{enumerate}

       \smallskip 
    
    Let $k,l>0$ and $\la \mcw_n:n\in\nn\ra$ be an arbitrary construction sequence.
    Using the unique readability of words in $\mcw_k$ a word $w$ in $\Sigma^{q_{k+l}}$ determines a unique sequence of words $w_j$ in $\mcw_k$ such that , 
\[w=u_0w_0u_1w_1\dots w_Ju_{J+1}.\] 
When $w\in \mcw_{k+l}$,  each $u_j$ is in the region of spacers added  in $\mcw_{k+l'}$, for $l'\le l$.
We will denote the \hypertarget{emptiest}{\emph{empirical distribution}} of $\mcw_k$-words in $w$ by EmpDist$_k(w)$. Formally:
\[\mbox{EmpDist}_k(w)(w')={|\{0\le j\le J: w_j=w'\}|\over J+1}, \ w'\in \mcw_k.\]
Then $EmpDist$ extends to a measure on $\mathcal P(\mcw_k)$ in the obvious way.

To finitize the idea of a generic point for a system  $(\bk,\mu)$ we introduce the notion of a generic sequence of words.  By $\mu_m$ we will denote the discrete measure on the finite 
 set $\Sigma^m$ given by $\mu_m(u)=\mu(\la u\ra)$.  Then $\mu_m$ is not a probability measure so we normalize it. Let  $\hat{\mu}_n(w)$ denote the discrete probability measure on $\mcw_n$ defined by
\[\hat{\mu}_n(w)={\mu_{q_n}(\la w\ra)\over \sum_{w'\in\mcw_n} \mu_{q_n}(\la w'\ra)}.
\]

Thus $\hat{\mu}_n(w)$ is the relative measure of $\la w\ra$ among all $\la w'\ra, w'\in \mcw_n$. The denominator is a normalizing constant to account for spacers at stages $m>n$ and for shifts of size less than $q_n$.

\begin{definition}\label{ED} A sequence $\la v_n\in\mcw_n:n\in\nn\ra$ is a \hypertarget{gen seq}{\emph{generic sequence of words}} if and only if
for all $k$ and $\epsilon>0$ there is an $N$ for all $m,n>N$,
\[\|EmpDist_k(v_m)-EmpDist_k(v_n)\|_{var}<\epsilon.\]
The sequence is generic for a measure $\mu$ if for all $k$:
\[\lim_{n\to \infty}\| \mbox{EmpDist}_k(v_n)-\hat{\mu}_k\|_{var}=0
\]
where $\|\ \|_{var}$ is the variation norm on probability distributions.
\end{definition}
The point here is that the ergodic theorem gives infinite generic sequences for measures $\mu$. These infinite generic sequences in turn, create generic sequences of finite words.  A generic sequence of finite words determines a measure. If the generic sequence is built from the measure then the measure it determines is the original measure
    \medskip

   We now deal with the first issue above for arbitrary $(\bk,\nu)$ (and not just those  that are strongly uniform).    Given an odometer based system 
   $(\bk, \nu)$ we must specify the measure $\nu^c$ we associate with  $\nu$.
      Section 2.6 of \cite{part2} gives a canonical method of
    constructing a \emph{generic} sequence of words $\la v_n:n\in\nn\ra$ that
    encode any ergodic measure on $\bk$.  The corresponding sequence of words
    $v^c_n=c_n(v_n)$ is also generic and  determines an ergodic  measure on
    $\bk^c$. The map $\mcf$ then takes $(\bk, \nu)$ to $(\bk^c, \nu^c)$    

\paragraph{Defining $\mcf$ on morphisms}   Given an arbitrary synchronous or anti-synchronous joining $\mcj$ between odometer based systems 
$\bk$ and $\bl^{\pm1}$ we can view $(\bk\times \bl, \mcj)$ as an odometer based system. Taking a generic sequence of pairs of words 
$\la (u_n, v_n):n\in\nn\ra$ for $\mcj$ as in \cite{part2} and lifting it with the sequence of $c_n$'s (and adjusting appropriately for reversing 
the circular operation with a mechanism denoted $\natural$ in \cite{part2}), one gets a joining $\mcj^c$ between $\bk^c$ and $\bl^c$.

Define $\mcf(\mcj)=\mcj^c$.

\paragraph{Is $\mcf$ primitive recursive?} Clearly the maps $c_n$ are primitive recursive so the map taking a construction sequence $\la \mcw_n:n\in\nn\ra$ to $\la \mcw_n^c:n\in\nn\ra$ is  primitive recursive.    For the same reason the map taking a joining $\mcj$ specified by a given generic sequence to $\mcj^c$ is primitive recursive. 
Thus, assuming that joinings $\mcj$ are presented in a manner that one can compute the generic sequences of words, the map $\mcj\mapsto \mcj^c$ is primitive recursive. 

In the context of the systems in the range of $F_\mco$, the relevant joinings between 
$\bk$ and $\bk^{-1}$ are given by limits of $\eta_n$'s, and the generic word sequences are easily seen to be primitive recursive and can thus be translated to the joinings of 
$\bk^c$ with $(\bk^c)^{-1}$.
   
\begin{remark} We have shown that if $\phi_N$ is true then $\mcf\circ F_\mco(N)$ is isomorphic to $\mcf\circ F_\mco(N)^{-1}$ and the isomorphism is primitive recursive. In section \ref{sec:CT:ss:tor} we build a primitive recursive realization function $R$ which maps from strongly uniform circular systems to measure preserving diffeomorphisms of the torus. Since $F=R\circ\mcf\circ F_\mco$, the result we prove is
something stronger than claimed in Theorem \ref{thm:main}. Namely we show that if $\phi_N$ is true then there is a measure isomorphism between $F(N)$ and $F(N)^{-1}$ coded by a primitive recursive generic sequence of words.
\end{remark}
   
\subsection{The Kronecker Factors}\label{K facts}
The second clause of the Main Theorem (Theorem \ref{thm:main:eqn}) says that if $M$ and $N$ are distinct natural numbers than the corresponding diffeomorphisms $T_M$ and $T_N$ are not isomorphic. To distinguish between them we use their Kronecker factors. (For more information on the Kronecker factors, see e.g. \cite{walters}.) For this purpose we prove the following proposition. This section is otherwise independent of the other sections. Readers who find the proposition and corollary obvious can skip to the next section.

 \begin{proposition}
\label{K-factor of circ}
 Let $\bk^c$ be circular system in the range of $\mcf\circ F_\mco$, built with coefficients $\la k_n, l_n:n\in\nn\ra$ and 
 $\alpha=\lim_n\alpha_n$ then the Kronecker factor of $\bk^c$ is measure theoretically isomorphic to the rotation 
 $\mcr_\alpha$.
 
 \end{proposition}
 An immediate corollary of this is:\footnote{See section \ref{elements} for an explanation of the $(N)$-notation.}
 
 \begin{corollary}
  \label{one one}
  Suppose that $M<N$ are natural numbers. Then:
  \begin{enumerate}
  \item $\alpha(N)<\alpha(M)$, where $\alpha(N)$ and $\alpha(M)$ are the irrationals associated with the rotation factors of $F(N)$ and $F(M)$.
  \item $(\bk^c)^M\not\cong(\bk^c)^N$.
  \end{enumerate}
  
  \end{corollary}

\pf This follows immediately from Lemma \ref{remember the alphas} and the fact that the Kronecker factor $(\bk^c)^M$ is isomorphic to  $\mcr_{\alpha(M)}$ and the Kronecker factor of $(\bk^c)^N$ is isomorphic to $\mcr_{\alpha(N)}$.
\qed

After Proposition \ref{K-factor of circ} is shown we will have proved the following intermediate step in the proof of Theorem \ref{thm:main}:

\begin{prop}
\label{victory for circs}
    For $N$ a code of a $\Pi^0_1$ sentence, then $\mcf\circ F_\mco(N)$ is a
    primitive recursive circular construction sequence and
        \begin{enumerate}
            \item \(N\) is the code for a true statement if and only if the
                circular system $T$ determined by \(\mcf\circ F_\mco(N)\)  is
                measure theoretically conjugate to \(T^{-1}\);
            \item $\mcf\circ F_\mco(N)$ is  ergodic--in fact strongly uniform;
                and
            \item For \(M\ne N\), \(\mcf\circ F_\mco(M)\) is not conjugate to
                \(\mcf\circ F_\mco(N)\).
        \end{enumerate}
\end{prop}

 \paragraph{Review of the Kronecker factor} Let $\vec{\gamma}=\la \gamma_m:m\in\poZ\ra$ be an enumeration of  the eigenvalues of the Koopman operator of a measure preserving transformation $(X,\mcb,\mu,T)$. Then  
 $\vec{\gamma}$ determines a measure preserving action on 
 $((S^1)^\poZ, \lambda^\poZ)$ (where $\lambda^\poZ$ is the product measure on $(S^1)^\poZ$) by coordinatewise multiplication.  The action is ergodic, discrete spectrum and isomorphic to the Kronecker factor of $(X,\mcb,\mu, T)$. 

If $\alpha$ is an eigenvalue of the shift operator then the powers of $\alpha$,
$\vec{\alpha}=\la \alpha^n:n\in\poZ\ra$ are also eigenvalues corresponding to a 
subsequence of $\vec{\gamma}$ and hence the coordinatewise multiplication of 
$\vec{\alpha}$ on $(S^1)^\poZ$  determines a factor of the Kronecker factor.  This is 
a proper factor if and only if there is an eigenvalue of the Koopman operator that is 
not a power of $\alpha$. In particular there is a non-trivial projection map from the 
 Kronecker factor to the dual of the countable group $\{\alpha^n:n\in\poZ\}$ 
 \medskip
   
  The proof of Proposition \ref{K-factor of circ} follows the outline of the proof of Corollary 33 of \cite{FRW}. Working in the context of odometer based 
  systems built with coefficients $\la k_n:n\in\nn\ra$,  it says that the Kronecker factor $\kron$ of  each system $\bk$ in the range of 
  $F_\mco$ is the odometer transformation $\mco$ based on $\la k_n:n\in\nn\ra$. Note that the odometer $\mco$ is a subgroup of the Kronecker factor since rotation by the $k_n^{th}$ root of unity is an eigenvalue of the Koopman operator. The steps there are:
  
  \begin{enumerate}
  \item Any joining $\mcj$ of $\bk$ with $\bk$ 
  projects to a joining $\mcj_\mco$ of 
  $\mco$ with itself. If $\mcj_\mco$ is not given by the graph joining coming from a finite shift of the odometer then $\mcj$ must be the relatively independent joining of $\bk$ with itself over $\mcj_\mco$. (This is Proposition 32 of \cite{FRW}.)
  \item If there is an eigenvalue of the unitary operator associated with $\bk$ that is not a power of $\alpha$ there is a non-identity element $t$ in the Kronecker factor whose projection to the odometer $\mco$ is the identity. 
  \item Multiplying  $t$ by an element $h\in\mco$ which is not a finite shift gives an element $t'$ of the Kronecker factor 
  $\kron$ that is not in $\mco$ and projects to an element of $\mco$ that is not a finite shift.
  \item Let $\mch^*$ be the sub-$\sigma$-algebra of the measurable subsets of $\bk$ generated by 
  $\kron$.  Then  multiplication by $t'$ gives a graph joining $\mcj^*$ of $\mch^*$ with itself that projects to the joining of $\mco$ given by multiplication by $h$. Extend $\mcj^*$ to a joining $\mcj$ of $\bk$ with $\bk$.  Then $\mcj$ does not project to a finite shift of the odometer but it is also not the relatively independent joining of $\bk$ with itself over the joining of $\mco$ with itself given by $h$.  This is a contradiction.
  \end{enumerate}
  
  To imitate this argument we first note that for circular systems, the analogue of the odometer is the rotation $\mck_\alpha$, and that every element $\beta\in S^1$ determines an invertible graph joining $\mcs_\beta$ of $\mck_\alpha$ with itself, corresponding to multiplication by $\beta$ in the group $S^1$. We need to identify the analogue of the ``finite shifts on the odometer" in the case of circular systems.  The appropriate notion is given in Definition 78 in \cite{part3}, namely the \emph{central values}. The central values form a subgroup of the unit circle. 
  
To prove Proposition \ref{K-factor of circ}, fix a circular system $\bk^c$ in the range of 
$\mcf\circ F_\mco$.
We first show that there is a $\beta\in S^1$ that is \emph{not} a central value.  This 
$\beta$ plays the roll of $h$ in the outline given above. Then the analogue of Proposition 32 is proved: any joining of 
$\bk^c$ with itself that does not project to the joining given by  multiplication on $S^1$  by a central value is the 
relatively independent joining over its projection. 

Suppose now that there is an  eigenvalue of the Koopman operator that is not a power of $\alpha$. Then the action of $\vec{\alpha}$ on 
$(S^1)^\poZ$ is a non-trivial projection of the Kronecker factor $\kron^c$ of 
$\bk^c$.  Hence we can fix a non-identity element $t$ of the Kronecker factor whose projection to the factor determined by the powers of $\alpha$ is 
the identity. As in step 3 above we multiply $t$ by a non-central $\beta$ to get a $t'$ in the Kronecker factor which:
	\begin{enumerate}[a.)]
	\item induces a joining $\mcj^*$ of $\mch^*$ with itself that projects to the graph 
	joining of $\mck_\alpha$ with itself  induced by $\mcs_\beta$. 
	
	Extending $\mcj^*$ to a joining $\mcj$ of $\bk^c$ with itself we see that:
	\item $\mcj$ is not the relatively independent joining over the joining of 
	$\mck_\alpha$ given by $\mcs_\beta$. 
	\end{enumerate}
  After the details are filled in, this contradiction establishes Proposition \ref{K-factor of circ}.
  
  \bigskip
 \bfni{Notation}  As in  previous sections we identify the unit interval $[0,1)$ with the unit circle via the map $x\mapsto e^{2\pi i *x}$, which identifies  ``addition mod one" on the unit interval with multiplication on the unit circle. When we write 
 ``$+$" in this section it means addition mod one, interpreted in this manner.
\medskip
  
  We use the following numerical requirement in explicit proof of Proposition \ref{K-factor of circ}:
  
   \begin{description}
 \item[Numerical Requirement F] \hypertarget{nr4}{The} $k_n$'s must grow fast enough that $\sum {6^n\over k_n}<\infty$.
\end{description}
  To finish the proof of Proposition \ref{K-factor of circ}, we  fix a circular system $\bk^c$ in the range of $\mcf\circ F_\mco$  and prove the following two lemmas.

 \begin{lemma}\label{not central}
 There is a non-central value $\beta$.
 \end{lemma}

 \begin{lemma}\label{yeah yeah yeah}
 Suppose that $\beta$ is not a central value.  Let $\mcj$ be a joining of $\bk^c\times \bk^c$ whose projection to $\mck_\alpha\times \mck_\alpha$ is the 
 graph joining of $\mck_\alpha$ with itself given by multiplication by $\mcs_\beta$.  Then $\mcj$ is the relatively independent joining of $\bk^c\times \bk^c$ over the joining of $\mck_\alpha$ with itself given by  $\mcs_\beta$.
 \end{lemma}
 
 \pf\!\![Lemma \ref{not central}] While it seems very likely that there is a measure one set of examples we just need one.  The example will be of the form $\beta=\sum_{n=1}^\infty {a_n\over k_nq_n}$ for an inductively chosen sequence of natural numbers $\la a_n:n\in\nn\ra$ with $0\le a_n< 6^n$.

 To describe $\beta$ completely and verify it is non-central we need several facts 
 from  sections 5, 6 and 7 in \cite{part3}, which discuss the relationship between the geometric and the symbolic representations of 
 $\mck_\alpha$.

 The geometric construction builds a sequence of periodic approximations of lengths 
 $\la q_n:n\in\nn\ra$ with the resulting limit being  the rotation of the circle by 
 $\mcr_\alpha$. Expanding on Lemma \ref{lrf}, these approximations are given by the towers of intervals $\mct_n=\{[0,1/q_n), [p_n/q_n, p_n/q_n+1/q_n), [2p_n/q_n, 2p_n/q_n+1/q_n), \dots [kp_n/q_n, kp_n/q_n+1/q_n), \dots\}$ viewed as a periodic system. 

  The symbolic representation uses the 
 $\mcc$ operation to build the construction sequence for the symbolic system.  The 
 latter is described in Example \ref{circ fact}. We give the geometric description of the periodic approximations presently.

\smallskip
 
 We use the following notions and notation from  \cite{part3}:
 \begin{enumerate} 
 \item
 $\phi_0:(\mck_\alpha, sh)\to ([0,1), +)$ 
is the measure theoretic isomorphism between the shift on $\mck_\alpha$ and the 
rotation $\mcr_\alpha$ given in Lemma \ref{lrf}. We use $s$'s to refer to elements of $\mck_\alpha$ and 
$x$'s to refer to elements of $[0,1)$ and $s$ \emph{corresponds} to $x$ if 
$\phi_0(s)=x$.
\item The notion of $s\in \mck_\alpha$ being \emph{mature} implies that $s$  has a principal $n$-subword and it is repeated multiple times both before and after $s(0)$.

\item $\mcs_\beta=\phi^{-1}_0\mcr_\beta\phi_0$ is the symbolic conjugate of the rotation $\mcr_\beta$, via the map $\phi_0$.   If $s$ corresponds to $x$ then $\mcs_\beta(s)$ corresponds to $x+\beta$ (mod 1). We will occasionally be sloppy and use the language $s+\beta$ for $s\in \mck_\alpha$ when we mean $\mcs_\beta(s)$.
\item In \cite{part3}, a set $S$ is defined as  the collection of elements  $s\in\mck_\alpha$ such that the 
left and right endpoints of the principal $n$-subwords of $s$ go to minus and plus infinity respectively. Explicitly suppose  that $s\in \mck_\alpha$ is  such that for all large enough $n$, the principal $n$-subword exists and lives on an interval $[-a_n, b_n]\subseteq \poZ$.  The point $s\in S$ if  $\lim_n a_n=\lim_n b_n=\infty$.

The set 
$S_\beta$ is $\bigcap_{n\in\poZ}\mcs_\beta(S)$, the maximal $\mcs_\beta$ invariant subset of $S$. It is of measure 
one for the unique invariant measure on $\mck_\alpha$.  Since $\mck_\alpha$ is a factor of every circular system 
$\bk^c$ with the same coefficient sequence $\la k_n, l_n:n\in\nn\ra$ for all invariant measures $\mu$ on $\bk^c$, 
$\{s\in \bk^c:$ the left and right endpoints of the principal $n$-subwords of $s$ go to minus and plus infinity is of $\mu$-measure one.
\item Given an arbitrary $\beta$ we can intersect $S_\beta$ with $S_q$ for all rational $q$ and get another set of measure one.  Hence in a slight abuse of notation we assume that $S_\beta$ is invariant under conjugation by  rational rotations.
\item For $s\in \mathcal K_\alpha$, if  $r_n(s)=i$ and $x$ is the corresponding element of $[0,1)$ then $x$ is in the $i^{th}$ level of the tower corresponding to the $n^{th}$ approximation to $\mcr_\alpha$.  This tower is given by $\mcr_{\alpha_n}$.
 \end{enumerate}

 In the geometric picture, at stage $n$, we have a tower of intervals of the form 
 $[{i\alpha_n}, {i\alpha_n+1/q_n})$ ordered in the dynamical ordering--where the successor of the interval
 $[{i\alpha_n}, {i\alpha_n+1/q_n})$ is 
 $[{(i+1)\alpha_n}, {(i+1)\alpha_n+1/q_n})$.
     Thus level $I_{i+1}$ is $I_{i}+\alpha_n$.

 Passing from stage $n$ to stage $n+1$ involves subdividing the old levels into  new levels, which are of the form $[i\alpha_{n+1}, i\alpha_{n+1}+1/q_{n+1})$. These subintervals move diagonally up and to the right through the 
 $n$-levels. The diagonal movement corresponds to addition of 
 $\alpha_{n+1}$.  The key formula is \ref{alphans}:
 \begin{equation*} \alpha_{n+1}=\alpha_n+1/q_{n+1}.
 \end{equation*}
 
 As illustrated in diagrams 9 and 10 of \cite{part1} and Figure \ref{nvn2}, the $n+1$ tower proceeds 
 diagonally up through the $n$-tower. This is evident from the form of equation 
 \ref{alphans} and the fact that $q_{n+1}=k_nl_nq_n^2$.

 \newgeometry{margin=1.5in}    
     \begin{sidewaysfigure}
    \includegraphics[height=.85\textheight]{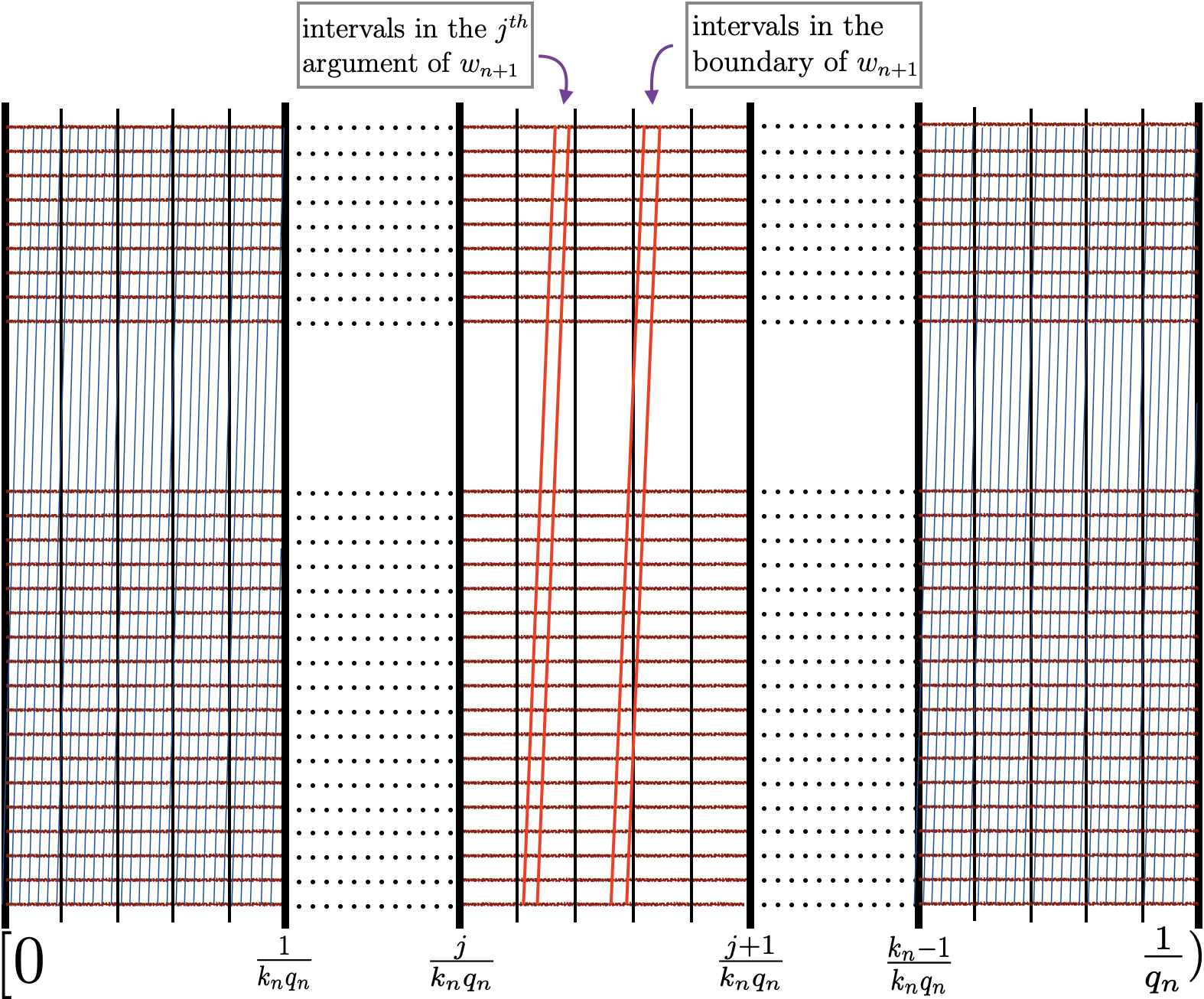}
    \caption{The $n$-tower, showing the diagonal progression of the $n+1$-tower. The heavy horizontal lines are the levels 
    $(i\alpha_n, i\alpha_n+1/q_n)$, starting with $[0, 1/q_n)$ on the bottom. The levels the $n+1$-tower are the horizontal segments between the diagonal lines.}
    \label{nvn2}
    \end{sidewaysfigure}
\restoregeometry     
\newpage
 Again, following \cite{part1}, the geometric picture in Figure \ref{nvn2} corresponds to  the symbolic 
 representation as a circular system in the following way.  Some of the diagonal paths 
 hit the left or right vertical strips bounding the $[j/k_nq_n, (j+1)/k_nq_n)$ subdivisions of the levels.  
 Those diagonals correspond to the  boundary portion of the $n+1$-words 
 (the $b$'s and the $e$'s). The diagonal paths that start in the region 
 $[j/k_nq_n, (j+1)/k_nq_n)$ and  traverse from the bottom to the top level while staying in that region correspond 
 to the $j^{th}$ argument of the $\mcc$ operator at stage $n$.
 
 Restating this in terms of the isomorphism $\phi_0$ between $\mck_\alpha$ and $([0,1),\mcr_\alpha)$, 
 if $s_1$ and $s_2$ are mature elements of $\mck_\alpha$ corresponding to $x_1, x_2\in S^1$, then:
 \begin{itemize}
 \item $x_1, x_2$ belong to two diagonal strips that do not touch a vertical strip and with base in the same interval 
 $[j/k_nq_n, (j+1)/k_nq_n)$
 
 iff
 
 \item Inside their principal $n+1$-subwords, $s_1(0)$ and $s_2(0)$ are in $n$-words coming from the same argument 
 $w^\alpha_j$ of $\mcc(w^\alpha_0, w^\alpha_1, \dots w^\alpha_{k_n-1})$.
 \end{itemize}
 
  We now continue the enumeration of basic notions in \cite{part3} we use here.
 \begin{enumerate}

 \item[7.] In a very slight variation of the notation of \cite{part3}, when we are 
comparing $s$ with $t$ we define $d^n(s,t)=r_n(t)-r_n(s)$ (mod $q_n$). 
In this argument frequently $t=\mcs_\beta(s)$ and if $\beta$ is clear from the context 
we simply write  $d^n(s)$.  If $x$ and $y$ correspond to $s$ and $t$,
the number $d^n(s,t)$ can be viewed  either as the number of levels in the $n$-tower between $x$ and $y$ or as the difference between the locations of $0$ in the principal $n$-subwords of $s$ and $t$.

\item[8.] For mature $s$ and $t$, the result of shifting $t$ by $-d^n(s,t)$ units is that the location of $0$ is in the same position in its principal $n$-subword as is the position of $s(0)$ in its principal $n$-subword. 

\item[9.] 
Applying the shift map $d^{n+1}(s,t)$ times to $s$  moves its zero to the same point as $t$'s is relative to its $n+1$ subword. Subsequently moving it back 
$-d^n(s,t)$ steps moves the zero of result back to the same position in its $n$-subword as zero is in $s$'s $n$-subword. In other words, if $s'$ is the result of applying the shift map to $s$ $d^{n+1}(s,t)-d^n(s,t)$ times, then    $0$ is  in the same position relative to the $n$-block of $s'$ as it is in $s$. 

\item[10.] The $n+1$-word in the construction sequence for $\mck_\alpha$ is of the form
\[\mcc(w^\alpha_0,w^\alpha_1\dots w^\alpha_{k_n-1})\]
and hence if $s$ is mature at stage $n$ then $s(0)$ occurs in an $n$-block 
corresponding to the position of $w^\alpha_{j_0}$ for some $j_0$. We can ask 
whether the $j_0$ corresponding to the principal $n$-subword of the 
$(d^{n+1}(s,t)-d^n(s,t))$-shift of $s$ is the same as the $j_0$ corresponding to the principal $n$-subword of $s$. 

If it does, then $s$ and $t$ are \emph{well-matched} at stage $n$ and if not the $s$ and $t$ are \emph{ill-matched} at stage $n$.

\item[11.]  If $s'$ is the result of shifting $s$ $d^{n+1}(s,t)-d^n(s,t)$ times then the $0$ of $s'$ is in the same $w^\alpha_j$ as is the zero  of $t$. So for the purposes of determining whether $s$ is  well-$\beta$-matched at stage $n$, we can compare which argument of $\mcc$ $s(0)$ and $t(0)$ belong to. As a result we can speak of $s$ and $t$ well or ill-matched at stage $n$.  If $x$ and $y$ are the corresponding members of $[0,1)$ we can say say that $x$ and $y$ are well or ill-matched at stage $n$.
 \end{enumerate}
    We are now ready to construct the non-central $\beta$. We do this by induction. At stage 1, $a_1=0$. At stage $n$, we let $\beta_n=\sum_{p=1}^{n-1}{a_p\over k_pq_p}$.  For $i=0, \dots 6^n-1$, in the terminology of item 11 consider 
\[M_i=\{x: x\mbox{ and }x+\beta_n+i/k_nq_n\mbox{ are well-matched}\}.\]
Since the $M_i$'s are disjoint,  for some $i$, $\lambda(M_i)\le{1\over 6^n}$. Let $a_{n}$ be such an $i$ and let $\beta_{n+1}=\beta_{n}+a_n/k_nq_n$. Finally we let $\beta=\sum_1^\infty {a_p\over k_pq_p}$, so $\beta=\lim_{n\to \infty}\beta_n$.

To see this works, we first show that:

\begin{quotation}\noindent For almost all $x$, for large enough $m$, if $s_m$ corresponds to  $x+\beta_m$ then for all mature $n\le m, r_n(s_m)=r_n(\mcs_\beta(s))$.\end{quotation}
This is a Borel-Cantelli argument.  
Note that if $r_m(s_m)=r_{m}(\mcs_\beta(s))$ then for all mature $n\le m, r_n(s_m)=r_n(\mcs_\beta(s))$.  Hence it suffices to show that for almost all $s$, all sufficiently large $m$, $r_m(s_m)=r_m(\mcs_\beta(s))$.

If $x$ corresponds to $s$ then the only way that $r_m(s_m)\ne r_m(\mcs_\beta(s))$ is 
if $x+\beta_m$ is in a different level of the $m$-tower than 
$x+\beta_m+ \sum_{p=m}^\infty {a_p\over k_pq_p}$. In turn, the only way that this can happen is if for some $i$,
\[ x+\beta_m\in [i\alpha_m+1/q_m-{\sum_{p=m}^\infty {a_p\over k_pq_p}}, i\alpha_m+1/q_m).\]
  The latter interval is the right hand portion of a level in the $m$-tower, i.e. of an interval of the form $[i\alpha_m, i\alpha_m+1/q_m)$. 

The collection of $x$ that have this property for a given level  $i$ has measure 
$\sum_{p=m}^\infty {a_p\over k_pq_p}$.  Since there are $q_m$ many levels $i$, the measure  of all of the $x$ with this property at stage $m$ is 
$q_m*\left(\sum_{p=m}^\infty {a_p\over k_pq_p}\right)$.
Computing:

\begin{align*}
q_m*\left(\sum_{p=m}^\infty {a_p\over k_pq_p}\right)&=
{a_m\over k_m}+q_m*\sum_{m+1}^\infty {a_p\over k_pq_p}\\
&<{a_m\over k_m}+{q_m\over q_{m+1}}\sum_{m+1}^\infty {a_p\over k_p}\\
&\le {a_m\over k_m}+{1\over k_ml_mq_m}C,
\end{align*}
where $C=\sum_1^\infty {a_p\over k_p}$. Since we assume that \hyperlink{nr4}{$\sum {6^n\over k_n}<\infty$} and $a_p<6^{n-1}$, $C$ is finite. We see immediately that the measures of the  collections of $x$ such that at some stage $m$ the level of $x+\beta_m$ in the 
$m$-tower is different from the level of $x+\beta$ in the $m$-tower is summable.  By Borel-Cantelli, it follows that for almost all $s$ there is an $N$ for all $m\ge N$, $r_m(s_m)=r_m(\mcs_\beta(s))$.

\medskip

From the choice of $a_n$ for all but  a set of measure at most $1/6^n$, the  $s$ are ill-matched with
$s_{n+1}$. Again by the Borel-Cantelli lemma, for almost all $s$ there is an $N_1$ for all 
$n\ge N_1$ $s$ and $s_{n+1}$ are ill-matched. Since for almost all $s$ and all large enough $n$ the level of $s_{n+1}$ is equal to the level  $\mcs_\beta(s)$ it follows that for almost all $s$ and all large enough $n$ $s$ is ill-matched with $\mcs_\beta(s)$.  If $\nu$ is the unique invariant measure on $\mck_\alpha$ then equation 33 of \cite{part3} defines
\[\Delta_n(\beta)=\nu(\{s:s\mbox{ is ill-$\beta$-matched at stage $n$}\}).\]
We have shown that $\Delta_n(\beta)\to_n 1$.  Hence 
\[\Delta(\beta)=\sum_n\Delta_n(\beta)\]
is infinite. 
\medskip
Hence we have shown that $\beta$ is not central.\qed

We now prove Lemma \ref{yeah yeah yeah}
\medskip

\pf\!\!\! [Lemma \ref{yeah yeah yeah}] First note that the 
analysis in section 6.3, on page 50 of \cite{part3}, says that  for any non-central 
$\beta$ we can choose $hd_1$ 
	and $hd_2$ and a spaced out set $G$ such that, as in equation 35 on page 50, letting 
	\begin{equation}\label{newdef misal}
	\misal_n=\{s:s \mbox{ is ill-$\beta$-matched at stage $n$ and in 
	configuration {$P_{{{hd_1},{hd_2}}}$}}\}\notag
	\end{equation}
	 we get equation 36 on page 50 of \cite{part3}:

    \begin{equation}
    \label{root cause} 
    \sum_{n\in G}\nu(\misal_n)=\infty.
    \end{equation}

We now observe that for $n<m\in G$, $\misal_n$ and $\misal_m$ are 
probabilistically independent. This follows from Lemma 75 on page 47 of \cite{part3}: belonging 
to $\misal_n$ is an issue of the value of $d^{n+1}-d^n$.
The differences $d^{m+1}-d^m$ are independent of the differences 
$d^{n+1}-d^n$, hence the sets $\misal_n$ and $\misal_m$ are pairwise 
independent.
Which level $x$ is on in the $n$-tower is independent of whether or not $x$ is misaligned at the next stage.

 Let $M_m$ be the collection of $s$ that are mature at stage $m$. Then applying the  ``hard" Borel-Cantelli lemma, for almost all $s\in M_m$, there are infinitely 
many $n\in G, s\in \misal_n$.  Since $\bigcup_m M_m$ has measure one, for almost all $s\in \bk^c$ there are infinitely many $n, s\in \misal_n$.
\bigskip

We now argue that if $\mcj_1$ and $\mcj_2$ are two joinings of 
$\bk^c\times \bk^c$ over $\mcs_\beta$, then they are identical. Thus they are 
both the relatively independent joining. The result follows from the following claim which is an analogue of the claim in Proposition 32 of \cite{FRW}:

	\paragraph{Claim} Let $\mcj$ be a joining of $\bk^c$  with itself that projects to the graph joining of 
	$\mck_\alpha$ with itself given by $\mcs_\beta$. Then for all cylinder sets $\la a\ra\times \la b\ra$ in 
	$\bk^c\times \bk^c$, the density of occurrences of $(a,b)$ in a generic pair 
	$(x,y)$ for $\mcj$ does not depend on the choice of $(x,y)$.

	\bigskip
\pf  Since $\beta$ is non-central, and $x,y$ are generic and $\mcj$ extends $\mcs_\beta$, we know that for infinitely many $n\in G$ the $n$-words of $x$ and $y$ are misaligned. Let $G^*$ be this set.

 It suffices to show that: \begin{itemize}
 \item There is a sequence of subblocks of the principal $n+1$-subwords of $x$ and $y$ of total  length $B_n$,
 \item as $n\in G^*$ goes to infinity, $B_n/q_{n+1}$ goes to $0$,
 \item after removing the subwords in $B_n$  the number of occurrences of $\la a\ra\times \la b\ra$ is independent of the choice of $(x,y)$.
 \end{itemize}

Fix a large $n\in G^*$. We count occurrences of $(a,b)$ in $(x,y)$ over the portion of the principal $n+1$-subwords of $x$ that overlap with the $n+1$-blocks of $y$. As in Proposition 32 of \cite{FRW}, we show that, up to a negligible portion, this is independent of $(x,y)$. From the definition of  $\misal_n$ for $n\in G$, there are fixed values of $hd_1$ and $hd_2$. The number $hd_2$ determines the overlap of the 
$n+1$-block of $x$ containing $x(0)$ is the left or right overlap. For convenience,  assume that $hd_1=L$ and $hd_2=R$. 

First: discard $n$-subwords that are not mature. This is a negligible portion.

Next, shift $y$ back by $d^n(x)$, so that the mature $n$-subwords of $x$ in the principal $n+1$-subword are aligned along the overlap of the 
principal $n+1$ subword of $y$ with the corresponding $n$-subword of $y$.\footnote{ Sections 4.3-4.6 of \cite{part3} discuss how the spacings of 
left and right overlaps correspond.}

Then by specification {J.10.1} and the fact that $x$ and $y$ are misaligned, any pair of 
$n$--words $(u,v)$ occurs almost exactly $1/s_{n}^2$  times. So, after discarding a 
negligible portion of  the occurrences all pairs occur the same number of 
times. Shifting them all  back by $d^n(x)$,  an amount determined by $\beta$ and thus independent of $x$ and 
$y$, gives a collection of counts of occurrences of $(a,b)$ in all pairs 
$(u, sh^{d^n}(v))$ with all pairs occurring essentially the same number of times.  The result 
is independent of the choice of $x$ and $y$. The errors from the {negligible portions} and they go to zero  in proportion to $n+1$.
This proves the claim. \qed

\subsection{Diffeomorphisms of the Torus}
\label{sec:CT:ss:tor}

The map $\mcf\circ F_\mco$ maps codes for $\Pi^0_1$ sentences to construction
sequence for circular systems. We now indicate how to realize circular systems
as diffeomorphisms and why these diffeomorphisms are computable. The realization
map is described completely in \cite{part1}.  We review it here to verify its
effectiveness.

The construction is in two stages. In both parts a sequence of periodic
transformations is constructed and the limits are isomorphic to the given
uniform circular system.  In both constructions, the torus, viewed as \(
[0,1]\times[0,1]\) with appropriate edges identified, is divided into rectangles. These are then permuted by the
periodic transformations according to the action of the shift operator on the
circular system.  In the first stage, this permutation is built without regard to
continuity. The result is an abstract measure preserving transformation. In the
second part, using smooth approximations to these permutations, the limit is a
$C^\infty$ diffeomorphism.

The main tool for moving from the  discontinuous, symbolic transformations to
the smooth geometric transformations is the Anosov-Katok method of Approximation
by  Conjugacy \cite{AK}.  To allow for this smoothing the parameter sequence
$\la k_n, l_n:n\in\nn\ra$  must have the sequence of $l_n$'s grow sufficiently
fast.  

The lower bounds $l_n^*(\la k_m:m\le n\ra, \la l_m:m<n\ra)$ will be
determined inductively, the complete list of requirements on $l_n^*$ appears in Appendix \ref{NumPar}. 

For the
moment we assume we are given the circular sequence $\la \mcw_n^c:n\in\nn\ra$ with prescribed coefficient sequences
$\la k_n, l_n:n\in\nn\ra$ where the $l_n$ grow sufficiently fast.

The periodic approximations to the  first stage transformation  are of the form 
    \begin{equation}
        T_n = Z_n\circ\rot_{\alpha_n}\circ Z_n^{-1}\label{form of approx}
    \end{equation}
which result from   conjugating  horizontal  rotations
\((x,y)\mapsto_{\rot_{\alpha_n}}(x+\alpha_n,y)\), with the more complicated
transformations \(h_n:\B T^2\to\B T^2\) that permute rectangular subsets of 
\(\B T^2\).  The $\alpha_n$ are the rationals constructed from the coefficient sequence
$\la k_n, l_n:n\in\nn\ra$  described in section~\ref{prelims}. The maps $Z_n$
are of the form \[Z_n = h_1\circ h_2\circ\ldots\circ h_n\] where  $h_i$ codes
the combinatorial behavior of the $i^{th}$ application of the $\mcc$-operation.
The initial, discontinuous transformation \(T\) will then be the almost-everywhere pointwise limit
of the sequence \(\la T_n:n\in\B N\ra\).

In the second part of the construction the $h_n$'s will be replaced by smooth
transformations $h_n^s$ that are close measure theoretic
approximations to the $h_n$'s. This results in a new sequence
    \begin{equation}
    \label{form of Hn}
        H_n=h_1^s\circ h_2^s \circ \dots \circ h_n^s.
    \end{equation}

The analogue of equation~\ref{form of approx} for the final smooth transformation is:
    \begin{equation}
    \label{AK approx}
        S_n=H_n\rot_{\alpha_n}H_n^{-1}
    \end{equation}
The sequence of $S_n$'s converge in the $C^\infty$-topology to a $C^\infty$ measure preserving transformation $S:\bt^2\to \bt^2$. 
    
\paragraph{Why do we do this?} In \cite{part1} it is shown that $T$ is measure isomorphic to $\bk^c$.  Hence if 
$\bk^c=\mcf\circ F(N)$ we have $\phi_N$ is true if and only $T\cong T^{-1}$.    Since $S\cong T$, $\phi_N$ is true if and only $S\cong S^{-1}$. 
    Thus if we define the realization function $R$ by setting $R(\bk^c)=S$, we see that $R\circ\mcf\circ F_\mco$ is a 
    reduction of the collection of codes for true $\Pi^0_1$-sentences to the set of recursive diffeomorphisms 
    isomorphic to their inverses. This is the content of figure  \ref{the square}.

In addition to these results in \cite{part1}, we will show that the sequence of $S_n$'s can be taken to be effective,
converge in the $C^\infty$ topology and that if $S(N)$ comes from $N$ and $S(M)$ comes from $M$, then $S(N)\not\cong S(M)$. This will complete the proof of Theorem \ref{thm:main}.

\subsubsection{Painting the circular system on the torus }

We encode the symbolic system $\bk^c$ on the torus by inductively constructing
the sequence of $h_n$'s.  The map $h_0$ is the identity map corresponding to $\mcw^c_0=\Sigma$.  To build $h_{n+1}$, \(\B T^2\) is subdivided  into
rectangles which are then permuted.

\begin{definition}[Rectangular subdivisions]
\label{def:rects}
    Let \(n,m\in\B N\).
        \begin{itemize}
            \item For an arbitrary natural number $q$,  \(\C I_q\) represent the collection of intervals\\
                \([0,\tfrac1q),\ [\tfrac1q, \tfrac2q),\ \ldots, [\tfrac{q-1}{q},
                1)\).
            \item Given \(\C I_q\) and \(\C I_s\), let \(\C I_q\otimes\C I_s\)
                be the collection of all rectangles \(R = I_0\times I_1\), where
                \(I_0\in\C I_q\) and \(I_1\in\C I_s\).
            \item Let \(D\subseteq\B T^2\). Then, for a collection of rectangles
                \(\xi\), the restriction of \(\xi\) to \(D\) is given by
                    \begin{equation*}
                        \xi\rest D = \{R\cap D: R\in\xi\}.
                    \end{equation*}
            \item Recall the parameter sequences \(\la q_n:n\in \nn\ra\) and
                \(\la s_n : n \in \B N\ra\). Further recall that
                $s_n=|\mcw^c_n|$ and \(q_n = |u|\) for \(u \in \mcw^c_n\).
                Define
                    \begin{equation*}
                        \xi_n = \C I_{q_n}\otimes \C I_{s_n}.
                    \end{equation*}
            \item Lastly, for \(0\leq i < q_n\) and \(0\leq j < s_n\), let
                \(R^n_{i,j}\) be the element of \(\xi_n\) given by
                \([\tfrac{i}{q_n}, \tfrac{i+1}{q_n})\times[\tfrac {j}{s_n},
                \tfrac{j+1}{s_n})\).
        \end{itemize}
\end{definition}
Note that  there is a straightforward description of the action of
\(\rot_{\alpha_n}\) on \(\xi_n\):
    \begin{equation*}
        \rot_{\alpha_n}: R^n_{i,j}\mapsto R^n_{i+p_n,j}
    \end{equation*}
where addition in the subscript is performed modulo \(q_n\).

The map $h_{n+1}$ will be defined as a permutation of $\mci_{k_nq_n}\otimes
\mci_{s_{n+1}}$ and thus induces a permutation  of $\xi_{n+1}$. It is important to make
$h_{n+1}$ commute with $\rid{\alpha_n}$.  To do this $h_{n+1}$ is first defined on
$(\mci_{k_nq_n}\otimes \mci_{s_{n+1}})\rest([0,1/q_n)\times [0,1))$ and then copied
over equivariantly to $\bt^2$.

\paragraph{Constructing  the $h_{n}$'s:}  The paper \cite{part1} is concerned with realizing circular systems, and so builds the $h_n$'s in terms of the  \emph{prewords} used to construct the  sequence $\la \mcw_n^c:n\in\nn\ra$.  In the case that $\la \mcw_n^c:n\in\nn\ra$ is in the range of $\mcf$, the prewords are determined by the underlying odometer based sequence $\la \mcw_n:n\in\nn\ra$. We describe $h_{n+1}$ directly in terms of the odometer sequence $\la \mcw_n:n\in\nn\ra=F_\mco(N)$.

 Fix enumerations $\la w^n_s:0\le s < s_n\ra$ of each  $\mcw_n$. 
The words in $\mcw_{n+1}$ are concatenations of words in $\mcw_n$: 
\[w^{n+1}_s=w_0w_1\dots w_{k_n-1}\]
where each $w_i=w^n_{s'}$ for some $s'$.

To each $w^{n+1}_s$ associate the horizontal strip $[0,1)\times [s/s_{n+1},
(s+1)/s_{n+1})$ and each $w^n_{s'}$ with $[0,1)\times [s'/s_{n},
(s'+1)/s_{n})$.
\begin{prop}\label{build perms}
There is  a permutation of $\mci_{k_nq_n}\otimes \mci_{s_{n+1}}\rest
    [0,1/q_n)\times [0,1)$ such that for all $0\le s<s_{n+1}$,
        \begin{quotation}
            \noindent if $w_i=w^n_{s'}$ then 
                \begin{eqnarray}\label{in place}
                    h_{n+1}([i/k_nq_n,(i+1)/k_nq_n)\times
                    [s/s_{n+1},(s+1)/s_{n+1}))\\
                    \subseteq [0,1/q_n)    \times [s'/s_n, (s'+1)/s_n).\notag
                \end{eqnarray}
        \end{quotation}
\end{prop}

\pf Equation~\ref{in place} gives  regions that each atom of
$\mci_{k_nq_n}\otimes \mci_{s_{n+1}}\rest[0,1/q_n)\times [0,1)$ must be sent to
by $h_{n+1}$. To prove there  is such a permutation we see that each region has
exactly the same number of subrectangles as  the cardinality of the collection
of atoms that must map into it.

We count occurrences of $n$-words in $(n+1)$-words.  Fix a word
    \[w^{n+1}_s=w_0w_1 \ldots w_{k_n-1} \in \mcw_{n+1}.\]
Then, by strong
uniformity each $n$-word $w^n_{s'}$ occurs $k_n/s_n$ times as a $w_i$. So each
word $w^{n+1}_s$ puts $k_n/s_n$ rectangles in a target region. Since there are
$s_{n+1}$ many words of the form $w^{n+1}_s$ the target regions must contain $s_{n+1}(k_n/s_n)$
rectangles. 

Each horizontal strip of $[0,1)\otimes \mci_{s_n}$ is divided into
$s_{n+1}/s_n$ many horizontal strips by $[0,1)\otimes \mci_{s_{n+1}}$ and each
vertical strip of $\mci_{q_n}\otimes [0,1)$ is divided into $k_n$ many vertical
strips by $\mci_{k_nq_n}\otimes [0,1)$.  Thus each atom of the partition $\xi_n\rest[0,1/q_n)\times \mci_s$
is divided into $k_n(s_{n+1}/s_n)$ rectangles by $\mci_{k_nq_n}\otimes
\mci_{s_{n+1}}$. In particular $\mci_{k_nq_n}\otimes \mci_{s_{n+1}}\rest
[0,1/q_n)\times [s'/s_n, (s'+1)/s_n)$ has $k_n(s_{n+1}/s_n)$ many atoms. 

Hence each target region contains the same number of rectangles as atoms sent to
it and there is a map $h_{n+1}$ satisfying equation~\ref{in place}.\qed

Since $h_{n+1}$ is a permutation of $\mci_{k_nq_n}\otimes \mci_{s_{n+1}}\rest
[0,1/q_n)\times [0,1)$ for each $1\le i<q_n$, it can be copied onto each $\mci_{k_nq_n}\otimes
\mci_{s_{n+1}}\rest [ip/q_{n},(ip+1)/q_n)$.  The result of this is  a permutation of
$\mci_{k_nq_n}\otimes \mci_{s_{n+1}}$ (and hence $\xi_{n+1}$) that commutes with
the rotation $\rid{\alpha_n}$.

\begin{remark}
It is a clear that $h_{n+1}$ can be defined in a primitive recursive way using the data $\mcw_{n+1}$.
\end{remark}

\paragraph{Remark} It is shown in \cite{part1} that having defined the sequence
of $h_n$'s in this manner, for sufficiently fast growing $l_n$ the
transformations $T_n$ converge in measure to a measure preserving transformation
$T:(\bt^2,\lambda)\to (\bt^2,\lambda)$ that is isomorphic to the original
circular system defined by $\la \mcw_n^c:n\in\nn\ra$. The map taking  $\la
\mcw^c_m:m\le n\ra$ to $T_n$ is primitive recursive.

\subsubsection{Smoothing the $T_n$}
\label{sec:smooth_tn}

We now must smooth the $T_n$'s to produce $S_n$'s that have 
measure-theoretic limit $S$ which is isomorphic to $T$. Secondly, we show that
$S$ is a recursive diffeomorphism.

For our discussion of smoothing we need an effective complete metric on the $C^\infty$-diffeomorphisms. 
Note that the \(C^\infty\) topology is the coarsest common
refinement of the \(C^k\) topologies for each \(k \in \B N\). There are many choices for 
effective/recursive metrics generating the \(C^k\) topology for each \(k\);
for instance the metric derived from the norm given by
	\begin{equation*}
        \|f\|_k = \max_{x \in \B T^2} \|f(x)\|^0 + \|D f(x)\|^1 + \cdots + \|D^k
        f(x)\|^k
	\end{equation*}
where \(\| - \|^j\) is the $j$-norm on \(\B R^{2j+2}\).  Given an effective sequence of complete metrics $\la d^k:k\in\nn\ra$ generating the $C^k$ topologies, with distances bounded by 1, then    
	\begin{equation*}
        d^\infty = \sum_{k = 0}^\infty 2^{-(k + 1)} d^k
    \end{equation*}
generates the \(C^\infty\) topology.

Fix such a complete effective metric giving rise to the \(C^\infty\) topology
on \(\B T^2\).
    Without loss of generality we can assume that
    \begin{equation}
    \label{just the obvious}
        d^\infty(S,T)\le \max_{x\in \bt^2}d_{\bt^2}(S(x),T(x)),
    \end{equation}
where $d_{\bt^2}$ is the ordinary metric on $\bt^2$.

To pass from the discontinuous $Z_n$'s to diffeomorphisms, the $h_i$'s are
replaced by  smooth $h^s_i$  which are very close approximations and give the
$H_n$'s  in equation~\ref{form of Hn}. Then the $H_n$'s  will also be
diffeomorphisms.  While there is no control over the $C^\infty$-norms of the
\(H_n\), the key observation at the heart of the Anosov-Katok method is the following: 
if \(h^s_{n+1}\) commutes with \(\rot_{\alpha_n}\) then
    \begin{align}
    \label{commutes}
            S_n &= H_n \circ\rot_{\alpha_n}\circ H_n^{-1}\notag \\
            &= H_n \circ h^s_{n+1} \circ (h^s_{n+1})^{-1} \circ
            \rot_{\alpha_n} \circ H_n^{-1}\notag\\
            &= H_n \circ h^s_{n+1} \circ \rot_{\alpha_n} \circ
            (h^s_{n+1})^{-1} \circ H_n^{-1}\notag\\
            &= H_{n+1} \circ \rot_{\alpha_n} \circ H_{n+1}^{-1}.
    \end{align}
Hence by taking $\alpha_{n+1}$ sufficiently close to $\alpha_n$, $S_{n+1}$ can
be taken as close as necessary to $S_n$ in the $C^\infty$-norm.

To carry out this plan we begin by describing how we smooth the $h_n$'s. This is done
explicitly in Theorem 35 of \cite{part1}, which says: 

\begin{theorem}[Smooth permutations]
\label{smp}
    Let \(\B T^2\) be divided into the collection of rectangles \(\C I_n
    \otimes \C I_m\) and choose \(\epsilon > 0\). Let \(\sigma\) be a
    permutation of the rectangles.  Then there is  an area preserving $C^\infty$-diffeomorphism
    \(\phi:\B T^2\to\B T^2\) such that \(\phi\) is the identity on a
    neighborhood of the boundary of \([0,1]\times[0,1]\) and for all but a
    set of measure at most \(\epsilon\), if \(x\in R\), then
    \(\phi(x)\in\sigma(R)\) for all \(R\in\C I_n \otimes \C I_m\).
\end{theorem}

  In   Lemma 36
of \cite{part1} it is shown that an arbitrary permutation of $\mci_n\otimes \mci_m$ can be built by taking a composition of transpositions of adjacent 
rectangles. The transformation $\phi$ is then built effectively as a composition of smooth
near-transpositions that swap adjacent rectangles. We summarize the proof.  More details appear in Appendix~\ref{diffeos}, where it is shown that it can be 
carried out recursively in a code for the permutation $\sigma$.

\medskip

The main technical point for building the near-transpositions of adjacent rectangles
is captured by showing that for all $0< \gamma<1$ and arbitrarily small
$\epsilon<1-\gamma$, there is a diffeomorphism $\phi_0$ of the unit disk in
$\poR^2$ such that:
    \begin{enumerate}
        \item $\phi_0$ rotates the top half of the disk of radius $\gamma$ to
            the bottom half and vice versa.
        \item $\phi_0$ is the identity in a neighborhood of the unit circle
            of width less than $\epsilon$.
    \end{enumerate}
The map  $\phi_0$ is constructed by considering a primitive recursive $C^\infty$
map $f:[0,1]\to [0,\pi]$ that is identically equal to $\pi$ on $[0,\gamma]$ and
is $0$ in a neighborhood of 1.  Then $\phi_0$ rotates the circle of radius $r$
by $f(r)$ radians. Taking $\gamma$ very close to $1$ gives a
smooth near transposition.

Using Riemann mapping theorem techniques, these rotations of the disk can be copied over to  measure preserving maps from 
$[-1,1]\times[0,1]$ to itself that 
	\begin{enumerate}
	\item take  all but $1-\epsilon/2$ mass of $[-1,0]\times[0,1]$ to $[0,1]\times[0,1]$ and vice versa,
	\item are analytic on the interior of $[-1,1]\times[0,1]$,
	\item are the identity in a neighborhood of the boundary.
	\end{enumerate}

Since every permutation of $\{0,1,\dots mn\}$ can be written as a composition of less than or equal to $(nm)^2$ transpositions of the form $(k, k + 1)$, given any $\sigma$ we can build $\phi$ by taking $\epsilon$ small enough and composing sufficiently good approximations between adjacent rectangles corresponding to the transpositions composed to create $\sigma$.

\paragraph{Building $S_n$.} Using Theorem \ref{smp} we can effectively
choose a smooth $h^s_{n+1}$ which well-approximates $h_{n+1}$ measure
theoretically.  By choosing the approximation well, we can guarantee that the
$S_n$ in equation~\ref{AK approx} moves the partitions $\xi_n$ very close to
where the $T_n$'s move the $\xi_n$'s.

Since $h^s_{n+1}$ is effective, using the continuity of composition with respect
to $d^\infty$,  $S_{n+1}$ can be made arbitrarily close to $S_n$  by taking
$\alpha_{n+1}$ sufficiently close to $\alpha_n$.  Thus if \(\alpha_n\) converges
to $\alpha$ sufficiently quickly, the sequence \(\la S_n:n\in\nn\ra\) is Cauchy
with respect to the complete metric $d^\infty$ and hence converges to a smooth measure preserving
diffeomorphism \(S\).  Taking the sequence of $h^s_n$'s to be sufficiently close
to the $h_n$'s the $S_n$'s are sufficiently close to the $T_n$'s to apply Lemma
30 of \cite{part1} to  show that the diffeomorphism $S$ is measure theoretically
isomorphic to $T$.  Hence $(\bt^2, \lambda, S)$ is measure theoretically
isomorphic to $(\bk^c, \nu, \sh)$.

\paragraph{The induction.} The discussion above was predicated on choosing the
$l_n$'s to grow \emph{fast enough}.  We now show how to inductively choose lower
bounds $l_n^*$ on the $l_n$. \hyperlink{nre}{Numerical Requirement E} gives one collection of lower bounds for the $l$'s, independently of the choices of the maps $H_n$ and numbers $\alpha_n$. Hence we choose $l_n^*$ to dominate this sequence of lower bounds, as well as the lower bounds we add here.

Suppose we have defined $H_{n+1}$ from $h_{n+1}^s$ and $H_n$ in a manner that
satisfies equation~\ref{commutes} holds and that the $H_n$'s can be computed effectively. 
Then, for any given $\epsilon$, and
small {rational} $\beta$,
    \begin{equation}
        {d^\infty(H_{n+1}\rid{\alpha_n + \beta}H^{-1}_{n+1}, H_n\rid{\alpha_n}H^{-1}_n)}
    \end{equation}
can be primitively recursively computed to within a given
$\epsilon$.
Moreover, this is a decreasing function of \(\beta\)  for small
$\beta>0$.  Thus one  one can effectively find a $\delta$ such that if
$|\alpha_{n+1}-\alpha_n|<\delta$, then $d^\infty(S_{n+1},S_n)< 2^{-(n + 1)}$.

Recall the definitions of the $\alpha_n=p_n/q_n$ from equations~\ref{qns},
\ref{pns} and \ref{alphans}. Then $\alpha_n$ does not depend on $l_n$ and 
    \[\alpha_{n+1}=\alpha_n+{1/k_nl_nq_n^2}.\]
Thus to make $\alpha_{n+1}$ close to $\alpha_n$ it suffices to make $l_n$
sufficiently large that 
	\begin{equation}\label{deldel}
	1/k_nl_nq_n^2<\delta.
	\end{equation}
 
\begin{description}
\item[Numerical Requirement G] \hypertarget{nr5}{The parameter} $l_n$ is
chosen sufficiently large that
    \begin{equation}
    \label{growth 2}
        d^\infty(S_{n+1}, S_n)< 2^{-(n + 1)}
    \end{equation}
\end{description}

The numbers $\alpha_n$, $\la \mcw^c_m:m\le n\ra$ and $\la h^s_m:m\le n+1\ra$
determine the $\delta$ in equation \ref{deldel} and thus how large $l_n$ must be.  All of this data can be computed
recursively from $\la \mcw_m:m\le n+1\ra$.  (We note that neither the choice of $s_{n+1}$
nor the definition $h^s_{n+1}$ uses $l_{n}$.)

\subsubsection{The effective computation of $S_n$}
\label{sssec:ecsn}

We now show that each element of the sequence \(\la S_n : n \in \B N\ra\) is
effectively computable (Definition~\ref{effcont}). 
    \begin{claim}
\label{clm:euc2}
        The functions \(h_n^s\) and \(\C R_{\alpha_n}\) are effectively computable $C^\infty$-functions. As a consequence each  \(S_n\) is
effectively uniformly continuous.
    \end{claim}
\begin{proof}[Proof of Claim~\ref{clm:euc2}]

 For simplicity of
exposition, we only show how to compute the modulus of continuity and
approximation for \(S_n\) itself; finding the modulus of continuity and
approximations to the higher differentials is conceptually identical but
notationally cumbersome.

Recall that we must produce two functions:
    \begin{itemize}
        \item A modulus of continuity, \(d : \B N \to \B N\), and
        \item An approximation, \(f : (\{0, 1\} \times \{0, 1\})^{< \B N} \to
            (\{0, 1\} \times \{0, 1\})^{< \B N}\).
    \end{itemize}
    
It is routine to check that 
          if \(T_0\) and \(T_1\) are effectively uniformly continuous---that is, if there exist moduli
        of continuity \(d_0\) and \(d_1\) and approximations \(f_0\) and \(f_1\)
        corresponding to each---then the composition, \(T_1 \circ T_0\) is
        effectively uniformly continuous.

The second part of the claim follows from the first since Equations \eqref{AK approx} and \eqref{form of Hn} show,
    \begin{equation}
    \label{eq:euc}
        S_n = h_1^s \circ h_2^s \circ \cdots \circ h_n^s \circ \rot_{\alpha_n}
        \circ (h_n^s)^{-1} \circ \cdots \circ (h_2^s)^{-1} \circ (h_1^s)^{-1}.
    \end{equation}

    The case of \(\rot_{\alpha_n}\) is particularly simple. Since
    \(\rot_{\alpha_n}\) is an isometry, it has a Lipschitz constant of \(1\). In
    particular, the modulus of continuity is simply given by \(d(n) = n\), and,
    since \(\rot_{\alpha_n}\) is well-defined on rational points, we can also
    determine the approximation by setting \(f\) to be
        \begin{equation*}
            ([x]_m, [y]_m) \mapsto ([x]_m + [\alpha_n]_m, [y]_m)
        \end{equation*}
    Where \([z]_m\) denotes the smallest dyadic rational \(k \times 2^{1-m}\) for \(0 \leq k
    \leq 2^m\) minimizing \(|z - [z]_m|\).

    In the case of \(h_m^s\) for \(m \leq n\), recall from the discussion after Theorem \ref{smp} that $h^s_m$ can be built as a composition of a sequence of smooth transpositions: 
        \begin{equation*}
            h_m^s = \sigma_0^s \circ \sigma_1^s \circ \cdots \circ
            \sigma_{t(m)}^s
        \end{equation*}
     Note that the number of transpositions necessary, \(t(m)< |\xi_{m+1}|^2 = (k_m\cdot q_m \cdot
    s_{m+1})^2\), and is a computable function of \(m\)
    since it is the number of transpositions necessary to build the permutation in Proposition \ref{build perms}.

    Since \(\sigma_j^s\) is a \emph{smooth} transposition of an explicit
    form (given in Appendix~\ref{diffeos}), one can calculate a uniform
    Lipschitz constant $L_j^s$ {for it; hence, taking \(L_m > \max_{s \le
    t(m)} L^s_j\), we have that}
        \begin{equation*}
            |h_m^s(x) - h_m^s(y)| < (L_m)^{t(m)+1}|x - y|.
        \end{equation*}
    Consequently, a suitable modulus of continuity for \(h_m^s\) is given by
        \begin{equation}
        \label{eq:explicit}
            d(n) = n + \lceil t(m) \cdot \log_2(L_m) \rceil,
        \end{equation}
    where \(\lceil x \rceil\) is the smallest integer greater than \(x\).
The  construction of a primitive recursive approximation  to \(h_m^s\)
    is straightforward from the primitive recursive approximations to the $\sigma^s_n$'s.  
    As we remarked in Section \ref{effcompdif}, it follows that  $(h_m^s)^{-1}$ is primitive recursive.
\end{proof}

In summary, the modulus of continuity and approximation for \(S_n\) can be
calculated using the following steps:
    \begin{enumerate}
        \item\label{enum:first_step} Compute the 
        \(\la h_m :
            m\leq n\ra\);
        \item Build the approximations to \(h_m^s\) and \((h_m^s)^{-1}\) using
            \(h_m\) and the smooth transpositions given in
            Appendix~\ref{diffeos} for \(m \leq n\);
        \item Compute the moduli of continuity of \(\la h^s_m : m \leq n\ra\) and
            their inverses;
        \item\label{enum:last_step} Compute \(\la \alpha_m : m \leq n\ra\) (and,
            consequently, the approximations and moduli of continuity for \(\la
            \rot_{\alpha_m} : m\leq n\ra\));
        \item Compute the approximation and modulus of continuity of \(S_n\) by
            composing the approximations and moduli of continuity calculated in
            Steps~\ref{enum:first_step} through~\ref{enum:last_step} according
            to Equation \eqref{eq:euc}.
    \end{enumerate}

\subsection{Completing the proof}
\label{sec:victory}

Theorem~\ref{thm:main} claims the existence a computable function $F$, which on
inputting a natural number $N$ (corresponding to the $\Pi^0_1$ sentence
{$\phi$}) outputs a code for a computable diffeomorphism $S(N)$ of $\bt^2$.
Whether or not $S(N)$ is measure theoretically conjugate to $S(N)^{-1}$ is equivalent
to the truth of falsity of {$\phi$}. Finally for different numerical inputs the
corresponding $S$'s will not be isomorphic. In summary, letting $S=S(N)$,
    \begin{enumerate}[(A)]
        \item If $N$ codes {$\phi$}, then $\phi$ is true if and only iff
            $S\cong S^{-1}$
        \item On input $N$, $F$ recursively determines a code for an effectively
            $C^\infty$ map of the torus to itself, i.e., \(F\)
            determines: 
                \begin{enumerate}[i.)]
                    \item A 
                        computable function \(d: \B N \times \B N \to \B N\), where $d(k,-)$ computes the moduli of uniformity of the 
                        $k^{th}$ differential of $S(N)$, 
                        and
                    \item A 
                        computable function $f(k,-)$ where \(f(k, -)\) is a map on dyadic
            rational points of \(\B T^2\) approximating \(D^k S(N)\) (the \(k\)-th
            differential of \(S(N)\)). Given an input that is precise to
            \(d(k, n)\) digits, $f(k,-)$ approximates the first $n$-partial derivatives $\{{\partial^n\over \partial^ix\partial^{n-i}y}:0\le i\le n\}$ to $n$-digits.
                \end{enumerate}

        \item If $N\ne M$, then the associated diffeomorphisms $S(N)$ and $S(M)$
            are not conjugate.
    \end{enumerate}
Because the function $F$ maps natural numbers to natural numbers (the codes for the diffeomorphisms) we let
 $F^\flat$ be the associated function $R\circ\mcf\circ F_\mco$ that maps into the space of actual diffeomorphisms.   It
produces a diffeomorphism of $\bt^2$ from a G\"odel number $N$ for a
$\Pi^0_1$ set. We show $F^\flat$ satisfies $(A)$ and $(C)$ and then argue there
is a (primitively) computable routine coded by $F(N)$ that has the same values.

\paragraph{Item (A)} Given $N$, $F_\mco(N)$ computes an odometer-based
construction sequence $\la \mcw_n:n\in\nn\ra$. By Theorem~\ref{red to
odos}, if $\bk(N)$ is the uniquely ergodic symbolic shift associated with the
construction sequence then  $\bk(N)\cong\bk(N) ^{-1}$ if and only if $\phi$ is
true.

The sequences $\la l_n:n\in\nn\ra$, $\la \mcw_n^c:n\in\nn\ra$ and $\la
h_n^s:n\in\nn\ra$ are computed. If $\bk^c$ is the circular system associated
with $\la \mcw_n^c:n\in\nn\ra$, then Proposition~\ref{victory for circs} shows
that $\bk^c\cong (\bk^c)^{-1}$ if and only if $\phi$ is true. 

Finally the realization map $R$ preserves isomorphism.  So if $S=R(\bk^c)$, then 
$(\bt^2,\lambda, S)\cong (\bt^2,\lambda, S^{-1})$ if and only if $\phi$ is true.

\paragraph{Item (C)} We need to see that for $M<N, S(N)\not\cong S(M)$. Since the
realization map $R$ preserves isomorphism if suffices to see that
$(\bk^c)^M=\mcf\circ F_\mco(M)$ is not isomorphic to $(\bk^c)^N$.

By Corollary \ref{one one} we see that the Kronecker factor
of $\bk^M$ is $\mck_{\alpha^M}$. Any isomorphism between
$(\bk^c)^N$ and $(\bk^c)^M$ must take the respective Kronecker factor of one to
the other, hence would imply an isomorphism between $\mck_{\alpha^N}$ and
$\mck_{\alpha^M}$. 

However this is impossible since Corollary~\ref{one one} implies that
$\pi>\alpha^M>\alpha^N>0$.

\paragraph{Item (B)}  \(F^\flat\) is a map from \(\B N\) to 
diffeomorphisms. By the result of Section~\ref{sssec:ecsn}, the diffeomorphisms
are recursive. We must show that there is a  recursive algorithm coding
a function $F$ that computes the moduli of continuity and approximations to each
$F^\flat(N)$ and its differentials.

We use the notation \(\mathbbm{d}^N\) to denote the modulus of continuity returned by
\(F(N)\), and we use the notation \(\mathbbm{f}^N\) to denote the approximation. Without
loss of generality, we restrict our attention to \(d(0, -)\) and \(f(0,
-)\)---that is, the $C^0$ modulus of continuity and approximation of \(S\). The
calculation for \(d(k, -)\) and \(f(k, -)\) for \(k > 0\) is virtually identical
conceptually. We simplify the notation of (B) above and  write $\mathbbm{d}^N(n)$ for $d(0,n)$ and 
$\mathbbm{f}^N(\vec{s},\vec{t})$ for $f(0,\vec{s},\vec{t})$.

\medskip

Let us first consider the modulus of continuity.  
The routine for computing \(\mathbbm{d}^N\)  depends on choosing a large number of numerical parameters:
\[\epsilon_n, e(n), s_n, k_n, l_n, P_N.\]
These have numerical dependencies that are generally of the form  \(a_n \gg b_n\) or \(a_n \ll b_n\).  It is routine that these can be satisfied in a 
primitive recursive manner--provided that the dependencies are consistent.  This is verified in Appendix \ref{NumPar} where it is shown that    the 
dependencies among these constants form a directed acyclic graph.
        
    The subroutine we describe next comes during the computation of $F(N)$, and hence we 
    may assume that we have the coefficients $k_{n}(N-1), l_n(N-1)$ already computed.  This computation was made during the  first $n$ steps of the computation of $F(N-1)$, but we neglect that recursion in this discussion.

For the inductive construction we note that:
\medskip

For each \(0 \leq m\leq n + 1\), make the following calculations, which recursively depend on
smaller \(m\). Specifically, $\mcw_m$ is built from $\mcw_{m-1}$ using the Substitution Lemma (Proposition \ref{lem:sl}) as described in section \ref{bldg the words}.  Then $h_m^s$ is built from the information in the  words in $\mcw_{m}$.  This allows $l_m$ to be chosen large enough that Numerical Requirement \hyperlink{nr5}{G} holds. This in turn defines $p_{m}$ and $q_m$ and  allows $\mcw^c_{m}$ to be built.

The algorithm is illustrated in Figure~\ref{second algo}.

\begin{enumerate}
    \item Using \(\la \epsilon_k : k \leq m \ra\) and \(\C W_{m-1}\), choose $k_m$ large enough to satisfy the Numerical Requirements \hyperlink{nr1}{C} and \hyperlink{nr2}{D} and apply the
        Substitution Lemma $m+1$ times to generate \(\C W_m\);
        \item Build $h_m$, smooth it to get $h_m^s$ and hence $H_m$.  Calculate $H_m$'s modulus of continuity.
        \item  Choose \(l_m\) sufficiently large that Numerical Requirements \hyperlink{nre}{E} and \hyperlink{nr5}{F} hold (with $n+1=m$).

     \item Build $\mcw_m^c$.
        
    \item Calculate the approximation and modulus of continuity corresponding to
        \(S_m\) using the methods of  Section~\ref{sssec:ecsn}.
        
    \item Continue until $m=n+1$ and  $d^{S_N}(n+1)$, the $n+1^{st}$ approximation to the modulus of continuity of $S_{n+1}$ is determined.
    \item Output $d(n+1)(n+1)=d^{S_N}(n+1)$.
    \end{enumerate}
At the end, using the modulus of continuity \(d\) corresponding to \(S_{n+1}\),
output \(\mathbbm{d}^N(n) = d(0,n+1)\) where $d(n+1)$ is the modulus of continuity of $S_{n+1}$.

To verify that this procedure actually yields a modulus of continuity for \(S\),
recall that by Numerical Requirement \hyperlink{nr5}{G}, Equation
\eqref{growth 2}, it follows that
    \begin{equation*}
        d^\infty(S_{n+1}, S) < 2^{-(n+1)}.
    \end{equation*}
By inequality~\ref{just the obvious},
    \[\max_{x\in\B T^2}d_{\bt^2}(S_{n+1}(x) , S(x))\le d^\infty(S_{n+1}, S)< 2^{-(n + 1)}.\]
Since \(d(n+1)\) yields the number of digits of input necessary to approximate
\(S_{n+1}\) to an accuracy of \(2^{-(n + 1)}\), it follows that the
approximation of \(S_{n+1}\) is \emph{itself} an approximation of \(S\) which,
given \(d(n+1)\) digits of binary input, is accurate to within \(2^{-n}\).

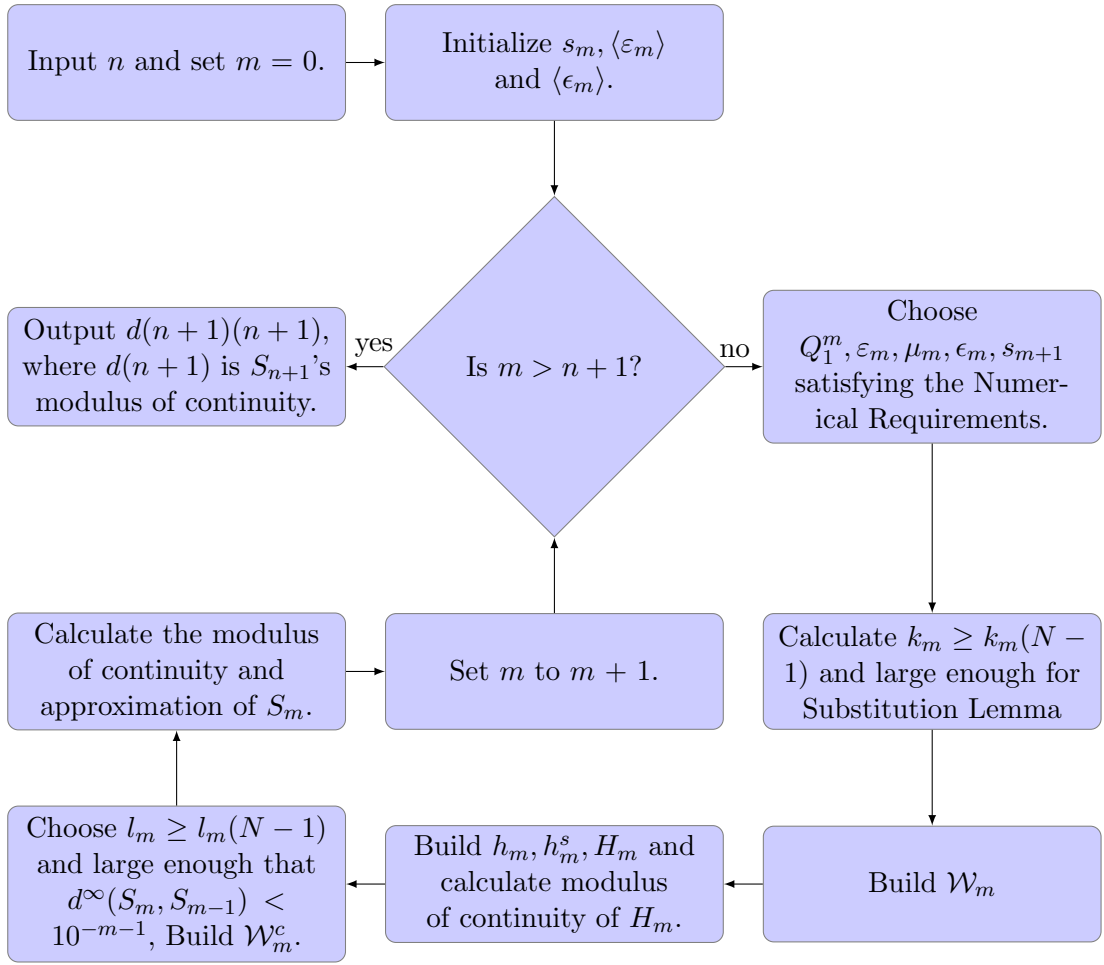
\begin{figure}
  \begin{tikzpicture}[
    decision/.style = {
      draw = black!50,
      diamond,
      fill = blue!20,
      text width = 11em,
      text badly centered,
      node distance = 3cm,
      inner sep = 0pt
    },
    block/.style = {
      draw = black!50,
      rectangle,
      fill = blue!20,
      text width = 11em,
      text centered,
      rounded corners,
      minimum height = 4em,
    },
    node distance = 3cm and 3cm,
    > = latex,
    auto
  ]
    \matrix[column sep = 0.5cm, row sep = 1cm]{
      \node[block] (start) {Input \(n\) and set \(m = 0\).};
        & \node[block] (init) {Initialize \(s_m, \la \varepsilon_m \ra\) and \(\la \epsilon_m \ra\).};
          & \\
      \node[block] (output) {Output \(d(n+1)(n+1)\), where \(d(n+1)\) is \(S_{n+1}\)'s modulus of continuity.};
        & \node[decision] (decision) {Is \(m > n + 1\)?}; 
          & \node[block] (sl) {Choose $Q^m_1, \varepsilon_m, \mu_m, \epsilon_m, s_{m+1}$ satisfying the Numerical Requirements.}; \\
      \node[block] (moc) {Calculate the modulus of continuity and approximation of 
      \(S_{m}\).};
        & \node[block] (increment) {Set \(m\) to \(m + 1\).};
          & \node[block] (l) {Calculate $k_m\ge k_m(N-1)$ and large enough for Substitution 
          Lemma};\\
      \node[block] (smooth) {Choose $l_m\ge l_m(N-1)$ and large enough that $d^\infty(S_{m}, S_{m-1})<10^{-m-1}$, Build $\mcw_m^c$.};
        & \node[block] (hn) {Build $h_m, h_m^s, H_m$ and calculate modulus of continuity of $H_m$.};
          & \node[block] (spin) {Build $\mcw_m$};\\
    };
  
    \draw[->] (start) -- (init);
    \draw[->] (init) -- (decision);
    \draw[->] (decision) -- node[near start] {no} (sl);
    \draw[->] (sl) -- (l);
    \draw[->] (l) -- (spin);
    \draw[->] (spin) -- (hn);
    \draw[->] (hn) -- (smooth);
    \draw[->] (increment) -- (decision);
    \draw[->] (decision) -- node[near start, swap] {yes} (output);
    \draw[->] (smooth) -- (moc);
    \draw[->] (moc) -- (increment);
  \end{tikzpicture}
\caption{
    The algorithm for calculating the modulus of continuity \(d^N\) of \(S\).
    This algorithm is easily altered to produce the approximation \(\mathbbm f^N\) of
    \(S\) simply by changing the output to \((f_0(\vec{x}), f_1(\vec{y}))\), where
    \(f(n+1) = (f_0, f_1)\) is \(S_{n+1}\)'s approximation.
}
\label{second algo}
\end{figure}

The approximation \(f\) for \(S\) is calculated almost identically, except
for the output. Given
    \begin{equation*}
        ([x]_{\mathbbm{d}^N(n)}, [y]_{\mathbbm{d}^N(n)}) \in (\{0,1\} \times \{0,1\})^{d(n)}
    \end{equation*}
the output is
    \begin{equation*}
       \mathbbm f^N(n+1)= \left( f_0([x]_{\mathbbm{d}^N(n)}), f_1([y]_{\mathbbm{d}^N(n)}) \right)
    \end{equation*}
where \(f(n+1) = (f_0, f_1)\) is the approximation of \(S_{n+1}\) produced in Step 5,
again in the notation that \([z]_m\) is a \(m\)-digit binary approximation of \(z\).

\appendix
\begin{center}
\underline{\bf Appendices}
\end{center}
\noindent Appendices B-D were originally written by Johann Gaebler. The author has  cleaned up exposition and done his best to fix various mistakes that appeared in Gaebler's original version.

\section{Numerical Parameters}
\label{NumPar}

\subsection{The Numerical Requirements Collected.}
 In this appendix we review the requirements on the numerical parameters used in the construction. Specifically, in constructing  the diffeomorphism $F(N)$ we build construction sequences $\la \mcw_n:n\in\nn\ra$, $\la \mcw_n^c:n\in\nn\ra$ that depend on $N$ and realize the corresponding circular system $\bk^c$ as a diffeomorphism.  These steps are intertwined--for example the circular system is built as a function of the sequence $\la k_n, l_n:n\in\nn\ra$.  In turn the $l_n$ are chosen as function of $\la \mcw_m:m\le n\ra$ in order to facilitate the smooth construction.
 To rigorously complete the proof we need to review all of these parameters and see that the inductive choices can be made consistently in a primitive recursive way.

At many stages in this paper we appeal to results from \cite{part3}.  Hidden in those appeals is a sequence of  parameters $\la \mu_n:n\in\nn\ra$ that is not explicitly mentioned in the construction presented here.  For this reason this review includes the inductive construction of $\la \mu_n:n\in\nn\ra$.

A substantial difference between this paper and earlier constructions is that the domain of the reductions in \cite{FRW} and 
\cite{part3} is the space of trees of finite sequences of natural numbers. The analogue in this 
paper is that the only trees  considered here are the trees of sequences 
$\la (0, 1, \dots n):n<\Omega\ra$ for $\Omega$ finite or infinite depending on the input $N$. We note that these trees are really ``stalks" and are finite or infinite depending on $\Omega$. 
Since the trees we consider in this paper are of this very special form, the requirements are easier to satisfy.  

Closely following section 10 of  \cite{part3} we begin with a review of the inductive requirements from \cite{FRW}. We give them in the notation of \cite{part3}. These inductive requirements are modified and simplified in the construction in the current manuscript.  We note the versions used in this paper.

\paragraph{Requirements that were instituted in \cite{FRW} and their modifications.} These requirements were dubbed \emph{Inherited Requirements} in \cite{part3}. Requirements that were new in \cite{part3} are called simply \emph{Numerical Requirements}, and requirements 
that explicitly mentioned in the text of the paper are labelled with capital letters A-F.  

Recall that  the number of elements of $\mcw_m$ is denoted $s_m$; 
 the numbers $Q^m_s$ and $C^m_s$ denote the
    number of classes and sizes of each class of $\mcq^m_s$ respectively. From the 
    construction in \cite{FRW} we have sequences $\la \epsilon_n:n\in\nn\ra$, 
$\la s_n,k_n, e(n):n\in\nn\ra$. 
    
\begin{description}
  
    \item[Inherited Requirement 1] $\la \epsilon_n:n\in\nn\ra$ is summable.
    
       \item[Inherited Requirement 2] $2^{e(n-1)}$ the number of $\mcq^n_{i+1}$ classes inside each
    $\mcq^n_i$ class. The numbers  $e(n)$ will be chosen to grow fast enough that 
    \begin{equation}
    2^n2^{-e(n)}<\epsilon_n \label{equ: crudite 2.1}
    \end{equation} 
    Similarly we set $C^n_n=2^{e(n)}$ as well. 
    
   {\bf Modification:}  In this paper the construction is simplified so to build $\mcw_{n+1}$ we  have exactly $n+2$-substitutions of each of size $2^{2e(n)}$.  Hence we can replace this requirement by the simple formula
   $s_n=2^{(n+1)e(n-1)}$. In particular  $s_m, Q^m_i$ and $C^m_i$ are all to be powers of 2.

    \item[Inherited Requirement 3]     	
    	For all $n$, 
	\begin{equation}
    	2\epsilon_{n}s_{n}^2<\epsilon_{{n-1}}\label{equ: crudite 1.1}
    	\end{equation}
    	
  \item[Inherited Requirement 4]  	
  \begin{equation}\epsilon_{n}k_ns_{{n-1}}^{-2}\to \infty\mbox{ as }n\to
    	\infty  \label{eqn: random prep2}
    	\end{equation}
    	
    \item[Inherited Requirement 5]	\item \begin{equation}
    	\prod_{n\in\nn}(1-\epsilon_n)>0 \label{equ: chisel2}
    	\end{equation}
	\emph{Comment:} Since this  is equivalent to the summability of the $\epsilon_n$-sequence, it is   redundant  and we will ignore in the rest of this paper
    	
    	\item[Inherited Requirement 6] 
	
	(Original Version) There will be  prime numbers $p_{i}$ such that $K_i=p_i^2s_{i-1}K_{i-1}$
    (i.e. $k_{i}=p_{i}^{2}s_{{i-1}}$). The $p_n$'s grow fast enough to allow the probabilistic arguments  in 
    	\cite{FRW} involving $k_n$ to go through.

	{\bf Modification} For all $n$, 
    	$k_{n}=P_N2^\ell s_{{n}}$, where for each $n$,  $\ell$ is large enough for the    	substitution argument involving $k_n$ to go through.
	
	\emph{Comment:} In \cite{FRW} $K_n$ was a product of a  sequence of prime 
numbers. The requirement 
on the sequences of prime numbers was that they were almost disjoint for different trees 
and that they grew sufficiently quickly.  In this paper $K_N$ is $P_N*2^\ell$ for a large $\ell$.  

 Since we have only one collection of very special 
trees the requirement simplifies to needing that the $\ell$ in the exponent grows sufficiently 
quickly for the Substitution Lemma (Proposition \ref{lem:sl}) argument to work.

	\item[Inherited Requirement 7] $s_n$ is a power of 2. 
	
	\emph{Comment:} This again is redundant as the modified Inherited Requirement 2 
	says directly  that $s_n=2^{(n+1)e(n-1)}$ 
		
	\item[Inherited Requirement 8] For all $n$,  $\epsilon_n<2^{-n}$.

\end{description}

\bigskip
\bfni{\underline{Numerical Requirements introduced in \cite{part3}}}
\medskip

\begin{description}
    \item[Numerical Requirement 1]
    $l_n>20*2^n$ and $1/l_{n-1}>\sum_{k=n} 1/l_k$.

    \item[Numerical Requirement 2]
    $\la \varepsilon_n:n\in\nn\ra$ is a  sequence of numbers in $\zoo$ such that $\varepsilon_N>4\sum_{n>N}\varepsilon_n$. 

    \item[Numerical Requirement 3] For all $n$, 
    \[\varepsilon_{n-1}>\\sup_{n<m}(1/q_m)\sum_{n\le k<m} 3\varepsilon_kq_{k+1}\] %

    \item[Numerical Requirement 4]
    $\mu_n$ is chosen sufficiently small  relative to $\min(\varepsilon_n, 1/Q^n_1)$. 
    
   {\bf Modification} Set $t_n=\min(\varepsilon_n, 1/Q^n_1)$ and take
    \[0<\mu_n<t_n\min_{k\le n}2^{-n-2}\left({1\over t_k}\right).\]
    Then for all $m$ 
    	\begin{equation*}
	t_m>\sum_{n=m}^\infty {\mu_n\over t_n}.
	\end{equation*}
    This is sufficient to carry out the various arguments, for example using the Borel-Cantelli Lemma.

    \item[Numerical Requirement 5]
    $\sum{|G^n_1|\over Q^n_1}<\infty$. 

    {\bf Modification} In this paper case $|G^n_1|\le 2$ so this becomes 
    \[\sum_n{1\over Q^n_1}<\infty.\]
    \emph{Comment:} Since $Q^n_1=e(n)$ and Requirement 2 implies that $2^{n-1}2^{-e(n)}\to 0$, this requirement is redundant in this paper.
    \item[Numerical Requirement 6]
     $l_n$ is big  enough relative to a lower   bound determined by $\la k_m, s_m:m\le n\ra$, $\la l_m:m<n\ra$ and $s_{n+1}$ to make the periodic approximations to the diffeomorphism $F(N)$ converge.

    \item[Numerical Requirement 7] 
    $s_n$ goes to infinity as $n$ goes to infinity and $s_{n+1}$ is a power of $s_n$.

	 \emph{Comment} Since $s_n$ is a power of $2^{e(n)}$ and $e(n)\to \infty$, this is trivial.
    \item[Numerical Requirement 8] $s_{n+1}\le s_n^{k_n}$.    
    \item[Numerical Requirement 9] The $\epsilon_n$'s are decreasing, $\epsilon_0<1/40$ and
    $\epsilon_n<\varepsilon_n.$

    \item[Numerical Requirement 10]
    $k_n$ is chosen  large enough relative to the lower bound determined by $s_{n+1}, \epsilon_n$ to apply the Substitution Lemma and construct the
    words in $\mcw_{n+1}$. Implicitly this requires that $1/k_n<\epsilon_n^3/4$.
    
    \emph{Comment:} This is essentially the same as Inherited Requirement 6.
    
    \item[Numerical Requirement 11]
     $\epsilon_n$ is small  relative to $\mu_n.$
    
    {\bf Modification} Remark 94 of \cite{part3} says that for  quantities $r(x,y), r(x,\mcc), f(x)$ 
    that 
    are determined by counting occurrences of $x, y$ in words in an alphabet $\mcl$ with $s$ 
    letters that have a given length $\ell$.  It says for all $\epsilon>0$ there is a $\delta$ such 
    that if for all $x,y\in \mcl$,
    \begin{eqnarray*}
    \left|{r(x,y)\over \ell}-{1\over s^2}\right|<\delta\\
    \end{eqnarray*}

{then} for all $x$:
\begin{equation*}
\left|{r(x,\mcc)\over f(x)}-{C\over s}\right|<\epsilon
\end{equation*}
From the proof of the lemma it is straightforward to find an explicit formula for 
$\delta(\epsilon, |C|,  \ell, s)$ for an upper bound on $\delta$.  The \emph{small relative clause} can be rephrased as asking that \[\epsilon_n<\delta(\mu_n, C^n_1, q_n, s_n).\]
Since $C^n_1=2^{-e(n)}s_n$, $\delta$ is really a function of $\mu, q_n, s_n$. 
    
    \item[Numerical Requirement 12]
    $\epsilon_0k_0>20$, the $\epsilon_nk_n$'s are increasing and\\ $\sum 1/\epsilon_nk_n<\infty$.

    \item[Numerical Requirement 13]
    The numbers $\epsilon_n$ should be small  enough, as a function of $Q^n_1$, that for all $w_0, w_1\in \mcw_{n+1}^c\cup \rev{\mcw_{n+1}^c}$ with $[w_0]_1\ne[w_1]_1$ the following inequality holds:
    \begin{equation*}\dbar(w_0\rest I_1, w_1\rest I_1)>(1- 2/Q^n_1)\gamma_n.
	\end{equation*}    
where the $\gamma_i$'s are defined inductively as:
	\begin{align*}
	\gamma_1&=(1-1/4-\epsilon_0)(1-1/\epsilon_0k_0)(1-1/l_0)\\
	\mbox{for }n\ge 2& \\
	\gamma_n&=\gamma_1\prod_{0<m<n} (1-10(1/k_m\epsilon_m+1/q_m+1/l_m+1/Q^m_1+\epsilon_{m-1}))
	\end{align*}
\end{description}

\bigskip
\bfni{\underline{Numerical Requirements introduced in this paper}}
\medskip

In this paper we have some supplemental numerical requirements.  We list only those that are not redundant given the requirements listed above. 
	\begin{description}
	\item[Numerical Requirement B] \hyperlink{nrb}{$k_n(N-1)\le k_n(N).$}
	\item[Numerical Requirement D]\hyperlink{nrd}{$1/k_n<\epsilon_n^3/100.$}
	\item[Numerical Requirement E] \hyperlink{nre}{$l_n(N-1)\le l_n(N).$}
	\item[Numerical Requirement F] \hyperlink{nr4}{$\sum{6^n\over k_n}<\infty$}. 
	\item[Numerical Requirement G] \hyperlink{nr5}{$d^\infty(S_{n+1}, S_n)<2^{-(n+1)}$}.
	
	\end{description}

 \subsection{Resolution}

\begin{center}
{\bf \underline{A list of parameters, their first appearances and their constraints}}
\end{center}

We classify the constraints on a given sequence according to whether they refer to other sequences or not. 

Requirements that inductively refer to the same sequence are straightforwardly consistent and can be satisfied with a primitive recursive construction. For example a requirement  that a certain inductively sequence involving a given variable be summable is  satisfied by asking that the $n^{th}$ sequence be less than $2^{-n}$. We call these \emph{absolute} conditions.

Those requirements that refer to other sequences risk the possibility of being circular and thus inconsistent or not being computable from the data in the other sequences.  We refer to these conditions  as \emph{dependent} conditions.
\begin{enumerate}
\item {\bf The sequence $\la k_n:n\in \nn\ra$.}  
\underline{Absolute conditions:}
    \begin{description}
    \item[A1] The sum $\sum_n 6^n/k_n$ is finite.
    \item[A2] $k_0=P_N$ and $k_n(N)\ge k_n(N-1)$
    \end{description}

\underline{Dependent conditions:}

    \begin{description}
    \item[D1] Numerical Requirement 10, is a lower bound for $k_n$ depends on $s_{n+1},\epsilon_n$, asking that $k_n$ be large enough for the word construction using the Substitution Lemma to work.
    
    \emph{Why is this primitive recursive?}    Given $s_{n+1}$ and $\epsilon_n$, the discussion in the proof of Lemma \ref{HoefPR} shows that a lower bound for $k_n$ can be given from Hoeffing's Inequality (Theorem \ref{Hoeffding's Inequality}) in a primitive recursing way.  So the only possible issue is circularity.
    
    \item[D2] Inherited Requirement 6.  In this context it says that $k_n=P_N2^\ell s_n$ for a large $\ell$. 
    
    \emph{Why is this primitive recursive?} $k_n$ is defined in equation \ref{kn def} where it explicitly is a multiple of  $s_{n}$. The size of $\ell$ is determined by {\bf D1}.

    \item[D3] \label{you got it} From Inherited Requirement 4, equation \ref{eqn: random prep2} requires that $\epsilon_{n}k_ns_{{n-1}}^{-2}$ 
    goes to $\infty$ as $n$ goes to $\infty$. 
    
    \emph{Why is this primitive recursive?:} This can be satisfied primitively recursively by choosing $k_n$ to be an integer larger than ${s_n^2\over \epsilon_n}$.
    
    We note that equation \ref{eqn: random prep2} implies that $\sum 1/\epsilon_nk_n$ is finite.
     \item[D4] Numerical Requirement 8 implies that $k_n$ is large enough that $s_{n+1}\le s_n^{k_n}$. 
     
      \emph{Comment:} This is easily satisfied by taking $k_n\ge {log(s_{n+1})\over log (s_n)}$.

	\item[D5] Numerical Requirement D says $1/k_n<\epsilon_n^3/100$. 
	\emph{Comment:} As long as $\epsilon_n$ is defined before $k_n$, Requirement D is immediate by taking $k_n>4/\epsilon_n^3$.
    \item[D6] Numerical Requirement 12 says that $\epsilon_0k_0>20$ and the $\epsilon_nk_n$'s are increasing and $\sum 1/\epsilon_nk_n$ is finite. 
        
    \emph{Why is this primitive recursive?} As noted the last condition follows from D3. The other parts of Numerical Requirement 12 are satisfied primitive recursively by taking $k_n$ to be an integer at least ${n\over \epsilon_n}$.

    \end{description}

From D1-D5, we see that $k_n$ is dependent on the choices of $\la k_m, l_m:m<n\ra, \la s_m:m\le n+1\ra$, and  $\epsilon_n$, and these dependencies can be satisfied primitively recursively.

\item {\bf The sequence $\la l_n:n\in\nn\ra$.}

\underline{Absolute conditions}
    \begin{description}
    \item[A3] Numerical Requirement E: $l_n(N)\ge l_n(N-1)$.
    \item[A4] Numerical Requirement 1 says that  $1/l_n>\sum_{k=n+1}^\infty 1/l_k$. We also require that $l_n>20*2^n$, an exogenous requirement.
    \end{description}

\underline{Dependent conditions}

    \begin{description}
    \item[D7] By Numerical Requirement 6, $l_n$ is bigger than a number determined by  $\la k_m, s_m:m\le n\ra, \la l_m:m<n\ra$ and 
    $s_{n+1}$. This is superseded by the more explicit Numerical Requirement F says that 
    $d^\infty(S_{n+1}, S_n)<2^{-(n+1)}.$
    \end{description}
 Thus $l_n$ depends on $\la k_m, s_m: m\le n\ra$, $\la l_n:m<n\ra$ and $s_{n+1}$. 
  
  \emph{Why is this primitive recursive?} The $\|\ \|_\infty$-norm of $S\circ T$ can be computed effectively from the $\|\ \|_\infty$-norms of $S$ and $T$.  In particular there is a primitively recursively computable real number $M$ such that 
  \begin{align*}
  d^\infty(S_{n+1}, S_n)&< M|\alpha_{n+1}-\alpha_n|\\
    &\le{M\over q_{n+1}}.
  \end{align*}
  and the latter inequality is  from equation \ref{alphans}. 
Since $q_{n+1}=k_nl_nq_n^2$ we get an explicit lower bound on $l_n$.

\item {\bf The sequences $\la s_n:n\in\nn\ra$ and $\la e(n):n\in\nn\ra$.}
We treat these sequences as equivalent since $s_n=2^{(n+1)e(n-1)}$.

\underline{Absolute conditions} 
\begin{description}
\item[A5] Inherited Requirement 7 says that $s_n$ is a power of 2.
\end{description}
\begin{description}
\item[A6] The sequence $s_n$ goes to infinity.
\item[A7] $s_{n+1}$ is a multiple of $s_n$.
\end{description}

\underline{Dependent conditions}

\begin{description}

\item[D8] The function $e(n):\nn\to \nn$ referred to in equation \ref{equ: crudite 2.1} gives the number of $\mcq^n_{s+1}$ classes inside each $\mcq^n_s$ class.  It has the dependent requirement that $2^n2^{-e(n)}<\epsilon_{n-1}$.  Moreover $s_n=2^{(n+1)e(n)}$.

\end{description}

The result is that $s_{n+1}$ depends on the first $n+1$ elements of $\mct$ $\la k_m, s_m, l_m:m<n\ra$, $s_n$,  and $\epsilon_n$.\footnote{It is important to observe that the choice of $s_{n+1}$ does \emph{not} depend on $k_n$ or $l_n$.} 

	\emph{Why is primitive recursive?} The only requirement for choosing $s_{n+1}$ is that 
	\[2^{-e(n+1)}<\epsilon_n2^{-n}\]
	and this is clearly primitively recursively satisfiable.
\item {\bf The sequence $\la \epsilon_n:n\in\nn\ra$.}

\underline{Absolute conditions}
    \begin{description}
    \item[A8] Numerical Requirement 9 and Inherited Requirement 1  say that the $\la \epsilon_n:n\in\nn\ra$ is decreasing and summable and $\epsilon_0<1/40$.
    \item[A9]  Inherited Requirement 8 says that  $\epsilon_n<2^{-n}$
    \end{description}

\underline{Dependent conditions}

    \begin{description}
    \item[D9] Numerical Requirement 9 says $\epsilon_n<\varepsilon_n$. 
    \item[D10] Equation \ref{equ: crudite 1.1} of Inherited Requirement 3 says $2\epsilon_{n}s_{n}^2<\epsilon_{n-1}$
    \item[D11] Numerical Requirement 11 says that $\epsilon_n$ must be small enough relative to $\mu_n$.
    \item[D12] Numerical Requirement 13 says that $\epsilon_n$ is small as a function of $Q^n_1$.
    \end{description}

   The result is that $\epsilon_n$ depends exogenously on the first $n$ elements of $\mct$, and on $Q^n_1, s_n$, $\varepsilon_n$, $\epsilon_{n-1}$ and $\mu_n$.

\emph{Why is this primitive recursive?} The only issue might be D10, but this is explicitly solved in Numerical Requirement 11, which describes how to calculate an explicit function 
$\delta(\mu_n,  q_n,  s_n)$ such that Numerical Requirement 11 holds if $\epsilon_n<\delta(\mu_n, q_n, s_n)$.

\item {\bf The sequence $\la \varepsilon_n:n\in\nn\ra$. }

\underline{Absolute conditions}
\begin{description}
\item[A10] Because the sequence $\la q_n:n\in\nn\ra$ is increasing Numerical Requirement 3 is satisfied if  $\varepsilon_{n-1}>4\sum_{k\ge n}\varepsilon_n$, a rephrasing of Numerical Requirement 2.  This is an absolute condition and implies that $\la \varepsilon_n:n\in\nn\ra$ is decreasing and summable.
\end{description}

\underline{Dependent conditions}

\begin{description}

\item[D13] $\la k_n\varepsilon_n:n\in\nn\ra$ goes to infinity.  This already follows from the fact that $\epsilon_n<\varepsilon_n$ and item D4 .
\end{description}

Since item D12 follows from item D4, all of the requirements on $\la\varepsilon_n:n\in\nn\ra$ are absolute or follow from previously resolved dependencies. Moreover they are trivial to satisfy primitively recursively.

\item {\bf The sequence $\la Q^n_1:n\in\nn\ra$.}

 Recall $Q^n_1$ is the number of equivalence classes in $\mcq^n_1$. 
We require:

\underline{Absolute conditions}
\begin{description}
\item[A11] The only requirement on the choice of $Q^n_1$ not accounted for by the choices of the other coefficients is that $\sum 1/Q^n_1<\infty.$

\end{description}

\underline{Dependent conditions}

None

\item {\bf The sequence $\la \mu_n:n\in\nn\ra$.} 

This sequence gives the required pseudo-randomness in the timing assumptions. 

\underline{Absolute conditions}

None.

\underline{Dependent conditions}

\begin{description}
\item[D14] Numerical Requirement 4 appearing in this paper is written explicitly as follows:
Set $t_n=\min(\varepsilon_n, 1/Q^n_1)$ and take
    \[0<\mu_n<t_n\min_{k\le n}2^{-n-2}\left({1\over t_k}\right).\]
    Then for all $m$ 
    	\begin{equation}
	t_m>\sum_{n=m}^\infty {\mu_n\over t_n}.\notag
	\end{equation}

\end{description}

\end{enumerate}

\noindent The recursive dependencies of the various coefficients are summarized in Figure \ref{ecosystem}, in which an arrow from a coefficient to another coefficient shows that the latter is dependent on the former. Here is the order the the coefficients can be chosen consistently.
\medskip

\bfni{Assume:}
\begin{quotation}
\noindent The coefficient sequences $\la k_m, l_m, Q^m_1, \mu_m, \epsilon_m, \varepsilon_m:m<n\ra$ and $s_{n}$ have been chosen. 
\end{quotation}
\medskip

\bfni{To do:}
\begin{quotation}\noindent Choose $k_n, l_n, Q^n_1, \mu_n, \epsilon_n, \varepsilon_n$ and $s_{n+1}$. Each requirement is to choose the corresponding variable \emph{large enough} or \emph{small enough} where these adjectives are determined by the dependencies outlined above. 
\end{quotation}
\begin{figure}
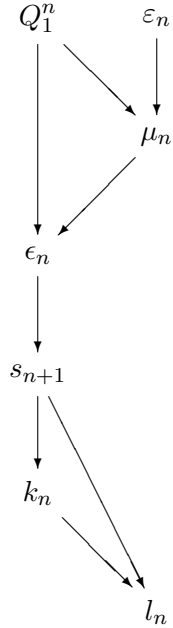

\centering

\[\begin{diagram}
\node{Q^n_1}\arrow[2]{s}\arrow{se}\node{\varepsilon_n}\arrow{s,r}{}\\
\node{}\node{\mu_n}\arrow{sw,r}{}\\
\node{\epsilon_n}\arrow{s}\node{}\\
\node{s_{n+1}}\arrow{s}\arrow{sse}\node{}\\
\node{{k_n}}\arrow{se}\node{}\\
\node{}\node{l_n}
\end{diagram}\]

\caption{Order of choice of Numerical parameters dependency diagram.}
\label{ecosystem}
\end{figure}

Figure \ref{ecosystem} gives an order to consistently choose the next elements on the sequences; Choose the successor coefficients  in the following order:

\[Q^n_1, \varepsilon_n, \mu_n,\epsilon_n, s_{n+1}, k_n, l_n.\]
\pagebreak

\section{Logical Background}
\label{app:LB}

\subsection{Logical Basics}

What follows is a very brief and relatively informal introduction to basic
first-order logic and the formal language \(\C L_{\PA}\), the language of
\emph{first-order} (or \emph{Peano}) \emph{arithmetic}. For a more thorough
treatment, see \cite{ENDERTON}.

\subsubsection{The language \(\C L_{\PA}\)}

The \emph{language of first-order arithmetic}, \(\C L_{\PA}\), is made up of the
following pieces:
    \begin{enumerate}
    	\item It has variables $x_0, x_1, \dots x_n, \dots $, two constant symbols $0,1$, a relation symbol $<$ and two function symbols $+, *$.
        \item \textbf{Terms:} Terms are expressions in \(\C L_{\PA}\) that are
            intended to correspond to actual objects (i.e., numbers). The most
            basic terms are \emph{variables} (which we frequently  denote with the informal
            symbols ``\(x\)'', ``\(y\)'', ``\(z\)'', etc.) and \emph{constants},
            namely ``0'' and ``1''. New terms can be built inductively using old terms
            by applying \emph{function} symbols to them.  Formally: if $\tau_1\dots \tau_n$ are terms and 
            $f$ is an $n$-place function symbol then $f\tau_1\dots \tau_n$ is a term. Since our function symbols are binary the new term would either be $+\tau_1\tau_2$  or $*\tau_1\tau_2$.
            
            It is necessary to prove unique readability for terms--that from the sequence of symbols constituting the term 
            one can uniquely recapture its inductive construction.  Frequently one drops the formal definition and inserts parenthesis to make the terms humanly readable. For example instead of writing $+x1$ for the term, we might right $(x+1)$.   Or instead of writing ${*}+x1+x2$ one might write $((x+1){*}(x+2))$.  However the first is the formally correct string of symbols constituting the term. 
            
           Thus,
                \begin{equation*}
                    ``x", ``y"
                \end{equation*}
            are terms, as is
                \begin{equation*}
                ``((x+y)+1)* y"
                \end{equation*}
                For $n$ a natural number, we often write $\underline{n}$ as though it were a term, but this is an abbreviation for the term  \begin{equation*}
            \overbrace{1 + 1 + \cdots + 1}^{\text{\(n\)-times}}
        \end{equation*} (with parenthesis suitably added).
                So $\underline{n}$ is a term in the language \(\C L_{\PA}\), \emph{not} a natural number
    in \(\B N\).)

               It is important to note that all polynomials with natural number 
               coefficients can be expressed as terms.
        \item \textbf{Atomic formulas:} An atomic formula is an expression of
            the form
                \begin{equation*}
                    ``t_1 = t_2"
                \end{equation*}
            or
                \begin{equation*}
                    ``t_1 < t_2''
                \end{equation*}
            where \(t_1\) and \(t_2\) are terms.
        
        \item \textbf{Compound formulas:} A \emph{compound formula}---also
            called a \emph{well-formed formula} or, more simply, a
            \emph{formula}---is defined recursively as follows:
                \begin{enumerate}
                    \item An atomic formula is a formula.
                    \item If \(F_1\) and \(F_2\) are formulas, then ``\(\lnot F_1\)'', ``\(F_1 \land F_2\)'', ``\(F_1
                        \lor F_2\)'', ``\(F_1 \rightarrow F_2\)'', etc., are
                        formulas. (Here, ``\(\land\)'' represents \emph{and},
                        ``\(\lor\)'' represents \emph{(inclusive) or},
                        ``\(\rightarrow\)'' represents \emph{implies}, and so
                        on.)
                    \item If \(F\) is a formula, then ``\((\forall u) F\)'' and
                        ``\((\exists u) F\)'' are formulas, where \(u\) here is
                        any variable.  (These quantifiers have their usual
                        meanings: ``\(\forall\)'' can be read ``for all'' and
                        ``\(\exists\)'' can be read ``there exists.'')
                \end{enumerate}
    \end{enumerate}
A variable is \emph{free} if it doesn't appear inside the scope of a quantifier;
otherwise, a variable is \emph{bound}. To illustrate, the variable \(x\) is free in the formula
``\(x=x\)'' but bound in the formula ``\((\forall x)(x = x)\)''.

A \emph{sentence} is a formula with no free variables. For instance, ``\(x + y =
y + x\)'' is a well-formed formula, but not a sentence. However, its universal
closure ``\((\forall x)(\forall y)(x + y = y + x)\)'' \emph{is} a sentence.

Whether the variable \(x\) is ``really'' universal---i.e., within the scope of
the quantifier \((\forall x)\)---or existential---i.e., within the scope of the
quantifier \((\exists x)\)---can be a slightly subtle matter. For instance, the
following two formulas are equivalent:

    \begin{equation*}
        ``[(\forall x) (y * x = y)] \rightarrow (0 = 1)"
    \end{equation*}
and
    \begin{equation*}
        ``[(\exists x)\neg(y * x = y)] \lor (0 = 1)"
    \end{equation*}
Note that these are both equivalent to the formula
 \begin{equation*}
        ``(\exists x)[\neg(y * x = y) \lor (0 = 1)]"
    \end{equation*}
    In the last formula the quantifier is ahead of all Boolean Combinations.
\subsubsection{Bounded quantifiers}

It is convenient to think of \(\C L_{\PA}\) as having two different kinds of
quantifiers: \emph{unbounded quantifiers}---e.g., \((\forall x, y, z)\)---and
\emph{bounded quantifiers}---e.g., \((\exists x < \underline{2})\). %
    Morally, bounded and unbounded quantifiers are different because unbounded
quantifiers \emph{increase} the logical complexity of a formula, whereas bounded
quantifiers leave the logical complexity of a formula the same.

The clearest manifestation of this is the fact that the order of
\emph{unbounded} quantifiers cannot be changed without changing the meaning of a
formula, while, in general, \emph{bounded} quantifiers can be pushed past
unbounded quantifiers, possibly at the cost of introducing additional unbounded
quantifiers, but without altering the truth or falsity of the formula in
question. This is a consequence of working in the setting of Peano Arithmetic.

For instance,
    \begin{equation*}
        ``(\forall x)(\exists y)(x = 0 \lor x = y + 1)"
    \end{equation*}
is a true statement about the natural numbers, whereas
    \begin{equation*}
        ``(\exists y)(\forall x)(x = 0 \lor x = y + 1)"
    \end{equation*}
is a false statement. However,
    \begin{equation*}
        ``(\forall x < \underline{10})(\exists y)(x = 0 \lor x = y + 1)"
    \end{equation*}
is equivalent to
    \begin{equation*}
        ``(\exists z)(\forall x < \underline{10})(\exists y < z)(x = 0 \lor x =
        y + 1)",
    \end{equation*}
which has the unbounded quantifier in the outermost position. In this example the two sentences are trivially equivalent because they are both true.  However this quantifier interchange would preserve true no matter what is inside the scope of the quantifiers.

Formulas containing \emph{only} bounded quantifiers are known as
\(\Delta_0^0\)-formulas. A \(\Delta^0_0\)-formula can be rewritten as a Boolean
combination of polynomials. For instance
    \begin{equation*}
        ``(\forall x < \underline{3})(\exists y < \underline{3})( x + y = z)"
    \end{equation*}
can be rewritten in the less compact but equivalent form
    \begin{align*}
        ``
            &\left[(0 + 0 = z) \lor (0 + 1 = z) \lor (0
                + z = z)\right]\\
            &\land \left[(1 + 0 = z) \lor (1 + 1 = z)
                \lor (1 + z = z)\right]\\
            &\land \left[(z + 0 = z) \lor
                (z + 1 = z) \lor (z +
                z = z)\right]''.
    \end{align*}

\subsubsection{Formula complexity}

As was the case for the expression
    \begin{equation*}
      ``[(\forall x) (y * x = y)] \rightarrow (0 = 1)"
    \end{equation*}
every formula in \(\C L_{\PA}\) can be rewritten in an equivalent form where all
of the quantifiers are at the beginning. As was noted in the previous section,
we can also ensure that any bounded quantifiers appear \emph{after} all of the
unbounded quantifiers. Thus, we can ensure that every formula has the following
bipartite structure: (1) alternating blocks of unbounded ``\(\forall\)'' and
``\(\exists\)'' quantifiers (not necessarily in that order); (2) a
\(\Delta^0_0\)-formula afterward. For instance,
    \begin{equation*}
        ``\overbrace{(\forall x)(\forall y)(\exists z) \ldots (\forall
        w)}^{\text{unbounded quantifiers}}\overbrace{(\exists x' <
        \underline{n}) \ldots (\forall y' < [z' + \underline{m}])(x = y + z \lor
        \ldots \lor (\pri(z') \land y ' \mid x)}^{\text{\(\Delta^0_0\) formula}}"
    \end{equation*}
The \emph{logical complexity} of a given expression is governed by the block of
unbounded quantifiers at its beginning when written in this way. Expressions are
classified according to (1) the outermost quantifier (i.e., ``\(\exists\)'' or
``\(\forall\)'') and (2) the number of blocks of quantifiers of the same type. A
formula that begins with a universal quantifier (``\(\forall\)'') and is
followed by \(n\) alternating blocks of quantifiers is a \(\Pi^0_n\)
formula;%
\footnote{
    The superscript 0 in \(\Pi^0_1\) refers to the fact that only
    \emph{first-order} quantifiers are allowed; that is, quantifiers may only
    make assertions about objects in the structure referred to, rather than subsets of the objects in a structure. If we allow
    expressions about sets of objects (e.g., ``There exists a set of numbers
    such that every element except the first is double the previous element.'')
    then we have moved beyond the so-called \emph{arithmetic} formulas into
    second-order arithmetic.} a formula beginning with an existential quantifier (``\(\exists\)'') and
followed by \(n\) alternating blocks of quantifiers is a \(\Sigma^0_n\)-formula.
Thus,
    \begin{equation*}
        ``(\forall x)(\exists y)(x = 0 \lor x = y + 1)"
    \end{equation*}
is a \(\Pi^0_2\) formula, as is
    \begin{equation*}
        ``(\forall x)(\exists y)(\exists z)(x = 0 \lor x = z + y + 1)".
    \end{equation*}
However, 
    \begin{equation*}
        ``(\exists y)(y = \underline{2})"
    \end{equation*}
is a \(\Sigma^0_1\) formula, and
    \begin{equation*}
        \underline{2} = 1 * \underline{3}
    \end{equation*}
and
    \begin{equation*}
        x = 0    
    \end{equation*}
are (possibly false) \(\Delta^0_0\)-formulas. (A formula is \(\Delta^0_n\) if it
is \emph{both} \(\Pi^0_n\) and \(\Sigma^0_n\). Note that this agrees with our
earlier definition of \(\Delta^0_0\)-formulas.)

\subsubsection{\(\Pi^0_1\)-formulas}
\label{GC}

A particularly interesting class of formulas is those that live at the second
level of the hierarchy, i.e., the \(\Pi^0_1\)-formulas. A large number of
statements in elementary number theory take this form. For instance, Goldbach's
Conjecture may be stated as
    \begin{equation*}
        ``(\forall x)(\exists y, z < x)(\underline{2} \mid x \rightarrow
        (\pri(y) \land \pri(z) \land x = y + z))".
    \end{equation*}
(I.e., ``For all even \(x\in\B N\), there exist primes \(y\) and \(z\) such that
\(x = y + z\). Here, ``\(\pri(x)\)'' is an abbreviation for the
\(\Delta^0_0\)-formula ``\((\forall y < x)(y \mid x \rightarrow y = 1)\)'' and
``\(y \mid x\)'' is an abbreviation for the \(\Delta^0_0\)-formula ``\((\exists
z < x) (y * z = x)\)''.)

Moreover, formalizations of consistency statements, such as ``\(\PA\) does not
prove that \(0 = 1\)'' or ``\(\ZFC\) does not prove that \(0 = 1\)'' are
likewise \(\Pi^0_1\). The reason is straightforward: informally, \(\Pi^0_1\)
sentences express ``for all'' statements where what comes after the ``for all''
is something that for any given number \(n\) can be verified with finite
resources. The statement ``\(\PA\) does not prove that \(0 = 1\)'' can be
understood as ``All valid proofs from the first order axioms of  \(\PA\) do not prove \(0 = 1\)'' or,
equivalently, ``There is no valid proof in \(\PA\) of the fact that \(0 = 1\)''.
Since proofs can be coded as numbers through the use of G\"odel numbering, and
for any given proof, it is possible to tell using finite resources whether it
proves a contradiction such as ``\(0 = 1\)'', these statements can be expressed
using \(\Pi^0_1\) formulas in \(\C L_{\PA}\).

The question of whether a formula is \(\Pi^0_1\) is sometimes quite subtle. The
Riemann hypothesis (\(\RH\)), for instance, can be expressed in the form of a \(\Pi^0_1\)
formula. This result was first shown in \cite{DMR}, and a simpler expression was
given in \cite{lag}. While fully unpacking its expression in \(\C L_{\PA}\) is
outside the scope of this appendix,\footnote{
    Naively, the expression depends upon the logarithm and exponential
    functions, as well as rational numbers. Although we do not do so, it is
    possible to adequately ``capture'' these notions in plain arithmetic. The
    approach generally used in reverse mathematics---whose reliance on second
    order concepts is, in this case, superficial---is adequate to the task; see
    \cite{SIMPSON}.
} the \(\Pi^0_1\) form of \(\RH\) given in \cite{lag} can be seen from the following paraphrase:
    \begin{equation*}
        ``(\forall n) \left(\sum_{d \mid n} d \leq H_n + \exp(H_n)\cdot
        \log(H_n)\right)",
    \end{equation*}
where \(H_n = \sum_{j = 1}^n \tfrac 1j\).

\subsubsection{Truth}

Formalizing the notion of truth is complicated. Tarski proved a theorem (\emph{Tarski's theorem}), that there can
be no first order formula in \(\C L_{\PA}\) that completely captures what it means for a
sentence in \(\C L_{\PA}\) to be true. That is, if we denote the G\"odel number
of a formula \(\phi\) by \(\ulcorner \phi \urcorner\), there is no formula
\(\Tr\) for which it is always true that
    \begin{equation*}
        \Tr(\ulcorner \phi \urcorner) \leftrightarrow \phi.
    \end{equation*}
In logical jargon, a formula that holds of the G\"odel numbers of sentences that are true and belong to some class $\Gamma$ is called a \emph{truth predicate} for $\Gamma$. 

The difficulty lurking behind \emph{Tarski's Theorem} lies in the fact that \emph{truth} is a notion that is highly
dependent on logical complexity. There \emph{is} a truth predicate for
\(\Delta^0_0\)-formulas, but its expression is itself a \(\Delta^0_0\)-formula;
likewise, we can construct truth predicates for \(\Pi^0_n\)- and
\(\Sigma^0_n\)-sentences which will themselves be \(\Pi^0_n\)- and
\(\Sigma^0_n\)-sentences.
    For this reason, a \emph{general} truth predicate would seem to have to be
    an element of \(\Delta^0_n\) for all \(n \in \B N\), i.e., have
    \emph{infinite} logical complexity.

Since truth for \(\Delta^0_0\)-sentences can be defined in a \(\Delta^0_0\)
way---that is, the truth of a \(\Delta^0_0\)-sentence can be ascertained using
\emph{only} finite resources---and \(\Pi^0_1\)-sentences are of the form
``\((\forall x, y, \ldots, z) F(x, y, \ldots, z)\)'', where ``\(F(x, y, \ldots,
z)\)'' is a \(\Delta^0_0\)-formula, we can define truth for \(\Pi^0_1\)-sentences
in a \(\Pi^0_1\) way.

\begin{definition}[\(\Delta^0_0\) Truth]
    Suppose ``\(F\)'' is a \(\Delta_0^0\)-sentence. Then ``\(F\)'' is equivalent
    to a Boolean combination of equalities and inequalities of polynomials in
    \(0\) and \(1\). The formula ``\(F\)'' is \emph{true} if and only if the
    Boolean combination is actually true in \(\B N\).
\end{definition}

\begin{definition}[\(\Pi^0_1\) Truth]
    The \(\Pi^0_0\)-sentence ``\((\forall x\forall y \ldots\forall z) F(x, y, \ldots,
    z)\)'' is true if and only if for all \(n_0, n_1, \ldots, n_k \in \B N\),
    the \(\Delta^0_0\)-sentence ``\(F(\underline{n_0}, \underline{n_1}, \ldots,
    \underline{n_k})\)'' is true.
\end{definition}

\subsection{Computability Theory}
\label{app:LB:CT}

Computability or recursion theory\footnote{
    As previously noted, we use the terms ``computable'' and ``recursive''
    synonymously.
} is the subdomain of logic concerned with the comparative difficulty of various
problems from the perspective of an idealized computer, such as a Turing machine
with an infinitely long tape. It turns out that a large number of abstract
models of computation coincide with one another and are captured by the notion
of a \emph{recursive} function. The celebrated \emph{Church's Thesis} is the meta-claim that \emph{all} general notions a inherently finite computability are equivalent. The general reference for this section is
\cite{ODIFREDDI}.

\subsubsection{Primitive recursion}

There are a large number of equivalent definitions of the primitive recursive
functions. The most straightforward is to define a set of \emph{basic} primitive
recursive functions and then to define the primitive recursive functions as
their closure under certain operations. That is:
    \begin{itemize}
        \item \textbf{Initial primitive recursive functions:} The following
            functions mapping \(\B N^k\) to \(\B N\) for some \(k \geq 1\) are
            all primitive recursive:
                \begin{itemize}
                    \item \textbf{Zero map:} The function \(\C O(x) = 0\).
                    \item \textbf{Successor map:} The function \(S(x) = x + 1\).
                    \item \textbf{Projection maps:} The functions \(\pi_k(x_0,
                        x_1, \ldots, x_k, \ldots, x_n) = x_k\).
                \end{itemize}

        \item \textbf{Composition:} The primitive recursive functions are closed
            under composition; that is, if \(f:\B N^n \to \B N\) is primitive
            recursive, and \(g_1,\ldots,g_n: \B N^m \to \B N\) are primitive
            recursive, then the function \(h: \B N^m \to \B N\) given by
                \begin{equation*}
                    h(\vec x) = f(g_1(\vec x), \ldots, g_n(\vec x))
                \end{equation*}
            is primitive recursive.

        \item \textbf{Primitive recursion:} If \(g: \B N^n \to \B N\) is
            primitive recursive and \(h: \B N^{n + 2}\) is primitive recursive,
            then the unique function \(f: \B N^{n+1} \to \B N\) satisfying
                \begin{align*}
                    f(\vec x, 0) &= g(\vec x)\\
                    f(\vec x, y + 1) &= h(\vec x, y, f(\vec x, y))
                \end{align*}
            is primitive recursive.
    \end{itemize}

A wide variety of familiar ``computable'' procedures are primtive recursive,
such as addition, subtraction, powers, and logarithms. Even comparatively
``expensive'' operations, such as calculating the prime factorization of a given
integer or solving an arbitrary 3-SAT problem, are primitive recursive. This is
because solving these problems does not involve a possibly ``unbounded'' search:
one can delimit in advance a {\emph{bounded}} set of values the function can take.
Informally, the set of primitive recursive functions consists of all functions
which can be implemented in a standard programming language using only FOR
loops and not WHILE loops.

What if one needs WHILE loops?  The broader class of \emph{recursive} functions
consists of all problems which could in principle be solved by a computer like a
Turing machine. Unlike primitive recursive functions, recursive functions can
involve \emph{unbounded} searches, and, indeed, it is not possible in general to
determine whether or not a recursive function is total (i.e., produces output
for any given input).

More precisely, the set of \emph{(partial) recursive functions} is the smallest
set of functions containing the primitive recursive functions and closed under
\emph{\(\mu\)-recursion}; that is, given a recursive function \(g : \B N^{n+1}
\to \B N\), the function \(f : \B N^n \to \B N\) given by
    \begin{equation*}
        f(\vec x) = \text{ the least \(z \in \B N\) such that for all \(y \leq
        z\), \(g(\vec x, y)\) is defined and \(g(\vec x, z) = 0\)},
    \end{equation*}
is also recursive.

\subsubsection{Computable real functions}
\label{app:LB:CRF1}

It is possible to extend the notion of a recursive function from the discrete
domain \(\B N\) to a general continuous framework. For a more thorough
introduction to the theory of computable continuous functions, see
\cite{WEIHRAUCH}.

The primary obstacle to a theory of computable real functions is the observation
that uncountable metric spaces, such as \([0,1] \subseteq \B R\), are inherently
``incomputable'' objects: there is no computer program that can enumerate all
\(x \in [0,1]\).

However, \([0,1]\) does have a computable \emph{presentation}, in the following
sense: the rationals are an inherently effective object, and we can imagine a
recursive function \(n \mapsto q_n\) mapping \(n\) to the \(n\)-th rational in
\([0,1]\) in an effective enumeration.\footnote{
    Strictly speaking, there is no such thing as a recursive map from \(\B N\)
    to \(\B Q\) as we have defined the term ``recursive.'' However, one can
    ``build'' the rational numbers \(\B Q\) ``inside'' the natural numbers \(\B
    N\) in any of a variety of ways. For instance, pairs of integers can be
    ``coded'' by representing \((m, n)\) using the single integer \((m + n)^n +
    m\). The rational number \(\tfrac p q\) can be represented by the pair \((p,
    q)\), etc.
} Then, the real numbers \(\B R\) can be ``presented'' as rapidly convergent
sequences of the form \(\la x_0, x_1, x_2, \ldots \ra\), where \(x_i \in \B Q\)
and \(|x_i - x_{i + 1}| < 2^{-(n+1)}\).

Since real numbers are given by rapidly converging sequences, \emph{computable}
functions on the real numbers should modify the sequences in a recursive way to
produce a new rapidly converging sequence. Intuitively, given a more and more
accurate representation of the real number \(x\)---that is, the rapidly
converging sequence of rational numbers \(\la x _n \ra\)---the function \(f\)
should output more and more accurate information about the real number \(f(x)\)
by producing a new rapidly converging sequence of rational numbers.

This process is most easily envisioned by means of a partial recursive function
\(f : \B Q^{<\B N} \to \B Q\) mapping finite sequences of rational numbers to
rational numbers. Using the real number \(x\) represented by the sequence
\begin{equation*}
        \la q_{i(0)}, q_{i(1)}, q_{i(2)}, \ldots \ra,
    \end{equation*}
one can imagine successively feeding \(f\) the inputs
    \begin{align*}
        \vec q_0 & = \la q_{i(0)} \ra\ \\
        \vec q_1 & = \la q_{i(0)}, q_{i(1)} \ra \\
        \vec q_2 & = \la q_{i(0)}, q_{i(1)}, q_{i(2)} \ra\\
        \vdots\ &
    \end{align*}
Let \(d(i) : \B N \to \B N\) enumerate the \(q_i\) for which \(f\) returns a
value. Then, \(f(x)\) is the sequence
    \begin{equation*}
        \la f(\vec q_{d(0)}), f(\vec q_{d(1)}), f(\vec q_{d(2)}), \ldots \ra.
    \end{equation*}
(Here, we assume that it does indeed hold that \(|f(\vec q_{d(i)}) - f(\vec
q_{d(i + 1)})| < 2^{-(i + 1)}\); otherwise, the value of \(f\) at \(\la
q_{i(0)}, q_{i(1)}, q_{i(2)}, \ldots\ra\) is undefined.) For this reason, we can
speak of the \emph{computable continuous function} \(f : \B R \to \B R\).

It is clear that the basic notion is not changed if the rational approximations are made with dyadic rationals.  The computability assumption of $f$ can be strengthened.  For example one can ask that $f$ be primitive recursive.

\subsubsection{Modulus of continuity and approximation}
\label{app:LB:CRF2}

For the proof  {Theorem}~\ref{thm:main}, we describe the notion of 
``computable real function" in an equivalent but slightly different way. We take them to have two parts:
    \begin{enumerate}
        \item \textbf{Modulus of continuity:} The modulus of continuity is a
            recursive map \(d : \B N \to \B N\) which calculates how much
            accuracy in the input is sufficient to get a desired accuracy in the
            output. In other words, an input accurate to \(2^{-d(n)}\) ensures that 
            the output can be specified to within
            \(2^{-n}\). 
        \item \textbf{Approximation:} This is a recursive function \(g\) taking
            in a \(d(n)\)-digit binary sequence \(s\) representing some \(q \in \B
            Q\) and producing an \(n\)-digit binary output \(g(s)\) which
            ``approximates'' the value of \(f\) at \(s\). 
    \end{enumerate}

\bigskip

\noindent  We illustrate this idea with the
example of \(f(x) = \exp(x)\), \(x \in [0, 1]\).

\paragraph{Modulus of continuity} To define the modulus of continuity, as in the
main text, we can avail ourselves of a Lipschitz constant. We illustrate the way
we use the terms \emph{modulus of continuity} and \emph{approximate} in the main
text with the example $f(x)=e^x$. Since \(\tfrac d{dx} \exp(x) = \exp(x)\), and
\(\max_{x \in [0, 1]} \exp(x) < 3\), it follows that
    \begin{equation*}
        |f(x) - f(y)| < 3{|x-y|}. 
    \end{equation*}
Therefore, an input accurate to \(2^{-(n+2)}\) can be used to generate an output
accurate to \(2^{-n}\) places. Consequently, we set \(d(n) = n + 2\).

\paragraph{Approximation} At the same time, we can generate good binary
approximations of \(\exp(x)\) in a primitive recursive way. Let \(s\) be an \(n\)-digit binary sequence representing the dyadic rational
\(k \cdot 2^{-n}\) for \(0 \leq k \leq 2^n\). We simply set
    \begin{equation*}
        g(s) = \left[
            \sum_{j = 0}^{n + 3} \frac{(k \cdot 2^{- n})^j}{j!}
        \right]_{n}
    \end{equation*}
where \([x]_n\) denotes rounding \(x\) to its nearest length \(n\) binary
approximation. By Taylor's Theorem, the \(m\)-th Taylor polynomial approximates
\(\exp(x)\) on \([0,1]\) to within \(6 / m!\). Moreover, \(|x - [x]_n | \leq 2
^{-(n + 1)}\). Therefore, it follows that the dyadic rational represented by
\(g(s)\) is within \(2^{-n}\) of \(\exp(k \cdot 2^{-n})\).

Putting together the modulus of continuity and approximation is sufficient to
recover the computable real function \(\exp(x)\) exactly. 

This process is equivalent to the process outlined in section \ref{app:LB:CRF1}. Suppose the real
number \(x\) is given by the rapidly converging sequence of rationals \(\la
q_{i(n)} \ra\). Then, to calculate the \(n\)-th rational in the representation
of \(\exp(x)\), simply:
    \begin{enumerate}
        \item Set \(m = d(n)\),
        \item Extract the binary sequence \([q_{i(m)}]_m\),
        \item Calculate \(f(s)\),
        \item Output \(y = k \cdot 2^{- m}\), where \(k \cdot 2^{-m}\) is the
            dyadic rational represented by the binary sequence \(f(s)\).
    \end{enumerate}
Note that \(y\) is a dyadic rational number. Since \(q_{i(m)}\) approximates
\(\exp(x)\) to within \(2^{- (n + 2)}\), \([q_{i(m)}]_{n + 2}\) approximates
\(q_{i(m)}\) to within \(2^{- (n + 2)}\), and the approximation \(f\) is
accurate on \(n + 2\)-digit binary sequences to within \(2^{-(n + 2)}\), it
follows that \(|y - \exp(x)| < 2^{-n}\), as desired.

\paragraph{Computable $C^k$ functions} 
Fix a  $k\in \nn$.  A function $T:\bt^2\to \bt^2$ is $C^k$ if and only for all $0\le i\le k'\le k, j\in \{0,1\}$, the partial derivatives 
${\partial^{k'}\over \partial^ix\partial^{k'-i}y}(T_j)$ is uniformly continuous.  For each $k', j, i$ we can apply the notions of \emph{modulus of continuity} 
and \emph{approximation} to ${\partial^{k'}\over \partial^ix\partial^{k'-i}y}(T_j)$  individually and ask the ensemble of partial derivatives as 
$0\le i\le k'\le k, j\in \{0,1\}$ is computable. This gives the definition of a computable $C^k$ function.  
Similarly for $k=\infty$ we ask that for all $0\le i\le k'\in\nn, j\in \{0,1\}$, ${\partial^{k'}\over \partial^ix\partial^{k'-i}y}(T_j)$ is uniformly continuous and 
that there is a single algorithm that computes the modulus of continuity and approximation uniformly in $k'$.

This generalizes easily to arbitrary smooth manifolds, adding complexity without content.

\section{Ergodic Theory Background}
\label{app:ET}

In this appendix, we briefly explain and define important notions in ergodic
theory relevant to Sections~\ref{sec:Odom} to~\ref{sec:CT}.

\subsection{Why $\poZ$? Why $\bt^2$? Why $C^\infty$?}

The short answer is that we want to work in  the simplest, best behaved and
most classical context.

Physical systems are often modeled by ordinary differential equations on a
smooth compact manifold $M$.  Solutions are formalized as dynamical systems:
\[\phi:\mathbb R\times M\to M\] such that $\phi(s,\phi(t,x_0))=\phi(s+t,x_0)$
and $\phi(s,\cdot):M\to M$ is measure preserving.  

Doing repeated experiments in a physical realization of such a system---say to
measure a constant of interest such as the average value of an $L^1$ function on
$M$---is viewed as measuring $\phi(t_0,x_0), \phi(t_0+t_0, x_0), \dots
\phi((N-1)t_0, x_0)$ and averaging:  ${\tfrac 1 N}\sum_i f(\phi(i*t_0, x_0))$.
Provided that the system is sufficiently mixing (``ergodic''), the Ergodic Theorem implies that  for almost every
$x_0$ the averages along trajectories converge to the integral of $f$ over $N$. 

Thus empirical experiments are construed as sampling along portions of a
$\poZ$-action given by: \[\psi(n,x_0)=\phi(nt_0,x_0).\]

The manifold is required to be compact to avoid wild behavior and asked to be of
the smallest possible dimension.  Dimension one is impossible because there are
very few conjugacy classes of measure preserving diffeomorphisms on one dimensional manifolds.  On
the unit circle there are exactly two.

Thus we move to two dimensional compact
manifolds.  The most convenient choice is $\bt^2$, the two torus.

As $k$ increases, the behavior of $C^k$ diffeomorphisms becomes more regular--the behavior of $C^1$-diffeomorphisms 
can be quite wild.  Thus the theorem involves   $C^\infty$-diffeomorphisms because it illustrates that the basic issue is not how wild the diffeomorphism is. 

It could be argued that the tamest situation of all involves real analytic transformations of the 2-torus.  The results in this paper  can be extended to real-analytic maps using the work of Banerjee and Kunde \cite{BK}.
\paragraph{In Summary} We are proving that the question of forward vs.\ backward
time encodes some of the most complex problems in mathematics.  This claim is
made stronger by taking the simples possible context: time is given by  a
$\poZ$-action, $\bt^2$ is the simplest, most concrete manifold possible, and the
diffeomorphisms in question are the most regular possible.

\subsection{Basics}\label{basicerg}

Unless otherwise stated, all definitions are modulo a null set. That is, \(A =
B\) is shorthand for \(\mu[A\diff B] = 0\), \(f = g\) is shorthand for
\(\mu[\{x:f(x)\neq g(x)\}] = 0\), and so on. Put differently, all statements in
this appendix contain a tacit ``a.e.''

\begin{definition}[Measurable Dynamical Systems]
    A \emph{measurable measure preserving system}
     is a quadruple \((X, \C F, \mu, T)\)
    where:
        \begin{itemize}
            \item \(X\) is a set,
            \item \(\C F\) is a \(\sigma\)-algebra of subsets of \(X\),
            \item \(\mu\) is a measure on \(\C F\) such that \(\mu[X] = 1\), and
            \item \(T: X\to X\) is an invertible, \(\C F\)-measurable, \(\mu\)-preserving
                transformation on \(X\), i.e., \(\mu[A] = \mu[T^{-1}(A)]\).
        \end{itemize}
    An \emph{ergodic} system (or, more specificially, an \emph{ergodic
    transformation}) is a measurable dynamical system in which \(T\) satisfies
    the additional hypothesis that if \(T^{-1}(A) = A\), then \(A = X\) or \(A =
    \emptyset\) a.e.
\end{definition}

One way of understanding this definition is that ergodic systems are the atomic
units of measurable dynamical systems; measurable dynamical systems can be
broken down into ergodic pieces, but no further. This intuition is formalized by
the ergodic decomoposition theorem; see Theorem 3.22 in \cite{glasbook}.

Measurable dynamical systems can be studied from several perspectives. In
addition to their straightforward pointwise actions, they can also be viewed
from the perspective of functional analysis.

\begin{definition}[Koopman Operator]
    Let \(T:X\to X\) be a measure-preserving transformation. Then \(T\) induces
    an operator \(U_T: L^p(X)\to L^p(X)\) in the following way:
        \begin{equation*}
            f\mapsto f\circ T.
        \end{equation*}
    The operator \(U_T\) is known as the \emph{Koopman operator}.
\end{definition}

The following theorem is a fundamental tool in ergodic theory.
\begin{theorem}[Pointwise Ergodic Theorem]
\label{app:ET:thm:ET}
    Let \((X, \C F, \mu, T)\) be a measure-preserving dynamical system. Then,
    for any \(f\in L^{1}(X)\) and a.e.\ \(x\in X\),
        \begin{equation*}
            \lim_{n\to\infty}\frac{1}{n}\sum_{k = 0}^{n-1} f(T^kx),
        \end{equation*}
    converges a.e.\ to an \(L^1(X)\) function \(\bar f\), which is invariant
    under the action of \(T\) and \(\int_X \bar f d\mu = \int_X f d\mu\).
\end{theorem}

While Theorem~\ref{app:ET:thm:ET} is stated for general measurable dynamical
systems, note that if \(T\) is ergodic, the only invariant functions are
constant a.e.\ on \(X\). Therefore, in the ergodic case, for a.e.\ \(x\in X\),
    \begin{equation*}
            \lim_{n\to\infty}\frac{1}{n}\sum_{k = 0}^{n-1} f(T^kx) = \int_X f
            d\mu
    \end{equation*}

\subsection{Symbolic systems}\label{symbsys}

Symbolic systems, some of which are examined in detail in Section~\ref{sec:Odom}
and Subsection~\ref{sec:CT:ss:circ}, represent a large and well-studied class of
measurable dynamical systems. What follows in this subsection collects and
further elaborates on the background given in Section~\ref{sec:Odom}.

\begin{definition}[Symbolic System]
\label{app:ET:defn:ss}
    Let \(\Sigma\) be some finite or countable alphabet. Then \(\Sigma^{\B Z}\)
    is the collection of all bi-infinite words in \(\Sigma\), i.e., functions
    \(s:\B Z\to\Sigma\). Let \(\sh\) represent the shift operator, i.e., for all
    \(n\in\B Z\),
        \begin{equation*}
            \sh(s)(n) = s(n + 1).
        \end{equation*}
    Let \(\C B\) be a shift-invariant \(\sigma\)-algebra of subsets of \(\Sigma^{\B Z}\).  Lastly, let \(\mu\) be a
    shift-invariant probability measure measure on \(\C B\). Then we call the
    quadruple \((\Sigma^\B Z, \C B, \mu, \sh)\) a \emph{symbolic system}.
\end{definition}

A common way of constructing symbolic systems is by combining finite words into
infinite words according to some specified pattern. For finite words, we make
use of the following notation and conventions:
    \begin{itemize}
        \item The length of a word \(u\) is \(|u|\);
        \item The concatenation of the words \(u\) and \(w\) is written \(u\cdot
            w\), or, where clear from context, \(uw\);
        \item The repeated concatenation of a word with itself is denoted as
            follows:
            \begin{equation*}
                w^n = \overbrace{w\cdot w\cdots w}^{\text{\(n\)-times}};
            \end{equation*}
            and,
        \item All finite words are assumed to be zero-indexed, i.e., if
            \(|w|=n\), then \(w(0)\) is the first letter in \(w\) and \(w(n)\)
            is undefined.
        \item For a finite or infinite word \(s\), \(s\rest[n,m)\) for \(n, m
            \in \B Z\) is the finite subword beginning at \(n\) and having
            length \(m - n\).
    \end{itemize}

\begin{definition}[Unique Readability]\label{unique}
    A collection of words \(\C W\) is \emph{uniquely readable} when for any \(u,
    v, w\in\C W\), if \(u\cdot v = p\cdot w \cdot p'\), then either \(p\) or
    \(p'\) is the empty word.
\end{definition}

\begin{definition}[Construction Sequence]\label{consteq}
    Let \(\Sigma\) be a finite or countable alphabet. We call a squence of sets
    of finite words \(\la\C W_n: n\in\B N\ra\) in the alphabet \(\Sigma\) a
    \emph{construction sequence} if:
        \begin{enumerate}
            \item All the words in \(\C W_n\) have the same length \(q_n\);
            \item Each \(w\in\C W_{n+1}\) contains each \(v\in\C W_n\) as a
                subword;
            \item For all \(n > 0\), if \(w\in\C W_{n+1}\), then there is a
                unique parsing of \(w\) into segments in terms of \(w_0, \ldots,
                w_l\in\C W_n\), \(u_0, \ldots, u_{l + 1}\notin\C W\), such that
                    \begin{equation}
                    \label{app:ET:defn:cs-eq}
                        w = u_0\cdot w_0\cdot u_1 \cdots w_l \cdot u_{l+1}.
                    \end{equation}
            \item There is a summable sequence \(\la \epsilon_n: n\in\B N\ra\)
                such that if \(u_0,\ldots, u_{l+1}\) are as in Equation
                \ref{app:ET:defn:cs-eq}, then
                    \begin{equation}
                    \label{app:ET:defn:eq2}
                        \frac{\sum_{i = 0}^{l + 1} |u_i|} {q_{n+1}} =
                        \epsilon_{n+1}.
                    \end{equation}
        \end{enumerate}
    In Equation~\ref{app:ET:defn:cs-eq}, we call the words \(u_i\)
    \emph{spacers}.

    A construction sequence gives rise to a symbolic dynamical system in a
    natural way. Define the set \(\B K\) to be
        \begin{equation*}
            \{s\in\Sigma^{\B Z}: (\forall n,m)(\exists k, N) s\rest[n,m) =
            w\rest[n + k, m + k)\text{ for some } w\in\C W_N\}.
        \end{equation*}
    That is, \(\B K\) consists of all bi-infinite words with the property that
    every finite subword is itself a subword of some word \(w\in\C W_N\) for
    some \(N\in\B N\).

    Let \(\C B\) be the Borel sets and let \(\nu\) be some \(\sh\)-invariant
    measure on \(\C B\) such that \(\nu[S]=1\). Then the system \((S, \C B, \nu,
    \sh)\) is the \emph{symbolic system corresponding to the construction
    sequence \(\la\C W_n:n\in\B N\ra\)}.
\end{definition}

\begin{remark}
While unique readability and construction sequences are fundamental for the systems constructed in this paper they are not a  property of typical symbolic systems.
\end{remark}
\noindent Suppose that $v\in \mcw_n$ and $w\in \mcw_{n+1}$.  Let $r(v,w)$ be the number of occurrences of $v$ in $w$.
\begin{definition}[Uniform Construction Sequence]
    We say that a construction sequence is \emph{uniform} if there is a summable
    sequence \(\la \epsilon_n: n \in\B N\ra\) and a sequence \(\la d_n: n\in\B
    N\ra\) in \((0, 1)\) of densities such that for all \(w\in\C W_{n+1}\) and
    \(v\in\C W_n\),
        \begin{equation*}
            \left|\frac{r(v,w)}{q_{n + 1} / q_n} - d_n \right| < \epsilon_n.
        \end{equation*}
    That is, each \(v\in\C W_n\) occurs in each \(w\in\C W_{n + 1}\) with
    approximately the same frequency as all of the others.

    If \(r(v,w)\) is actually constant over \(v\in\C W_n\) and \(w\in \C
    W_{n+1}\), then we say that \(\la\C W_n\ra\) is \emph{strongly uniform}.
\end{definition}

\begin{definition}[Cylinder Sets]
    Let \(w\) be a finite word in \(\Sigma\). Then the \emph{cylinder set}
    associated to \(w\), also written \(\la w\ra\), is given by
        \begin{equation*}
            \la w\ra\ = \{s\in\Sigma^\B Z: s\rest[0, |s|) = w\}.
        \end{equation*}
\end{definition}

These notions are useful largely for the following lemma which is proved in \cite{part1}, which allows us to
further restrict the set \(\B K\) to especially well-behaved bi-infinite words
\(s\).  
\begin{lemma}
\label{app:ET:lem:ue}
    Let \((\Sigma^\B Z, \C B, \nu, \sh)\) be a symbolic system arising from the
    construction sequence \(\la\C W_n\ra\). Then:
        \begin{enumerate}
            \item The set \(\B K\) is the smallest closed shift-invariant subset
                of \(\Sigma^\B Z\) with non-empty intersection with every basic
                open set \(\la w\ra\) for \(w\in\C W_n\) for some \(n\).
            \item Suppose \(\la\C W_n\ra\) is
                 uniform. Then, let \(S\) be the collection of
                \(s\in\Sigma^\B Z\) such that there are increasing integer
                sequences \(\la a_n: n \in\B N\ra\) and \(\la b_n: n\in\B N\ra\)
                such that \(a_n, b_n\to\infty\) as \(n\to\infty\) and
                \(s\rest[-a_n, b_n)\in\C W_n\). Then there is a unique
                non-atomic shift-invariant measure \(\nu\) concentrating on
                \(S\), and this \(\nu\) is ergodic. 
        \end{enumerate}
\end{lemma}

In Lemma~\ref{app:ET:lem:ue}, the fact that \(\nu\) exists and
is unique is proved in Lemma 11 of \cite{part1}. For strongly uniform systems it is a 
direct consequence of the ergodic theorem that for a generic $x$
    \begin{equation*}
        \nu(\la v\ra) = \lim_{n\to\infty}\frac {|\{\text{Occurences of \(w\) in
        \(x\rest[-a_n, b_n)\)}\}|}{n}
    \end{equation*}
which, because \(\la\C W_n\ra\) is strongly uniform, is fixed independently of $x$.

\subsection{Odometers}\label{append:odo}

Odometers are an important class of transformations. They give a symbolic realization  
of the Kronecker factors of the odometer-based
transformations built  in Section~\ref{sec:Odom} using the Substitution Lemma.

\begin{definition}[Odometers]
    Let \(\la k_n: n \in\B N\ra\), \(k_n > 1\) be a sequence of natural numbers.
    Then the group
        \begin{equation*}
            O=\prod_{i=0}^\infty\B Z/k_n\B Z
        \end{equation*}
    is the \(\la k_n\ra\)-adic integers. By giving each factor of the form \(\B
    Z/k_n\B Z\) a uniform measure, we can form the product measure \(\mu\).
    (This is also the Haar measure on \(O\).)
    \smallskip

    Define the transformation \(\C O\) which acts by \(\la k_n\ra\)-adic addition
    by one. That is, the transformation is given by addition by one with a carry
    to the right:
        \begin{equation*}
            T: \la x_0, x_1, x_2, \ldots \ra \mapsto \la x_0 + 1, x_1, x_2,
            \ldots\ra,
        \end{equation*}
    unless \(x_0 = k_0 - 1\), in which case
        \begin{equation*}
            T: \la x_0, x_1, x_2, \ldots \ra \mapsto \la 0, x_1 + 1, x_2,
            \ldots\ra,
        \end{equation*}
    unless both \(x_0 = k_0 - 1\) and  \(x_1 = k_1 - 1\), in which case 
        \begin{equation*}
            T: \la x_0, x_1, x_2, \ldots \ra \mapsto \la 0, 0, x_2 + 1,
            \ldots\ra,
        \end{equation*}
    and so on.
\smallskip

\noindent    A measure preserving system given in this way is called an \emph{odometer}.
\end{definition}

A useful property of the odometers is that the eigenvalues of their associated
Koopman operators are easily characterized.

\begin{theorem}
\label{app:ET:thm:odom-eigen}
    Let \(q_n = k_0\cdot k_1 \cdots k_{n-1}\), and let \(A_n\subseteq O\) be the set of
    points whose first \(n \) coordinates are zero. Define
        \begin{equation*}
            \C R_n = \sum_{j = 0}^{q_n - 1}(e^{2\pi i / q_n})^j\cdot\chi_{\C
            O^j(A_n)},
        \end{equation*}
where we denote the result of applying the transformation $\mco$ $j-$times to $A_n$ by $\mco^j(A_n)$.
    Then:
        \begin{itemize}
            \item The function \(\C R_n\) is an eigenvector of \(U_{\C O}\) with
                eigenvalue \(e^{2\pi i / q_n}\),
            \item The function \(\C R_{n + 1}^{q_{n + 1}}\) is equal to \(\C
                R_n\), and
            \item The collection \(\{\C R_n^k: 0\leq k \leq q_n, n\in\B N\}\)
                form a basis of \(L^2(O)\).
        \end{itemize}
\end{theorem}

Odometer systems are ergodic and canonically isomorphic to their inverses via
the map \(x\mapsto -x\).

We have the following immediate corollary, which gives a simple sufficient
condition for two parameter sequences \(\la k_n^0\ra\) and \(\la k_n^1\ra\)
define distinct (i.e., non-isomorphic) odometers.

\begin{cor}
\label{cor:odoms}
    Let \(\C O_0\) and \(\C O_1\) be odometers with associated parameter
    sequences \(\la k_n^0\ra_n\) and \(\la k_n^1\ra_n\). Let
        \begin{equation*}
            \F p_0 = \{p \in \B N: \text{\(p\) is prime and for some \(m \in \B
            N\), } p\mid k_m^0\},
        \end{equation*}
    and
        \begin{equation*}
            \F p_1 = \{p \in \B N: \text{\(p\) is prime and for some \(m \in \B
            N\), } p\mid k_m^1\}.
        \end{equation*}
    Then, if \(\F p_0\neq \F p_1\), \(\C O_0\not\cong \C O_1\).
\end{cor}

\begin{proof}
    The Koopman operators \(U_{\C O_0}\) and \(U_{\C O_1}\) are invariants of
    the respective odometers. By Theorem~\ref{app:ET:thm:odom-eigen}, the
    eigenvalues of these operators are roots of unity, and the prime divisors of
    these orders are \(\F p_0\) and \(\F p_1\), respectively.
\end{proof}

\subsection{Factors, joinings, and conjugacies}\label{factors and joinings}

Like most objects, measurable dynamical systems interact with one another via
certain kinds of morphisms. These morphisms must respect the action of the
transformations, but only need be defined a.e.

\begin{definition}[Factors]
    Let \((X, \C B(X), \mu, T)\) and \((Y, \C B(X), \nu, S)\) be ergodic
    systems. A map \(\phi:X\to Y\) is a \emph{factor} if and only if
    \(\phi^*\mu\), the pushforward measure, equals \(\nu\), and the following
    diagram commutes for almost all \(x\in X\): 
        \begin{equation*}
            \begin{tikzcd}
                X \rar["T"]\dar["\phi"]
                    & X \dar["\phi"]\\
                Y\rar["S"]
                    & Y
            \end{tikzcd}
        \end{equation*}
    A map \(\phi\) meeting these requirements is a \emph{factor map} and \((Y,
    \C B(X), \nu, S)\) is said to be a \emph{factor} of \((X, \C B(X), \mu,
    T)\). (More briefly, \(S\) is a factor of \(T\).)
\end{definition}

\begin{definition}[Isomorphism]
    Let \({\mathbb X}=(X, \C B(X), \mu, T)\) and \({\mathbb Y}=(Y, \C B(X), \nu,
    S)\) be ergodic systems.  Then $\mathbb X$ is \emph{measure isomorphic} to $\mathbb Y$ if
    and only if there is an invertible factor map $\phi:X\to Y$. We write
    $S\cong T$.
    
    Informally we say ``$S$ is isomorphic" or ``$S$ is congruent" to $T$. It is
    easy to verify that if $X=Y$, then $S\cong T$ if and only if $S$ and $T$ are
    conjugate in the group of measure preserving transformations of $X$. For
    this reason we  use \emph{congruent} as a synonym of \emph{isomorphic}.
\end{definition}

We now consider  generalizations of factor maps, \emph{joinings.}
\begin{definition}[Joining]
    A joining of two ergodic systems \(\B X =(X, \C B(X), \mu, T)\) and \(\B Y =
    (Y,\C B(Y), \nu, S)\) is a \(T\times S\)-invariant measure \(\eta\) on
    \((X\times Y,\C B(\B X)\otimes \C B(\B Y))\) such that:
        \begin{itemize}
            \item \(\eta(A\times Y)=\mu(A)\), and
            \item \(\eta(X\times B)=\mu(B)\).
        \end{itemize}
\end{definition}

Let $\eta$ be a joining of $\B X$ and $\B Y$. If $\B X'$, $\B Y'$ are factors of
$\B X$,$\B Y$, respectively, then the measure algebras associated with $\B X'$
and $\B Y'$ can be viewed as subalgebras of the measure algebras of $\mathbb X$
and $\mathbb Y$.  Hence $\eta$ induces an invariant measure on $\mathbb X'
\otimes \mathbb Y'$ and hence a joining $\eta'$ of $\mathbb X'$ and $\mathbb
Y'$.

\begin{definition}[Graph Joinings]
    If $\pi:X\to Y$ is a factor map, then viewed as a subset of $X\times Y$, $\pi$ is 
    $T\times S$-invariant and can be canonically identified with a joining $\mcj$ by setting:
    \[\mcj(A)=\mu(\{x:(x,\pi(x))\in A\}).\]
   Joinings coming 
    from  factor maps are called \emph{graph joinings}. An
    \emph{invertible graph joining} is a graph joining coming from an isomorphism.
\end{definition}

Joinings arise in other forms as well. First, we need the standard notion of a
disintegration:

\begin{definition}[Disintegration]
    Let \(\phi:\B X\to\B Y\) be a factor map, and, for \(y\in Y\), $A\subseteq X$,  let \(A^y\)
    be the fiber of \(A\) over \(y\). Then there exists a family of measures
    \(\{\nu_y:y\in Y\}\) concentrating on  respective fibers \(X^y\), such
    that:
        \begin{enumerate}
            \item Each \(\nu_y\) is a standard probability measure on \(X^y\);
            \item For \(A\in\C B(X)\), \(r\in[0,1]\), and \(\epsilon>0\), the
                set \[\{y\in Y:A^y\text{ is \(\nu_y\)-measurable and
                \(|\nu_y(A^y)-r|<\epsilon\)}\}\] is \(\nu\)-measurable; and
            \item For all \(A\in\C B(X)\), \(\mu(A)=\int_Y\nu_y(A^y)d\nu\).
        \end{enumerate}
   The family  \(\{\nu_y:y\in Y\}\) is the \emph{disintegration} of \(\nu\) over \(\phi\).
\end{definition}

\begin{definition}[Relatively Independent Joinings]
    Let \(\B X=(X,\mcb,\mu, T)\), \(\B Y=(Y,\mcc, \nu, S)\), and \(\B Z=(Z,\mathcal D, \rho, R)\) be ergodic systems such that
        \begin{equation*}
            \begin{tikzcd}
                \B X\ar[swap, ddr,"\phi_1"] && \B Y\ar[ddl,
                "\phi_2"]\\
                    \\
                & \B Z &\\
            \end{tikzcd}
        \end{equation*}
    and $\{\mu_z\}, \{\nu_z\}$ be the respective disintegrations.
    Then define \(\eta\), the \emph{relatively independent joining} of \(\B X\)
    and \(\B Y\) over \(\B Z\) by
        \begin{equation*}
            \eta(A\times B)=\int_Z\mu_z(A^z)\nu_z(B^z)d\rho(z)
        \end{equation*}
    which concentrates on \(\{(x,y):\phi_1(x)=\phi_2(y)\}\).
\end{definition}

\section{Diffeomorphisms}
\label{diffeos}

This appendix gives a very brief description of the role of diffeomorphisms play in this paper, as opposed to abstract measure-preserving transformations.

\subsection{Diffeomorphisms of the torus}

\begin{definition}[The two-torus]
    We define the two-torus \(\B T^2\) to be the product of the
    unit interval with itself with opposite edges identified: define an equivalence relation on $[0,1]\times [0,1]$ by setting $(0,y)\sim (1,y)$ for all $y$ and $(x,0)\sim (x,1)$ for all $x$. We view \(\B T^2\) as the unit square modulo the equivalence relation, i.e.,
        \begin{equation*}
            \B T^2 = [0, 1] \times [0, 1] /\sim.
        \end{equation*}
 \end{definition}
For this reason, to ensure that a continous map  \(T:[0, 1]\times [0,1]\to [0,1]\times [0,1]\) induces a
diffeomorphism of the \emph{torus}, it suffices to ensure that
	\begin{enumerate}
	\item for all  $x, T(x,0)=T(x,1)$ and for $y, T(0,y)=T(1,y)$,
	\item  there is an $\mathbb R^2$ neighborhood
	$U$ of $[0,1]\times [0,1]$ such that $T$ can be extended smoothly to a diffeomorphism 
	$T^*:U\to U$ such that  
	$T^*(x,y)=T([x]_1,[y]_1)$.	
	\end{enumerate}

The set of \(C^\infty\)-maps on a compact manifold have a natural topology in
which two maps are close if the transformations themselves are close, their
differentials are close, and so on. More precisely, if \(M\) is a \(C^k\)-smooth
compact finite-dimensional manifold and \(\mu\) is a standard measure on \(M\)
determined by a smooth volume element, then, for each \(k\) there is a Polish
topology on the \(k\)-times differentiable homeomorphisms of \(M\). The
\(C^\infty\)-topology is the coarsest common refinement of these topologies.

This topology is induced by a very explicit metric.

\begin{definition}[The \(d^\infty\) metric]
    Let \(S\) and \(T\) be smooth maps from \(\B T^2\) to \(\B T^2\). Then, set
    the distance between \(S\) and \(T\), \(d^\infty(S, T)\) as follows:
        \begin{equation}
        \label{eq:d-infty-def}
            d^\infty(S, T) = \sum_{k = 0}^\infty 2^{-k} \cdot \frac{\max_{x \in
            \B T^2} ||D^kS(x) - D^kT(x)||}{1 + \max_{x \in \B T^2} ||D^kS(x) -
            D^kT(x)||}.
        \end{equation}
\end{definition}

Here \(D^k f\) denotes the \(k\)-th differential of \(f\). Note that this
measure is \emph{effective} in the sense that to determine the distance between
\(S\) and \(T\) to an accuracy of \(2^{-k}\), it is only necessary to examine
the first \(k\) differentials of \(S\) and \(T\).

The two-torus carries a natural Lebesgue measure and the diffeomorphisms we
consider will preserve the Lebesgue measure. However the conjugacy relation
between diffeomorphisms does not require smoothness: two diffeomorphisms $S, T$
are conjugate if and only if there is a (not necessarily smooth) invertible
measure preserving transformation $\theta$ such that:
    \[\theta\circ T=S\circ\theta.\] 

\subsection{Smooth permutations}
\label{smoothPerms}

In our construction of diffeomorphisms of \(\B T^2\), it was required to find smooth approximations to arbitrary permutations of a rectangular partition of a larger rectangle. This is discussed in Theorem \ref{smp}. This is done by composing $C^\infty$ maps that swap most of the interior of adjacent squares.  These maps emulate the transpositions used to construct an arbitrary permutation. (See \cite{part1}). For completeness we exhibit concrete proofs to illustrate that the swap-diffeomorphisms can be taken to be primitive recursive.

As throughout the paper we say \emph{smooth} to mean $C^\infty$.  When the domain has boundary we mean $C^\infty$ in the interior and continuous at the boundary. 
One ingredient is the following example:

\begin{ex}(Bump function)\label{lem:ssf}
Define 
\begin{equation*}
            g(x) = \begin{cases}
                  e^{-{1\over x}},      & x>0\\
                       
                    0, &  x\le 0
                \end{cases}
\end{equation*}
and set
\[f(x)={g(x)\over g(x)+g(1-x)}.\]
One verifies directly that $f\in C^\infty$, $f\rest(-\infty,0]=0$ and $f\rest[1,\infty)=1$ and $f$ maps the unit interval to the unit interval.

By rescaling and translating the input to $f$ and possibly replacing it by $1-f$,  for all 
$\alpha<\beta\in \mathbb R, a\in \{0,1\}$ one can create total  $C^\infty$ functions that map 
$[\alpha,\beta]$ to $[0,1]$, are constantly 
$a$ on $(-\infty,\alpha]$ and $1-a$ on $[\beta,\infty)$.

\end{ex}

Here is an explicit version of Lemma 36 in \cite{part1}. The author is grateful to A. Gorodetski for providing the complex analysis background for this argument.  A reference for the results cited is 
\cite{SCM}. A guide to a Matlab software package for computing the relevant mappings claimed is available at 
\url{http://www.math.udel.edu/~driscoll/SC/guide-v5.pdf}.
 \begin{lemma} \label{squareswaps}
 Let $A=[0,1]\times [-1,1]$ and $\epsilon>0$. Then there is an area preserving map $f:A\to A$ that is $C^\infty$ on the interior of $A$, is the identity on a neighborhood of the boundary of $A$ and there is an open set $\mco\subseteq A$ such that
 \begin{enumerate}
 \item $\mco$ is symmetric about the $x$ axis,
 \item  $\lambda(\mco)> 2-\epsilon$,
 \item {For each $x$ on the boundary of $\mco$ the distance from $x$ to the boundary of $A$ is less than 
 $\epsilon$.}
 \item $f$ maps the top half of $\mco$ to the bottom half of $\mco$.
 \end{enumerate}
 \end{lemma}

 \pf Let $D$  be the closed disk of area 2 centered at the origin viewed as a subset of $\mathbb C$.  By the 
 Schwarz-Christoffel theorem there is  an analytic bijection $\theta_0$ from the upper  half disk,   
 $\overline{D\cap\{\text{Im}z>0\}}$,   to $[0,1]\times [0,1]$ that is analytic on the interior and continuous on the boundary. {Moreover it sends the  intersection of the disk with the $x$-axis to $[0,1]\times \{0\}$.}

 By the Schwartz Reflection Principle, if $\theta_0$ is extended symmetrically about the $x$-axis the result 
 (which we also call $\theta_0$) is still analytic. (So the map $\theta_0$ satisfies the identity 
 $\theta_0(\bar{z})=\overline{\theta_0(z)}$.)
 
 \paragraph{Note:} For uncomplicated regions such as the half-disk and a square, there are explicit integral formulas for the Schwarz-Christoffel mapping.  These are given in \cite{SCM}.
\medskip

 Using a slight variation on Moser's `Lemma 2' in  \cite{moser}, this map can be composed with another map so that the result $\theta$ is measure preserving, $C^\infty$ and still symmetric about the horizontal line.  
 
Since the  proof of Moser's lemma in \cite{moser} is a little bit confusing, the following is a painfully explicit proof of a slight variation of  a very special case.
 
 \begin{lemma}(Moser prime) Suppose that $\theta_0:D\to A$ is a bijection that is analytic on the interior of D.
  Then there is a bijection $u:A\to A$ that is $C^\infty$ in the interior of $A$
   and preserves the boundary of $A$ 
   such that $\theta=u\circ \theta_0$ is Lebesgue measure preserving and takes the top half of $A$ to the top half of $A$. 
 \end{lemma}
\pf  Let $\mu_0$ be the measure on $A$ given by 
 \[\mu_0(A)=\lambda^2(\theta_0^{-1}(A)).\]
 Then $\mu$ is absolutely continuous with respect to Lebesgue measure on $A$. Let $h$ be density associated with  
  $\mu_0$. Then $h$ is continuous and  analytic on the interior of $A$.  We first show that it suffices to find a $u:A\to A$ such that the Jacobian of $u$ is $h$. Let $\theta=u\circ \theta_0$ and $X\subseteq A$. Then:
 \begin{align*}
 \lambda^2(X)&=\int_X1\ d\lambda^2\\
 &=\int_{u^{-1}(X)}(1\circ u)det(D(u))d\lambda\\
 &=\int_{u^{-1}(X)}h\ d\lambda\\
 &=\lambda^2((\theta_0\circ u)^{-1}(X))\\
 &=\lambda^2(\theta^{-1}(X)).
 \end{align*}
 
Let  $u=(u_1,u_2)$ and suppose that $u_1$ is a function of $x_1$ and $u_2$ is a function of $(x_1, x_2)$. Then   $det(D(u))=\frac{du_1}{dx_1}\frac{du_2}{dx_2}$. We rewrite $h=h_1(x_1)h_2(x_1, x_2)$ and find functions $u_1, u_2$ so that 
\begin{align}
\frac{du_1}{dx_1}&=h_1(x_1)\label{de1}\\
\frac{du_2}{dx_2}&=h_2(x_1, x_2).\label{de2}
\end{align} 
Set
\begin{align*}
h_1(x_1)=\int_{-1}^1h(x_1,t)dt\\
h_2(x_1,x_2)={h(x_1,x_2)\over h(x_1)}
\end{align*} 
 Then equations \ref{de1} and \ref{de2} give two equations that can be solved effectively by ordinary integration 
 and yield smooth solutions.
 \begin{align*}
 u_1(x_1)&=\int_0^{x_1}h_1(x_1)\\
 &=\int_0^{x_1}\int_{-1}^1h(x_1,t)dt
 \end{align*}
and
 \begin{align*}
 u_2(x_1,x_2)&=\int_0^{x_2}h_2(x_1,x_2)dx_2\\
 &={\int_0^{x_2}h(x_1,x_2)dx_2\over \int_{-1}^{1}h(x_1,t)dt}.
 \end{align*}
  Clearly $u_1(0)=0$ and $u_1(1)=1$ and $u_2(x_1,0)=0, u_2(x_1,1)=1$, so $u$ preserves the boundary.
 
 By the symmetry of $\theta_0$, for each $x_i$, 
 \[\int_{-1}^0h(x_1, t)dt=\left(1\over 2\right)\int_{-1}^1h(x_1,t)dt.\] It follows immediately that $u_2(x_1,x_2)\ge 0$ for $x_2\ge 0$. Hence $u$ takes the top half of $A$ to the top half of $A$. (In fact it is easy to verify that $u$ is symmetric about the $x$-axis.)
 
 It is also easy to verify that $u$ is one to one.  If $x_1\ne y_1$ then $u_1(x_1)\ne u_1(y_1)$. If $u_1(x_1)=u_1(y_1)$ but $x_2\ne y_2$ then $u_(x_1,x_2)\ne u(y_1, y_2)$. Since $u$ is continuous and takes the faces of the $A$ to faces of $A$ it follows that $u$ is a surjection.\qed

We now finish the proof of Lemma \ref{squareswaps}. Let $R$ be the radius of the disk $D$ of area 2. By the uniform continuity of  $\theta$ we can choose an $\gamma>0$ so small that each point on the circle of radius 
$R-\gamma$ is sent to a point within $\epsilon$ of the boundary of $A$.  Without loss of generality, the area of the disk of radius 
$R-\gamma$ is bigger than $2-\epsilon$. As in Lemma 36 of \cite{part1} we consider a 
$C^\infty$ function 
$F:[0,R]\to [0,\pi]$ such that 
 $F\rest[0,R-\gamma]$ is identically equal to $\pi$, and $F$ is identically equal to zero in a neighborhood of 
 $R$.  Define a $C^\infty$, area preserving map $\phi:D\to D$ in polar coordinates by 
\[\phi(r,\theta)=(r,\theta+F(r))\]  
Then $\phi$ rotates the upper  half disk of radius $R-\gamma$ to the lower half disk of radius $R-\gamma$. 
Such a function is given explicitly in Lemma \ref{lem:ssf}.

It is now routine to check that the map $f=\theta\circ \phi\circ \theta^{-1}$ satisfies the conclusions of 
Lemma \ref{squareswaps}.

 \qed

\bibliography{godel}

\bibliographystyle{plain}

\end{document}